\documentclass[reqno,12pt]{amsart}

\usepackage[colorlinks=true, linkcolor=blue, citecolor=blue]{hyperref}
    
\usepackage{amssymb}
\usepackage{amsmath, graphicx, rotating}
\usepackage{color}
\usepackage{soul}
\usepackage[dvipsnames]{xcolor}
           
\usepackage{ifthen}
\usepackage{xkeyval}   
\usepackage{todonotes}  
\usepackage[T1]{fontenc}
\usepackage{lmodern}
\usepackage[english]{babel}

\usepackage{ upgreek }
\usepackage{stmaryrd}
\SetSymbolFont{stmry}{bold}{U}{stmry}{m}{n}
\usepackage{amsthm}
\usepackage{float}

\usepackage{ bbm }
\usepackage{ stmaryrd }
\usepackage{ mathrsfs }
\usepackage{ frcursive }
\usepackage{ comment }

\usepackage{pgf, tikz}
\usetikzlibrary{shapes}
\usepackage{varioref}
\usepackage{enumitem}

\usepackage{mathtools}
\usepackage{dsfont}

\definecolor{rouge}{rgb}{0.7,0.00,0.00}
\definecolor{vert}{rgb}{0.00,0.5,0.00}
\definecolor{bleu}{rgb}{0.00,0.00,0.8}

\usepackage[margin=1.1in]{geometry}

\newtheorem{theorem}{Theorem}[section]
\newtheorem*{theorem*}{Theorem}
\newtheorem{lemma}[theorem]{Lemma}

\newtheorem{corollary}[theorem]{Corollary}
\newtheorem{proposition}[theorem]{Proposition}

\labelformat{hypothesis}{\textbf{M\kern-0.1mm#1}}

\newtheorem{condition}{Condition}

\newtheorem{conditionA}{A\kern-0.1mm}
\labelformat{conditionA}{\textbf{A\kern-0.1mm#1}}

\renewcommand\dots{\hbox to 1em{.\hss.\hss.}}

\theoremstyle{definition}

\numberwithin{equation}{section}

\def\bb#1{\mathbb{#1}}

\def\scr#1{\mathscr{#1}}

\def\geq{\geqslant}
\def\leq{\leqslant}

\newcommand\ee{\varepsilon}

\DeclareMathOperator{\supp}{supp}

\begin{document}

\title[Limit theorems for products of positive random matrices]
{Edgeworth expansion and large deviations for the coefficients of products of positive random matrices}

\author{Hui Xiao$^1$}\footnotetext[1]{Corresponding author: xiaohui@amss.ac.cn}
\author{Ion Grama}
\author{Quansheng Liu}

\curraddr[Xiao, H.]{Academy of Mathematics and Systems Science, Chinese Academy of Sciences, Beijing 100190, China}
\email{xiaohui@amss.ac.cn}
\curraddr[Grama, I.]{Univ Bretagne Sud, CNRS UMR 6205 LMBA, 56017, Vannes, France}
\email{ion.grama@univ-ubs.fr}
\curraddr[Liu, Q.]{Univ Bretagne Sud, CNRS UMR 6205 LMBA, 56017, Vannes, France}
\email{quansheng.liu@univ-ubs.fr}

\begin{abstract}
Consider the matrix products $G_n: = g_n \ldots g_1$, 
where $(g_{n})_{n\geq 1}$ is a sequence of independent and identically distributed positive random $d\times d$ matrices.
Under the optimal third moment condition, we first establish a Berry-Esseen theorem and an Edgeworth expansion 
for the $(i,j)$-th entry $G_n^{i,j}$ of the matrix $G_n$, where $1 \leq i, j \leq d$. 
Utilizing the Edgeworth expansion for $G_n^{i,j}$ under the changed probability measure, 
we then prove precise upper and lower  large deviation asymptotics for the entries $G_n^{i,j}$
subject to an exponential moment assumption.
As applications, we deduce local limit theorems with large deviations for $G_n^{i,j}$
and establish upper and lower large deviations bounds for the spectral radius $\rho(G_n)$ of $G_n$.  
A byproduct of our approach is the local limit theorem for $G_n^{i,j}$ under the optimal second moment condition. 
In the proofs we develop a spectral gap theory for both the norm cocycle and the coefficients, 
which is of independent interest.
\end{abstract}

\date{\today}

\maketitle


\section{Introduction and main results}\label{sec.prelim}

\subsection{Background and objectives}\label{subsec.notations}

For any integer $d \geq 2$, 
the Euclidean space $\bb R^d$ is endowed with the canonical basis $(e_i)_{1\leq i\leq d}$. 
We equip $\bb R^d$ 
with the standard scalar product $\langle u, v \rangle = \sum_{i = 1}^d u_i v_i$ 
and the $L_1$-vector norm $\|v\| = \sum_{i = 1}^d |\langle v, e_i \rangle|$. 
Let $\mathbb{R}^{d}_+ = \{ v \in \bb R^d: \langle v, e_i \rangle \geq 0, \forall 1 \leq i \leq d \}$ be the positive quadrant of $\bb R^d$
and denote $\mathcal{S} = \{v \in \mathbb{R}^{d}_+, \|v\| = 1\}$.
Throughout the paper, the letters $c, c', c''$ represent constants, while $c_{\alpha}, c'_{\alpha}, c''_{\alpha}$ denote constants that depend on their indices. All these constants may vary from one line to another.

Let $\bb G$ be the semigroup of non-negative $d\times d$ matrices. 
Let $(g_{n})_{n\geq 1}$ be a sequence of independent and identically distributed  elements of $\bb G$ 
defined on a probability space $(\Omega,\mathscr{F},\mathbb{P})$, 
with common law $\mu$. 
Denote by $\supp \mu$ the support of the measure $\mu$ and by $\Gamma_{\mu}$ 
the closed semigroup spanned by $\supp \mu$. 
For $g \in \Gamma_{\mu}$, 
we define the matrix norm $\|g\| : = \sum_{i = 1}^d  \sum_{j = 1}^d \langle e_i, g e_j \rangle$ 
and set $N(g) = \max \{ \|g\|, \|g\|^{-1} \}$.

In this paper we study the asymptotic behaviour of the $(i,j)$-th entry $G_n^{i,j} = \langle e_i, G_n e_j \rangle$
of the product matrix $G_{n}:=g_{n}\ldots g_{1},$ 
and more generally of the coefficients $\langle f, G_n v \rangle$ of $G_n$,
where $f, v \in \mathcal S$.  
The study of asymptotic behaviour of the coefficients attracted much attention, 
due to its own theoretical interest and a variety of applications. 
In the pioneering work \cite{FK60}, Furstenberg and Kesten introduced the following condition on the law $\mu$:
\begin{conditionA}\label{Condi-FK} 
The matrix $0$ does not belong to $\Gamma_{\mu}$ and there exists a constant $\varkappa_0 > 1$ such that  
for $\mu$-almost every $g = (g^{i,j})_{1 \leq i, j \leq d} \in \bb G$, 
\begin{align}\label{Inequality-FK60} 
g>0 \quad \mbox{ and } \quad  \frac{\max_{1\leq i, j\leq d}  g^{i,j} }{ \min_{1\leq i,j\leq d} g^{i,j} } \leq  \varkappa_0.
\end{align}
\end{conditionA}   

Condition \eqref{Inequality-FK60} 
has been introduced by Furstenberg and Kesten in the pioneering work \cite{FK60}. 
Notice that under this condition the hypothesis that $0\notin \Gamma_{\mu}$ 
is equivalent to the following hypothesis used in \cite{BDGM14}: 
all elements $g$ of $\Gamma_{\mu}$ are allowable in the sense that every row and every column has at least one strictly positive entry.

Assuming \ref{Condi-FK} and the first moment $\int_{\bb G} \log N(g) \mu(dg) < \infty$, 
Furstenberg and Kesten \cite{FK60} established
the strong law of large numbers for the coefficients: for any $f, v \in \mathcal S$, 
it holds that $\lim_{n \to \infty} \frac{1}{n} \log \langle f, G_n v \rangle = \lambda_{\mu}$, $\bb P\mbox{-a.s.},$
where $\lambda_{\mu} \in \bb R$ is independent of $f$ and $v$, and is called the first Lyapunov exponent of $\mu$. 
The central limit theorem was also established:
assuming \ref{Condi-FK} and $\int_{\bb G} \log^{2} N(g) \mu(dg) < \infty$,
\begin{align} \label{CLT-001}
\frac{ \log \langle f, G_n v \rangle - n \lambda_{\mu} }{ \sqrt{n} } \to \mathcal N(0, \sigma^2) \quad \mbox{in law},
\end{align}
where $\sigma^2 \geq 0$ is a constant (see \eqref{def-sigma-001} for its definition). 
The convergence \eqref{CLT-001} was initially proved in \cite{FK60} under the moment condition of order $2+\delta$
for some $\delta>0$; 
it also holds under the second moment condition, as a special case of 
Theorem 3 of Hennion \cite{Hen97}.
Further extensions of the above results can be found in the work of Kingman \cite{Kin73},  Kesten and Spitzer \cite{KS84}, 
Cohn, Nerman and Peligrad \cite{CNP93}, 
  Hennion and Herv\'e \cite{HH01, HH08}. 
However, important limit theorems such as Berry-Esseen theorem, Edgeworth expansion,
upper and lower large deviations and local limit theorems for $\langle f, G_n v \rangle$ 
 have not yet been established in the literature.   

The goal of the paper is twofold.
Our first objective is to study the convergence rate in the central limit theorem \eqref{CLT-001}
under the optimal moment condition. 
In this direction, 
we prove a Berry-Esseen theorem for the coefficient $\langle f, G_n v \rangle$ under
the third moment condition $\int_{\bb G} \log^3 N(g) \mu(dg) < \infty$. 
Under this same moment assumption, we also establish the first-order Edgeworth expansion
for $\langle f, G_n v \rangle$; see Section \ref{Berry-Esseen theorem and Edgeworth expansion}. 

Our second objective is to determine the convergence rate in the law of large numbers for $\langle f, G_n v \rangle$, 
which is precisely determined by upper and lower large deviation asymptotics.  
In this regard, we establish the exact asymptotic behavior of $\bb P (\log \langle f, G_n v \rangle > nq)$ when $q > \lambda_{\mu}$, and $\bb P (\log \langle f, G_n v \rangle < nq')$ when $q' < \lambda_{\mu}$.
As a secondary outcome of our approach, we also establish a local limit theorem provided the optimal second moment condition is met. 
Specific statements of these results can be found in Theorem \ref{Thm-Petrov-Uppertail} and Eq.\ \eqref{LLT-001} in Theorem \ref{Thm_LLT_03}.

It is worth noting that analogous results for the norm cocycle $\log \|G_n v\|$ have been demonstrated in \cite{XGL19a}. 
However, for the lower large deviations, these findings have only been established for 
$q'$ less than but sufficiently close to $\lambda_{\mu}$.
An enduring open question pertains to the case where $q' < \lambda_{\mu}$ is entirely arbitrary.
In this paper, we successfully tackle this issue by applying our lower large deviation asymptotics 
for coefficients $\langle f, G_n v \rangle$, 
as demonstrated in Theorem \ref{Thm-Petrov-Norm-cocycle}.
Additionally, our large deviation results lead to new findings 
regarding local limit theorems with large deviations for coefficients, 
as well as large deviation bounds for the spectral radius $\rho(G_n)$ of $G_n$, 
as presented in Theorems \ref{Thm_LLT_03} and \ref{Thm_LDSpectPosi}.

For the proofs, we have developed a spectral gap theory for operators related to a new cocycle 
$(g, v) \mapsto \sigma_f(g, x) = \log \frac{\langle f, g v \rangle}{\langle f, v \rangle}$, 
which is associated with the coefficients. 
Additionally, we have extended the spectral gap properties 
to encompass the norm cocycle $\log \|G_n v\|$ for all negative exponents $s < 0$ (allowing $q' < \lambda_{\mu}$ to be arbitrary), 
and we utilize these properties to establish the lower large deviations.

The cocycle $\sigma_f$ has not been explored in previous literature but proves to be highly valuable in addressing certain related problems.
By leveraging the results obtained for the cocycle related to the scalar product, it becomes possible to investigate conditioned integrals  
and local theorems for the coefficients of products involving positive random matrices.
It is worth noting that these problems remain open when dealing with invertible matrices. 
The primary challenge lies in the development of the spectral gap theory for the transfer operator associated with the cocycle $\sigma_f$.

\subsection{Edgeworth expansion and Berry-Esseen theorem} \label{Berry-Esseen theorem and Edgeworth expansion}
The Edgeworth expansion and the Berry-Esseen theorem, for sums of independent and identically distributed real-valued random variables, are classical results in probability theory. 
These can be found, for instance, in Petrov's work \cite{Pet75}.
The results presented in this section extend these classical results to the coefficients of products of positive random matrices.

We start by formulating our Edgeworth expansion. 
We need the following condition:

\begin{conditionA}\label{Condi-FK-weak} 
The matrix $0$ does not belong to $\Gamma_{\mu}$ and 
there exists a constant $\varkappa > 1$ such that for $\mu$-almost every $g = (g^{i,j})_{1 \leq i, j \leq d} \in \bb G$, 
\begin{align}\label{Inequality-FK-weak} 
g>0 \quad \mbox{ and } \quad   \frac{\max_{1\leq i \leq d} g^{i,j} }{ \min_{1\leq i \leq d} g^{i,j} } \leq  \varkappa  \quad \forall  1 \leq j \leq d. 
\end{align}
\end{conditionA}

Condition \ref{Condi-FK-weak} is indeed weaker than \ref{Condi-FK},  
since condition \ref{Condi-FK-weak} says that all the entries of each fixed column of the matrix $g \in \supp \mu$ are comparable, 
while condition \ref{Condi-FK} requires that all the entries of  $g \in \supp \mu$ are comparable.

For $g \in \Gamma_{\mu}$ and $v \in \mathcal{S}$, 
we write $g \!\cdot\! v = \frac{g v}{\|g v\|}$ for the projective action of $g$ on $\mathcal{S}$.
With any starting point $G_0 \cdot v: = v \in \mathcal{S}$, 
the Markov chain $(G_n \cdot v)_{n \geq 0}$ 
has a transition operator $P$ defined as follows: 
for any bounded measurable function $\varphi$ on $\mathcal S$, 
\begin{align}\label{def-transition-oper-P}
P \varphi(v) = \int_{\bb G} \varphi(g \cdot v) \mu(dg). 
\end{align}
Under \ref{Condi-FK-weak}, 
there exists a unique invariant probability measure $\nu$ on $\mathcal{S}$
such that for any bounded measurable function $\varphi$ on $\mathcal S$, 
\begin{align}\label{def-invariant-mea}
(\nu P) \varphi : =  \int_{\mathcal{S}} \int_{\bb G} \varphi(g \cdot v) \mu(dg) \nu(dv) 
= \int_{\mathcal{S}} \varphi(v) \nu(dv) =: \nu(\varphi). 
\end{align}
With this notation, under \ref{Condi-FK-weak},  and $\int_{\bb G} \log N(g) \mu(dg) < \infty$, 
the Lyapunov exponent $\lambda_{\mu}$ can be written as
\begin{align}\label{formu-lambda-mu}
\lambda_{\mu} =  \int_{\mathcal S} \int_{\bb G}  \log \|g v\|  \mu(dg)  \nu(dv),  
\end{align}
see for instance \cite{HH08}. 

We say that $\mu$ is {\em arithmetic}, if there is $t>0$ together with $\theta \in [0, 2\pi)$ 
and a measurable function 
$\varphi: \mathcal S \to \bb R$ such that for all $g \in \Gamma_\mu$ 
and all $v \in \supp\nu$, it holds that
\begin{align*}
\exp(i 2 \pi  t  \log \|gv\| - i \theta + i (\varphi(g \cdot v) - \varphi(v))) = 1.
\end{align*}
In other words, $t \log \|g v\|$ is contained in $\bb Z$ 
up to a shift that may depend on $g$ and $v$ through the function $\varphi$. 
We need the following non-arithmeticity condition on $\mu$. 
\begin{conditionA}\label{Condi-NonLattice}
The measure $\mu$ is not arithmetic. 
\end{conditionA}

Condition \ref{Condi-NonLattice} is indeed an extension of the non-lattice 
condition typically applied to sums of independent random variables.
According to \cite[Lemma 2.7]{BS2016}, condition \ref{Condi-NonLattice} implies that $\sigma > 0$, 
which is crucial for the central limit theorem and related probabilistic results to hold. 
Furthermore, condition \ref{Condi-NonLattice} is satisfied if the set $\{ \log \rho(g) : g \in \Gamma_{\mu} \}$ generates a dense additive subgroup of $\mathbb{R}$, where $\rho(g)$ represents the spectral radius of the matrix $g$, defined as $\rho(g) = \lim_{k \to \infty} \|g^k\|^{1/k}$.
The latter condition, introduced by Kesten in \cite{Kes73},  is often more straightforward to verify than
 \ref{Condi-NonLattice} itself, making it a useful criterion in practical applications.

Recall that $N(g) = \max \{ \|g\|, \|g\|^{-1} \}$. 
Under \ref{Condi-FK} and $\int_{\bb G} \log^3 N(g) \mu(dg) < \infty$, the following limits exist 
(cf.\ Lemma \ref{transfer operator_Pit}  below): 
\begin{align}
&\sigma^2 :=  \lim_{n \to \infty} \frac{1}{n} \int_{\mathcal S} \bb E [ \left( \log  \| G_n v \| - n \lambda_{\mu} \right)^2 ] \nu(dv), 
\label{def-sigma-001}\\
& m_3 : = \lim_{n \to \infty} 
\frac{1}{n} \int_{\mathcal{S}} \bb E \left[ \left( \log \|G_n v \| - n \lambda_{\mu} \right)^3  \right] \nu(dv) 
 \in  \bb R.  \label{Def-m3}
\end{align}
For $f, v \in \bb{R}^{d}_+ \setminus \{0\}$, let
\begin{align} \label{drift-b001}
b(v): = \lim_{n \to \infty}
   \mathbb{E} \big[ ( \log \|G_n v\| - n \lambda_{\mu} ) \big], 
\qquad
d(f) :=  \int_{ \mathcal{S} }  \log \langle f, u \rangle \nu(du). 
\end{align}
It will be shown in Lemma \ref{Lem-B0}  that the limit and the integral in \eqref{drift-b001} exist, and that
both functions $b$ and $d$ 
 are Lipschitz continuous on $\mathcal{S}$.
Clearly the definition of $b$ and $d$ can be extended to $f, v \in \bb{R}^{d}_+ \setminus \{0\}$. 
Let $\phi(y) = \frac{1}{\sqrt{2 \pi}} e^{- \frac{y^2}{2}}$, $y \in \bb R$, be the standard normal density function
and let $\Phi(t) = \int_{-\infty}^t \phi(y) dy$, $t \in \bb R$ be the standard normal distribution function.

\begin{theorem}\label{MainThm_Edgeworth_01}
Assume \ref{Condi-FK-weak}, \ref{Condi-NonLattice} and $\int_{\bb G} \log^3 N(g) \mu(dg) < \infty$. 
Then, for any compact set $K \subset \bb{R}^{d}_+ \setminus \{0\}$, 
we have, as $n \to \infty$, uniformly in $f, v  \in K$ and $y \in \mathbb{R}$,
\begin{align}\label{Edge-Exp-Coeff}
& \bb P \left( \frac{\log \langle f, G_n v \rangle - n \lambda_{\mu} }{\sigma \sqrt{n}} \leq y \right)   \notag\\
& =  \Phi(y) + \frac{ m_3 }{ 6 \sigma^3 \sqrt{n}} (1-y^2) \phi(y) 
    -  \frac{b(v) + d(f) }{ \sigma \sqrt{n} } \phi(y)  +  o \left( \frac{ 1 }{\sqrt{n}} \right). 
\end{align}   
\end{theorem}

With the  choice $f = \mathbf 1$, we have that $\| G_n  v\| = \langle \mathbf 1, G_n v \rangle$ and  $d(\mathbf 1) = 0$,
where $\mathbf 1 = (1, \ldots, 1)^{\rm T} \in \bb{R}^{d}_+$. 
Therefore, from \eqref{Edge-Exp-Coeff} we get the following first-order Edgeworth expansion for the vector norm $\|G_n v\|$: as $n \to \infty$,  
uniformly in $v  \in K$ and $y \in \mathbb{R}$,
\begin{align}\label{Edge-Exp-VecNorm}
& \bb P \left( \frac{\log \| G_n v \| - n \lambda_{\mu} }{\sigma \sqrt{n}} \leq y \right)  \notag\\
& =  \Phi(y) + \frac{ m_3 }{ 6 \sigma^3 \sqrt{n}} (1-y^2) \phi(y) 
    -  \frac{b(v) }{ \sigma \sqrt{n} } \phi(y)  +  o \left( \frac{ 1 }{\sqrt{n}} \right). 
\end{align}   
With the choice $f = v = \mathbf 1$, we have that $\| G_n \| = \langle \mathbf 1, G_n \mathbf 1 \rangle$ and again $d(\mathbf 1) = 0$. 
Hence, \eqref{Edge-Exp-Coeff} implies
the first-order Edgeworth expansion for the matrix norm $\|G_n\|$: as $n \to \infty$,  uniformly in $y \in \mathbb{R}$,
\begin{align}\label{Edge-Exp-Norm}
& \bb P \left( \frac{\log \| G_n \| - n \lambda_{\mu} }{\sigma \sqrt{n}} \leq y \right)  \notag\\
& =  \Phi(y) + \frac{ m_3 }{ 6 \sigma^3 \sqrt{n}} (1-y^2) \phi(y) 
    -  \frac{b(\mathbf 1) }{ \sigma \sqrt{n} } \phi(y)  +  o \left( \frac{ 1 }{\sqrt{n}} \right). 
\end{align}   
Both \eqref{Edge-Exp-VecNorm} and \eqref{Edge-Exp-Norm} are new to our knowledge. 

The third moment condition assumed in Theorem \ref{MainThm_Edgeworth_01} is indeed optimal. 
It is worth noting that, for invertible matrices, the first-order Edgeworth expansion for the coefficients $\langle f, G_n v \rangle$ 
has recently been derived in \cite{XGL22arXiv}. 
This achievement was made using a different method, which relies on the H\"older regularity of the invariant measure $\nu$ 
and a partition of unity on the projective space $\mathbb{P}^{d-1}$. 
However, it is important to mention that \cite{XGL22arXiv} requires an exponential moment, 
and the question of whether this requirement can be relaxed to the third moment remains an open problem.

Theorem \ref{MainThm_Edgeworth_01} clearly entails the Berry-Esseen theorem for the coefficients $\langle f, G_n v \rangle.$ However, the condition \ref{Condi-NonLattice}, as required in Theorem \ref{MainThm_Edgeworth_01}, can be somewhat relaxed. 
We only need to assume that $\sigma > 0,$ as stated below.
The result presented below encompasses also the matrix norm $\|G_n\|$ and the spectral radius $\rho(G_n)$.

\begin{theorem}\label{Thm-BerryEsseen}
Assume \ref{Condi-FK-weak}, $\sigma >0$ and $\int_{\bb G} \log^3 N(g) \mu(dg) < \infty$. 
Then, for any compact set $K \subset \bb{R}^{d}_+ \setminus \{0\}$, 
there exists a constant $c>0$ such that for any $n \geq 1$, $f, v  \in K$ and $y \in \mathbb{R}$,
\begin{align}\label{Main-BerryEsseen-Coeff}
\left| \bb P \left( \frac{\log \langle f, G_n v \rangle - n \lambda_{\mu} }{\sigma \sqrt{n}} \leq y \right)  - \Phi(y)  \right|
\leq  \frac{c}{\sqrt{n}}.
\end{align}
Moreover, \eqref{Main-BerryEsseen-Coeff} remains valid when $\langle f, G_n v \rangle$ is replaced by $\|G_n\|$ or $\rho(G_n)$.
\end{theorem}    

Theorem \ref{Thm-BerryEsseen} provides the convergence rate in the central limit theorem for the coefficients 
$\langle f, G_n v \rangle$, 
as originally established by Furstenberg and Kesten \cite{FK60} under the stronger condition \ref{Condi-FK}.
It is noteworthy that in \cite{XGL19c}, the bound \eqref{Main-BerryEsseen-Coeff} has been derived under the exponential moment condition, specifically $\int_{\bb G} N(g)^{\eta} \mu(dg) < \infty$ for some constant $\eta > 0$. 
However, in our context, we  require only the third moment condition, which is optimal.
Theorem \ref{Thm-BerryEsseen} 
has been announced in the arXiv version \cite{XGL22arXiv-b}. 
The Berry-Esseen bound for the matrix norm $\|G_n\|$ has then been improved 
  in \cite[Theorem 7.1]{CDM23}
under a condition weaker than \ref{Condi-FK-weak}, 
namely that $\mu$ is strictly contracting, cf.\ Definition 2.2 in \cite{CDM23}.
Under this less restrictive condition, additional results concerning 
the spectral radius $\rho(G_n)$
and coefficients $\langle f, G_n v \rangle$ have also been established in \cite{CDM23}.
It is worth noting that the proofs in \cite{CDM23} rely on the blocking approach developed recently in \cite{CDMP21}. 
This approach is powerful and opens new research avenues, 
particularly in the study of limit theorems for products of random matrices under minimal moment assumptions.

Our approach is different from those in \cite{CDMP21, CDM23}. 
 To establish Theorem \ref{MainThm_Edgeworth_01}, we develop a spectral gap theory 
for the following cocycle: for any $f \in \mathcal S$,  $g\in \Gamma_{\mu}$ and $v\in \mathcal S$,
\begin{align}\label{def-cocycle-new}
(g, v) \mapsto  \sigma_f(g, v) := \log  \frac{\langle f, g v \rangle}{ \langle f, v \rangle }. 
\end{align}
To the best of our knowledge, the cocycle \eqref{def-cocycle-new} has not yet been studied in the literature. 
Then we apply the Nagaev-Guivarc'h method to establish the first-order Edgeworth expansion
for  $\sigma_f(G_n, v)$. 
The next step in our strategy is to pass from the Edgeworth expansion for  $\sigma_f(G_n, v)$
to that for  $\log \langle f, G_n v \rangle$
by making use of condition \ref{Condi-FK-weak} and the technique of taking conditional expectation.

Interestingly, for invertible matrices, the situation becomes more intricate. 
Cuny, Dedecker, Merlev\`ede and Peligrad \cite{CDMP21b} recently established a rate of convergence of order $\log n/\sqrt{n}$ 
for the coefficients $\langle f, G_n v \rangle$ under the exponential moment condition. 
Dinh, Kaufmann and Wu \cite{DKW21, DKW21b} improved this result by achieving the optimal rate of $1/\sqrt{n}$ 
under the same moment condition. 
However, proving the Berry-Esseen theorem under the third moment condition remains an open problem.
We also note that for the more straightforward case of the norm cocycle $\log \|G_n v\|$, 
recent significant progress has been made by Cuny, Dedecker, Merlev\`ede and Peligrad \cite{CDMP21}, 
where the Berry-Esseen theorem for $\log \|G_n v\|$ is established under the fourth moment condition.

\subsection{Upper and lower large deviations}\label{Upper and lower large deviations}

To formulate the large deviation results for the coefficients $\langle f, G_n v \rangle$,  
we need to state some spectral properties of transfer operators 
related to the vector norm $\|G_n v\|$.
Set $I_{\mu} = I_{\mu}^+ \cup  I_{\mu}^-$, where 
\begin{align*}
I_{\mu}^+ = \left\{ s \geq 0: \mathbb{E}( \|g_1\|^{s}) < \infty \right\}
\quad \mbox{and} \quad
I_{\mu}^- = \left\{ s \leq 0: \mathbb{E}(\|g_1\|^{s}) < \infty \right\}. 
\end{align*}
The sets $I_{\mu}^+$ and $I_{\mu}^-$ are non-empty since they contain at least the point $s = 0$. 
Denote by $(I_{\mu}^+)^{\circ}$  the interior of $I_{\mu}^+$
and by  $(I_{\mu}^-)^{\circ}$ the interior of $I_{\mu}^-$. 
In the sequel we assume that $(I_{\mu}^+)^{\circ}$ (or $(I_{\mu}^-)^{\circ}$) 
is non-empty, meaning that there exists $s_0 \neq 0$ such that $\mathbb{E}(\|g_1\|^{s_0}) < \infty$, 
which is in fact an exponential moment condition for $\log \|g_1\|$.
For any $s \in I_{\mu}$,  
the following limit exists: 
\begin{align*}
\kappa(s) = \lim_{n\to\infty}\left(\mathbb{E}\| G_n \|^{s}\right)^{\frac{1}{n}},  
\end{align*}
see Proposition \ref{change-measure-neg01} for more details.  
Let $\Lambda(s) = \log \kappa(s)$, $s \in I_{\mu}$. Then 
the function $\Lambda$ plays the same role as the cumulant generating function in the case of 
sums of independent and identically distributed real-valued random variables, 
and it is convex and analytic on $I_{\mu}^{\circ}$. 
Introduce the Fenchel-Legendre transform of $\Lambda$: 
\begin{align*}
\Lambda^{\ast}(t) = \sup_{s \in I_{\mu}} \{ st - \Lambda(s) \}, \quad t \in \bb R. 
\end{align*}
If $q = \Lambda'(s)$ for some $s\in I_{\mu}^{\circ} \setminus \{0\}$, then $\Lambda^*(q) = s q - \Lambda(s) >0$;
if $q = \Lambda'(0)$, then $\Lambda^*(q) = 0$.

For any $s\in I_{\mu}$, the transfer operator $P_s$ and the conjugate transfer operator $P_{s}^{*}$ are defined as follows: 
for any bounded measurable function $\varphi$ on $\mathcal{S}$ and $v \in \mathcal S$, 
\begin{align}\label{transfoper001}
P_{s}\varphi(v)   = \int_{\bb G} \|g v\|^{s} \varphi( g \!\cdot\! v ) \mu(dg),  \quad  
P_{s}^{*}\varphi(v)   = \int_{\bb G}  \|g^{\mathrm{T}} v\|^{s} \varphi(g^{\mathrm{T}} \!\cdot\! v) \mu(dg), 
\end{align}
where $g^\mathrm{T}$ is the transpose of $g$.
In the sequel,   
for any Borel measure $\rho$ on $\mathcal S$,
$P_s \rho$ means the operator defined by the identity
 $(P_s \rho)\varphi = \rho (P_s \varphi)$
for any bounded measurable function $\varphi$ on $\mathcal S$; 
similar convention will apply for other operators like $P_s^*$.
The $n$-th iteration of $P_s$ is denoted by $P_s^n$ and can be computed as follows: 
for any bounded measurable function $\varphi$ on $\mathcal S$ and $v \in \mathcal S$, 
\begin{align*} 
P_s^n \varphi(v) = \int_{\bb G}\ldots \int_{\bb G}  \| g_n\ldots g_1 v\|^{s}  \varphi( (g_n\ldots g_1) \cdot v) \mu(dg_1)\ldots \mu(dg_n).
\end{align*}
Under condition \ref{Condi-FK-weak}, 
the operator $P_{s}$ has a unique probability eigenmeasure $\nu_{s}$ on $\mathcal S$
and a unique (up to a scaling constant) strictly positive and continuous eigenfunction $r_{s}$ on $\mathcal S,$ 
corresponding to the eigenvalue $\kappa(s)$: 
\begin{align*}
P_s r_s=\kappa(s)r_s, \quad P_s\nu_{s}=\kappa(s)\nu_{s}.
\end{align*}
Similarly, there exist a unique probability eigenmeasure $\nu_{s}^{*}$ on $\mathcal S$
and a unique (up to a scaling constant) strictly positive and continuous function $r_{s}^{\ast}$
 such that
$$ P_{s}^{*}r_{s}^{*}=\kappa(s)r_{s}^{*}, \quad P_{s}^{*}\nu_{s}^{*}=\kappa(s)\nu_{s}^{\ast}.$$
With a particular choice of the scaling constant,
 the eigenfunction $r_{s}$ (resp.\ $r_{s}^*$) 
 can be expressed as below: 
\begin{align*}
r_{s}(v)   = \int_{\mathcal{S}} \langle v, u \rangle^{s} \nu^*_{s}(du),  
\quad    
r_{s}^*(v)   = \int_{\mathcal{S}} \langle v, u \rangle^{s} \nu_{s}(du),  \quad   v \in \mathcal{S}.
\end{align*}
From these expressions, the eigenfunctions $r_{s}$ and $r_{s}^*$ can be extended to $\bb R_+^d \setminus \{0\}$.

Next we introduce a Banach space which will be used in the sequel. 
The space $\mathcal{S}$ is equipped with 
the Hilbert cross-ratio metric $\mathbf d$: 
for any $u, v \in \mathcal{S}$,    
\begin{align}\label{Def-distance}
\mathbf{d}(u, v) = \frac{1- m(u,v)m(v,u)}{1 + m(u,v)m(v,u)}, 
\end{align}
where $m(u,v)=\sup\{ \lambda \geq 0: \   \lambda v_i \leq u_i,\  \forall i=1,\ldots, d  \}$ for $u =(u_1, \ldots, u_d)$ 
and $v = (v_1, \ldots, v_d)$. 
Let $\mathcal{C}(\mathcal{S})$ be the space of complex-valued continuous functions on $\mathcal{S}$. 
For any $\varphi\in \mathcal{C(S)}$ and $\gamma>0$, denote
\begin{align*}
\|\varphi\|_{\infty}: = \sup_{v \in \mathcal{S}} |\varphi(v)|
\quad \mbox{and} \quad
[\varphi]_{\gamma} : = \sup_{u, v \in \mathcal{S}: u \neq v} \frac{|\varphi(u) - \varphi(v)|}{\mathbf{d}(u, v)^{\gamma}}, 
\end{align*}
and the Banach space
\begin{align*}
\scr B_{\gamma}:= \Big\{ \varphi \in \mathcal{C(S)}: 
                    \|\varphi\|_{\gamma} := \|\varphi\|_{\infty} + [\varphi]_{\gamma} < \infty  \Big\}.
\end{align*}

Now we state the Bahadur-Rao-Petrov type upper and lower large deviation asymptotics 
for the coefficients $\langle f, G_n v \rangle$ under an exponential moment condition.
Notice that in the following,  
condition \ref{Condi-FK} will be needed instead of \ref{Condi-FK-weak} while considering negative $s$.

\begin{theorem}\label{Thm-Petrov-Uppertail}
Let $K \subset \bb{R}^{d}_+ \setminus \{0\}$ be a compact set. 

\noindent 1.   
Assume \ref{Condi-FK-weak} and \ref{Condi-NonLattice}. 
Let $J^+ \subseteq (I_\mu^+)^\circ$ be a compact set. 
Then, for any $0 < \gamma \leq \min_{t \in J^+} \{ t, 1\}$,  we have, as $n \to \infty$, 
uniformly in $s \in J^+$, $f, v  \in K$ and $\varphi \in \scr B_{\gamma}$, with $q = \Lambda'(s)$ and $\sigma_s = \sqrt{\Lambda''(s)}$, 
\begin{align}\label{Petrov-Upper}
& \mathbb{E} \left[ \varphi(G_n \cdot v) \mathds 1_{ \left\{ \log \langle f,  G_n v \rangle \geq n q \right\} }  \right]   \notag\\
& = \frac{r_s(v)}{\nu_s(r_s)} \frac{ \exp(- n\Lambda^{*}(q)) }{ s\sigma_{s}\sqrt{2\pi n} }  
    \left[ \int_{\mathcal{S}} \varphi(u) \langle f, u \rangle^s \nu_s(du) + \| \varphi \|_{\gamma} o(1) \right]. 
\end{align}

\noindent 2.  
Assume \ref{Condi-FK} and \ref{Condi-NonLattice}. 
Let $J^- \subseteq (I_\mu^-)^\circ$ be a compact set. 
Then, for any $0 < \gamma \leq \min_{t \in J^-} \{ |t|, 1\}$,  we have, as $n \to \infty$,  
uniformly in $s \in J^-$, $f, v  \in K$ and $\varphi \in \scr B_{\gamma}$, 
with $q = \Lambda'(s)$ and $\sigma_s = \sqrt{\Lambda''(s)}$, 
\begin{align}\label{Petrov-Lower}
& \mathbb{E} \left[ \varphi(G_n \cdot v) \mathds 1_{ \left\{ \log \langle f,  G_n v \rangle \leq n q  \right\} }  \right]  \notag\\
& = \frac{r_s(v)}{\nu_s(r_s)} \frac{ \exp(- n\Lambda^{*}(q)) }{ -s\sigma_{s}\sqrt{2\pi n} }  
    \left[ \int_{\mathcal{S}} \varphi(u) \langle f, u \rangle^s \nu_s(du) + \| \varphi \|_{\gamma} o(1) \right]. 
\end{align}
\end{theorem}

The uniformity in $s$ over compact sets is very important in various applications. 
For example, it allows us to derive from the above result the local limit theorem with large deviations, 
see \eqref{LLT-LD} of Theorem \ref{Thm_LLT_03}. 
It has important applications
in related models such as 
multitype  branching processes in random environment. 
It also plays a crucial role in \cite{MX22} 
to study limit theorems for exceedance times of multivariate perpetuity sequences arising from financial mathematics,
which generalizes the main results in \cite{BurCollDamZien2016} from the univariate case to the multivariate case.

Taking $\varphi = 1$ in \eqref{Petrov-Upper}, we get the following upper large deviation result for the coefficients $\langle f,  G_n v \rangle$:   
uniformly in $s \in J^+$ and $f, v  \in K$, with $q = \Lambda'(s)$ and $\sigma_s = \sqrt{\Lambda''(s)}$, 
\begin{align}\label{LD-BRP-Coeff}
\mathbb{P} \Big( \log \langle f,  G_n v \rangle \geq n q \Big)
= \frac{r_s(v) r_s^*(f)}{\nu_s(r_s)} \frac{ \exp(- n\Lambda^{*}(q)) }{ s\sigma_{s}\sqrt{2\pi n} }  \big[ 1 + o(1) \big].
\end{align}
A lower large deviation asymptotic can also be  obtained as a consequence of \eqref{Petrov-Lower}.

Since $\| G_n \| = \langle \mathbf 1, G_n \mathbf 1 \rangle$,  
taking $f = v = \mathbf 1$ in \eqref{LD-BRP-Coeff} and using the fact that $r_s(\mathbf 1) = r_s^*(\mathbf 1) = 1$, 
we get the following upper large deviation asymptotic
for the matrix norm $\|G_n\|$:   uniformly in $s \in J^+$, with $q = \Lambda'(s)$ and $\sigma_s = \sqrt{\Lambda''(s)}$, 
\begin{align}\label{LD-BRP-MatrixNorm}
\mathbb{P} \Big( \log \|G_n \| \geq n q \Big)
= \frac{1}{\nu_s(r_s)} \frac{ \exp(- n\Lambda^{*}(q)) }{ s\sigma_{s}\sqrt{2\pi n} }  \big[ 1 + o(1) \big].  
\end{align}
A similar lower large deviation asymptotic can also be formulated. 

Theorem \ref{Thm-Petrov-Uppertail} is proved
by developing a spectral gap theory for the new cocycle $\sigma_f(G_n, v)$ and by using the Edgeworth expansion 
under the changed probability measure 
for the couple $(G_n \!\cdot\! v, \sigma_f(G_n, v))$
with a target function on the Markov chain $(G_n \!\cdot\! v)$, see \eqref{Thm_Edgeworth_Target_02}. 

For invertible matrices, 
the corresponding upper large deviation asymptotic \eqref{Petrov-Upper}
has been recently established in \cite{XGL19d} using a different method based on 
the H\"older regularity of the eigenmeasure $\nu_s$ together with 
a smooth partition of unity of the projective space $\bb P^{d-1}$. 
The corresponding lower large deviation asymptotic \eqref{Petrov-Lower}
has also been proved in \cite{XGL19d}, but only for $s<0$ sufficiently close to $0$. 
It remains an open problem to prove  \eqref{Petrov-Lower}
for arbitrary $s\in (I_\mu^-)^\circ$ (in the case of invertible matrices).

\subsection{Applications}\label{Subsect-Application}

The following theorem presents three results for local probabilities for the coefficients $\langle f, G_n v \rangle$. 
The first is an "ordinary" local limit theorem, the second involves moderate deviations, 
and the third involves large deviations. 
Let $\zeta(t)$ be the Cram\'{e}r series of $\Lambda$  given by:
\begin{align*}
\zeta (t) = \frac{\gamma_{3} }{ 6 \gamma_{2}^{3/2} }  
  +  \frac{ \gamma_{4} \gamma_{2} - 3 \gamma_{3}^2 }{ 24 \gamma_{2}^3 } t
  +  \frac{\gamma_{5} \gamma_{2}^2 - 10 \gamma_{4} \gamma_{3} \gamma_{2} + 15 \gamma_{3}^3 }{ 120 \gamma_{2}^{9/2} } t^2
  +  \ldots 
\end{align*}
with $\gamma_{k} = \Lambda^{(k)} (0)$, cf.\ \cite{Pet75, XGL19b}. 
Note that $\gamma_{1} = \lambda_{\mu}$ and $\gamma_{2} = \sigma^2$. 

\begin{theorem}\label{Thm_LLT_03}
Assume \ref{Condi-FK-weak} and \ref{Condi-NonLattice}.  
Let $a_1 < a_2$ be real numbers and 
$K \subset \bb{R}^{d}_+ \setminus \{0\}$ be a compact set.

\noindent 1.  
If  $\int_{\bb G} \log^2 N(g) \mu(dg) < \infty$, 
then, as $n \to \infty$,  
uniformly in $f, v  \in K$, $|l| = o(\frac{1}{\sqrt{n}})$ and $\varphi \in \scr B_{1}$, 
\begin{align}\label{LLT-001}  
\mathbb{E} \left[ \varphi(G_n \!\cdot\! v) 
      \mathds 1_{ \left\{ \log \langle f,  G_n v \rangle - n \lambda_{\mu} \in   [a_1, a_2] + \sigma nl  \right\} }  \right] 
= \frac{ a_2 - a_1 }{\sigma \sqrt{2\pi n} } 
  \big[ \nu(\varphi)  + \| \varphi \|_{1}  o(1) \big].
\end{align}

\noindent 2.  
Let $(l_n)_{n \geq 1}$ be any sequence of positive numbers satisfying $l_n \to 0$ as $n \to \infty$.  
If $\int_{\bb G} N(g)^{\eta} \mu(dg) < \infty$ for some constant $\eta>0$,  
then,  as $n \to \infty$,  
 uniformly in $f, v  \in K$, $|l| \leq l_n$ and $\varphi \in \scr B_{1}$, 
\begin{align}\label{LLT-MD}  
&  \mathbb{E} \left[ \varphi(G_n \!\cdot\! v) 
      \mathds 1_{ \left\{ \log \langle f,  G_n v \rangle - n \lambda_{\mu} \in   [a_1, a_2] + \sigma nl  \right\} }  \right]  \notag\\
 & = \frac{a_2 - a_1}{\sigma \sqrt{2 \pi n}}  
  e^{ -\frac{nl^2}{2} +   n l^3 \zeta (l) }
  \big[ \nu(\varphi)  + \| \varphi \|_{1}  o(1) \big].
\end{align} 

\noindent 3.  
Let $J^+$ be a compact subset in $(I_\mu^+)^\circ$.       
Then, for any $0 < \gamma \leq \min_{t \in J^+} \{ |t|, 1\}$, 
 as $n \to \infty$,  
 uniformly in $s \in J^+$, $f, v  \in K$ and $\varphi \in \scr B_{\gamma}$, with $q = \Lambda'(s)$ and $\sigma_s = \sqrt{\Lambda''(s)}$, 
\begin{align}\label{LLT-LD}  
& \mathbb{E} \left[ \varphi(G_n \!\cdot\! v) \mathds 1_{ \left\{ \log \langle f,  G_n v \rangle  \in   [a_1, a_2] + nq  \right\} }  \right]  
   \notag\\
& = \frac{r_s(v)}{\nu_s(r_s)}  \frac{e^{-s a_1} - e^{-s a_2}}{s} 
   \frac{ \exp(- n\Lambda^{*}(q)) }{ \sigma_{s}\sqrt{2\pi n} }  
    \left[ \int_{\mathcal{S}} \varphi(u) \langle f, u \rangle^s \nu_s(du) + \| \varphi \|_{\gamma} o(1) \right]. 
\end{align}
Moreover, with the assumption \ref{Condi-FK} replacing \ref{Condi-FK-weak}, 
 the third assertion also holds when $J^+$ is replaced by a compact subset  $J^-$ in $(I_\mu^-)^\circ$. 
\end{theorem}

Note that the asymptotic \eqref{LLT-001} holds under the optimal
second moment condition $\int_{\bb G} \log^2 N(g) \mu(dg) < \infty$. 
Similar result for invertible matrices  was recently obtained in \cite{GQX20, XGL23SPA} under an exponential moment condition. 
The question is open whether \eqref{LLT-001} holds under the second moment assumption. 

Contrary to the ordinary local limit theorem \eqref{LLT-001}, the local limit theorem with moderate deviations \eqref{LLT-MD} is obtained 
under the exponential moment condition, which is expressed by the fact $\int_{\bb G} N(g)^{\eta} \mu(dg) < \infty$ for some constant $\eta>0$. 
For invertible matrices, the asymptotic \eqref{LLT-MD} is obtained recently in \cite{XGL23SPA}, also under an exponential moment condition,  
using a different approach 
via discretization techniques of the projective space $\bb P^{d-1}$ and via the application of the H\"older regularity of the invariant measure $\nu$.

The local limit theorem with large deviations \eqref{LLT-LD} holds under the exponential moment condition
$s \in J^+$, where $J^+ \subseteq (I_\mu^+)^\circ$.
This result is a consequence of Theorem \ref{Thm-Petrov-Uppertail},
where the uniformity in $s$ plays an important role. 
For invertible matrices, the asymptotic \eqref{LLT-LD} is obtained recently in \cite{XGL19d}, 
for $s \in (I_{\mu}^+)^{\circ}$ and for  $s \in (I_{\mu}^-)^{\circ}$ sufficiently close to $0$, using a different approach.

From Theorem \ref{Thm-Petrov-Uppertail}, 
one can deduce the following lower large deviation asymptotic for the vector norm $\|G_n v\|$. 

\begin{theorem}\label{Thm-Petrov-Norm-cocycle}
Assume \ref{Condi-FK} and \ref{Condi-NonLattice}. 
Let  $K \subset \bb{R}^{d}_+ \setminus \{0\}$ and $J^- \subseteq (I_\mu^-)^\circ$ be compact sets. 
Then, for any $0 < \gamma \leq \min_{t \in J^-} \{ |t|, 1\}$, as $n \to \infty$,  
uniformly in $s \in J^-$, $v  \in K$ and $\varphi \in \scr B_{\gamma}$, 
with $q = \Lambda'(s)$ and $\sigma_s = \sqrt{\Lambda''(s)}$, 
\begin{align}\label{Petrov-Lower-Norm-cocycle}
 \mathbb{E} \left[ \varphi(G_n \cdot v) \mathds 1_{ \left\{ \log \| G_n v \| \leq n q  \right\} }  \right]  
  = \frac{r_s(v)}{\nu_s(r_s)} \frac{ \exp(- n\Lambda^{*}(q)) }{ -s\sigma_{s}\sqrt{2\pi n} }  
    \big[ \nu_s(\varphi) + \| \varphi \|_{\gamma} o(1) \big]. 
\end{align}
\end{theorem}

Theorem \ref{Thm-Petrov-Norm-cocycle} improves the result in  \cite{XGL19a} where 
the asymptotic \eqref{Petrov-Lower-Norm-cocycle} has been established 
for $s \in (I_\mu^-)^\circ$ sufficiently close to $0$.

We also obtain the following upper and lower large deviation bounds for the spectral radius $\rho(G_n)$.

\begin{theorem}\label{Thm_LDSpectPosi}  
\noindent 1.   
Assume  \ref{Condi-FK-weak} and \ref{Condi-NonLattice}. 
Let $J^+ \subseteq (I_\mu^+)^\circ$ be a compact set. 
Then there exist constants $0 < c < C < \infty$ such that
uniformly in $s \in J^+$, with $q = \Lambda'(s)$, 
\begin{align}\label{LD_Spec01}
c  & <   
\liminf_{n\to \infty} 
\sqrt{n}  \,  e^{ n \Lambda^*(q) }  \mathbb{P}  \big(\log \rho(G_n)  \geq nq \big)   \nonumber\\
& \leq  
\limsup_{n\to \infty} 
\sqrt{n}  \,  e^{ n \Lambda^*(q) }  \mathbb{P}  \big(\log \rho(G_n)  \geq nq \big) <  C. 
\end{align}

\noindent 2.  
Assume  \ref{Condi-FK} and \ref{Condi-NonLattice}. 
Let $J^- \subseteq (I_\mu^-)^\circ$ be a compact set. 
Then there exist constants $0 < c < C < \infty$ such that 
uniformly in $s \in J^-$, with $q = \Lambda'(s)$, 
\begin{align}\label{LD_Spec02}
c & <  
\liminf_{n\to \infty} 
\sqrt{n}  \,  e^{ n \Lambda^*(q) }  \mathbb{P}  \big(\log \rho(G_n) \leq nq \big)  \nonumber\\
& \leq    
\limsup_{n\to \infty} 
\sqrt{n}  \,  e^{ n \Lambda^*(q) }  \mathbb{P}  \big(\log \rho(G_n) \leq nq \big) <  C. 
\end{align} 
\end{theorem}

Theorem \ref{Thm_LDSpectPosi}  implies the following
large deviation principles for $\rho(G_n)$: under conditions \ref{Condi-FK-weak} and \ref{Condi-NonLattice}, 
uniformly in $s \in J^+$,  with $q = \Lambda'(s)$, 
\begin{align}\label{LDP-101}
\lim_{n \to \infty}  \frac{1}{n}  \log 
\mathbb{P} \big(  \log \rho(G_n)   \geq  n q  \big)
 = - \Lambda^*(q);
\end{align}
under conditions \ref{Condi-FK} and \ref{Condi-NonLattice}, 
uniformly in $s \in J^-$, with $q = \Lambda'(s)$, 
\begin{align}\label{LDP-102}
\lim_{n \to \infty}  \frac{1}{n}  \log 
\mathbb{P} \big(  \log \rho(G_n)   \leq  n q  \big)
 = - \Lambda^*(q).  
\end{align}

Theorem \ref{Thm_LDSpectPosi} is proved by using Theorem \ref{Thm-Petrov-Uppertail} 
and the Collatz-Wielandt formula for positive matrices.
For invertible matrices, much less is known. 
In this case, the Collatz-Wielandt formula does not hold, 
and the question of obtaining results similar to \eqref{LD_Spec01} and \eqref{LD_Spec02}  is still open.
The large deviation principles \eqref{LDP-101} and \eqref{LDP-102} for invertible matrices  
have been conjectured by Sert \cite{Ser19},  
and have been proved recently by Boulanger, Mathieu, Sert and  Sisto \cite{BMSS20} only 
 in the special case of products of $2\times 2$ matrices.

\section{Spectral gap theory for the norm cocycle} \label{Sec-SG-Norm}

The goal of this section is to develop the spectral gap theory for the transfer operator $P_{s}$ (cf.\ \eqref{transfoper001})
associated with the vector norm $\|G_n v\|$. This theory will be used later to establish the spectral gap properties of the transfer operator $P_{s,f}$, which is related to the coefficients $\langle f, G_n v \rangle$ and discussed in Section \ref{sec:spec gap entries}. The spectral gap properties of $P_s$ have numerous applications and have been studied in \cite{GL08, BDGM14, Gui15, GL16, BS2016} for $s \in I_{\mu}^+$, as well as in \cite{LeP82, BQ2017, XGL19b} for $s<0$ sufficiently close to $0$.
However, for arbitrary $s \in I_{\mu}^-$ (where $s<0$ may be far from $0$), 
the spectral gap properties of $P_{s}$ has not yet been developed in the literature. 
We shall study these properties 
 by following the approach from \cite{BDGM14, GL16}. 

\subsection{Preliminary statements}

Recall that the Hilbert cross-ratio metric $\mathbf d$ is defined by \eqref{Def-distance}. 
By \cite{Hen97}, 
there exists a constant $c >0$ such that for any $u, v\in \mathcal{S}$, 
\begin{align} \label{DistanPosi}
\|u-v\| \leq c \mathbf{d}(u,v), 
\end{align}
and, under condition \ref{Condi-FK}, there exists a constant $0 < a <1$ such that 
for any $g \in \Gamma_{\mu}$ and $u, v\in \mathcal{S}$, 
\begin{align} \label{Contracting inequa}
\mathbf{d}(g\cdot u, g\cdot v)\leq a \mathbf{d}(u,v). 
\end{align}
For $\epsilon \in (0,1)$, let
\begin{align*}
\mathcal{S}_{\epsilon}
= \left\{ v \in \mathcal{S}:  \langle f, v \rangle \geq \epsilon \ \mbox{for all $f \in \mathcal{S}$} \right\}. 
\end{align*}
Below we give equivalent formulations of condition \ref{Condi-FK} and \ref{Condi-FK-weak}. 
The proof is similar to that in \cite[Lemma 3.2]{XGL19c}, where the difference is that,   
 instead of the $L_1$-vector norm, the $L_2$-vector norm $\|v\| =  \sqrt{ \sum_{i=1}^d \langle v, e_i \rangle^2 }$ is considered in \cite{XGL19c}. 
For any set $B \subseteq \mathcal{S}$, we denote $g \cdot B = \{ g \cdot v: v \in B \}$.

The following lemma taken from \cite[Lemma 2]{FK60} says that if \eqref{Inequality-FK60} holds 
for $g \in \supp \mu$, 
then it also holds for $g \in \cup_{n=1}^{\infty} (\supp \mu)^n$ (for some different constant).

\begin{lemma}[\cite{FK60}]  \label{Lem-FK60}
Assume that \eqref{Inequality-FK60} holds for every $g \in \supp \mu$. 
Then there exists a constant $\varkappa'_0 > 1$ such that for any 
 $g \in \cup_{n=1}^{\infty} (\supp \mu)^n$, 
\begin{align*} 
g>0 \quad \mbox{ and } \quad  
\frac{\max_{1\leq i, j\leq d}  g^{i,j} }{ \min_{1\leq i,j\leq d} g^{i,j} } \leq  \varkappa'_0.
\end{align*}
\end{lemma}

\begin{lemma} \label{lem equiv Kesten}
Assume that the matrix $0$ does not belong to $\Gamma_{\mu}$. 
Then the following four conditions are equivalent: 
\begin{enumerate}

\item 
\eqref{Inequality-FK60} holds for $\mu$-almost every $g \in \bb G$. 

\item 
\eqref{Inequality-FK60} holds for all $g \in \supp \mu$.

\item 
\eqref{Inequality-FK60} holds for all $g \in \Gamma_{\mu}$.

\item 
There exists a constant $\epsilon \in (0, 1)$ such that
\begin{align} \label{equicon A4-bb}
g \cdot \mathcal{S} \subseteq \mathcal{S}_{\epsilon} 
\quad  \mbox{and} \quad  g^{{\mathrm T}} \cdot \mathcal{S} \subseteq \mathcal{S}_{\epsilon}
\quad  \mbox{for $g \in \supp \mu$}. 
\end{align}

\end{enumerate}

\end{lemma}

\begin{proof}
(a) We first prove that (1) $\Leftrightarrow$ (2). 
Notice that (2) implies (1) since $\mu (\supp \mu) =1$.
 So we just need to prove that (1) implies (2). 
For this, we take $g \in \supp \mu$. Let 
$\bb G_1 = \{g \in \bb G:  g \mbox{ satisfies \eqref{Inequality-FK60}} \}$. 
Then $\mu (\bb G_1) = 1$. 
For each $n \geq 1$, there is a ball $B(g, 1/n) = \{ g' \in \bb G: \|g' - g\| < 1/n \}$ with center $g$ and radius $1/n$
such that $\mu (B(g, 1/n)) > 0$. Therefore, there is $g_n \in B(g, 1/n) \cap \bb G_1$ such that $g_n$ satisfies \eqref{Inequality-FK60}
(with the same constant $\varkappa_0$). 
As $g_n \to g$,  by passing to the limit in $g_n^{i,j} \leq \varkappa_0 g_n^{k,l}$, we see that $g$ also satisfies the same inequality $g^{i,j} \leq \varkappa_0 g^{k,l}$ for all $1 \leq i, j, k, l \leq d$. 
With this and the condition $0 \notin \Gamma_{\mu}$ (so  $g \neq 0$),  we conclude that $g>0$. 
So far, we have shown that \eqref{Inequality-FK60} holds for every $g \in \supp \mu$. 
Hence, we have proved that (1) $\Leftrightarrow$ (2). 

(b) We next prove that (2) $\Leftrightarrow$ (3). 
Clearly (3) implies (2). So we just need to prove that (2) implies (3). 
For any $g \in \Gamma_{\mu}$, there is a sequence $g_n \in (\supp \mu)^n$, $n \geq 1$, such that $g_n \to g$.
By Lemma \ref{Lem-FK60},  we see that $g_n$ satisfies \eqref{Inequality-FK60}
with $\varkappa_0$ replaced by some constant $\varkappa_0' >0$. Using the same argument as above, 
we conclude that \eqref{Inequality-FK60} still holds for $g \in \Gamma_{\mu}$, again with the constant $\varkappa_0'$. 
This ends the proof of the equivalence (2) $\Leftrightarrow$ (3).

(c) We finally prove that (2) $\Leftrightarrow$ (4). 
To do this, we first show that (4) implies (2).
Using $g \cdot \mathcal{S} \subseteq \mathcal{S}_{\epsilon}$ and the definition of $\mathcal{S}_{\epsilon}$, 
there exists $\epsilon \in (0, 1)$ such that 
 for any $g = ( g^{i,j} )_{ 1\leq i, j \leq d } \in \supp \mu$ 
and any $1\leq i, j  \leq d$, 
\begin{align} \label{equilem 01}
\epsilon \leq  \langle e_i, g \cdot e_j \rangle
=  \frac{ \langle e_i, g e_j \rangle }{\|g e_j\|} 
=  \frac{g^{i,j}}{ \sum_{i=1}^d \langle e_i, g e_j \rangle }
= \frac{g^{i,j}}{ \sum_{i=1}^d g^{i,j} }.
\end{align}
It follows that for any $g \in \supp \mu$, 
\begin{align}\label{g-S-epsilon-a}
g>0 \quad \mbox{ and } \quad  
\epsilon \leq \frac{ \min_{1 \leq i \leq d} g^{i,j} }{ \sum_{i=1}^d g^{i,j} }  \leq  \frac{ \min_{1 \leq i \leq d} g^{i,j} }{ \max_{1 \leq i \leq d} g^{i,j} } \quad \forall  j \in \{1, \ldots,  d\}. 
\end{align}
Applying this with $g^{{\mathrm T}}$ instead of $g$, we see that $g^{{\mathrm T}} \cdot \mathcal{S} \subseteq \mathcal{S}_{\epsilon}$
implies that for any $g \in \supp \mu$, 
\begin{align}\label{g-S-epsilon-b}
g>0 \quad \mbox{ and } \quad  
\epsilon \leq \frac{ \min_{1 \leq i \leq d} g^{j, i} }{ \sum_{i=1}^d g^{j,i} }  \leq  \frac{ \min_{1 \leq i \leq d} g^{j, i} }{ \max_{1 \leq i \leq d} g^{j, i} } \quad \forall  j \in \{1, \ldots,  d\}. 
\end{align}
With \eqref{g-S-epsilon-a} and \eqref{g-S-epsilon-b}, we conclude that \eqref{Inequality-FK60} holds for all $g \in \supp \mu$
with $\varkappa_0 = 1/ \epsilon$.

We then prove that (2) implies (4). 
For any $v = \sum_{j=1}^d v_j e_j \in \mathcal{S}$ with $v_j\geq 0$ and $\sum_{j=1}^d v_j =1$,
any $g = ( g^{i,j} )_{ 1\leq i, j \leq d } \in \bb G$ with $gv \neq 0$, and any $1 \leq i \leq d$, 
 it holds that 
\begin{align*}
\langle e_i, g\!\cdot\! v \rangle 
= \frac{1}{\|gv\|} \sum_{j=1}^d v_j \langle e_i, g e_j \rangle  
=  \frac{ \sum_{j=1}^d  v_j  g^{i,j}  }{ \sum_{i=1}^d  \sum_{j=1}^d  v_j   g^{i,j}   }.
\end{align*}
By (2), there exists a constant $\varkappa_0 >1$ such that for any $1 \leq i \leq d$,
\begin{align*}
\frac{ \sum_{j=1}^d  v_j  g^{i,j}  }{ \sum_{i=1}^d  \sum_{j=1}^d  v_j   g^{i,j}   }
\geq  \frac{ \sum_{j=1}^d  v_j  \min_{1 \leq i \leq d} g^{i,j}  }{ d  \sum_{j=1}^d  v_j  \max_{1 \leq i \leq d} g^{i,j}   }
\geq  \frac{1}{\varkappa_0 d}, 
\end{align*}
for every $g \in \supp \mu$, 
so that $g \cdot \mathcal{S} \subseteq \mathcal{S}_{\epsilon}$ holds with $\epsilon =  \frac{1}{\varkappa_0 d}$. 
In the same way, we can prove that $g^{{\mathrm T}} \cdot \mathcal{S} \subseteq \mathcal{S}_{\epsilon}$
holds for every $g \in \supp \mu$. 
\end{proof}

With the same argument, we can prove the following result on the weaker version of the Furstenberg-Kesten condition.

\begin{lemma}\label{Lem-equiv-weak-FK}
Assume that the matrix $0$ does not belong to $\Gamma_{\mu}$. 
Then, the following four conditions are equivalent: 

\begin{enumerate}

\item 
\eqref{Inequality-FK-weak}  holds for $\mu$-almost every $g \in \bb G$. 

\item 
\eqref{Inequality-FK-weak} holds for all $g \in \supp \mu$.

\item
\eqref{Inequality-FK-weak} holds for all $g \in \Gamma_{\mu}$.

\item
There exists a constant $\epsilon \in (0, 1)$ such that
\begin{align} \label{equicon A4}
g \cdot \mathcal{S} \subseteq \mathcal{S}_{\epsilon}
\quad \mbox{for $g \in \supp \mu$.}
\end{align}

\end{enumerate}
\end{lemma}

Since $\nu$ is the unique invariant probability measure of the Markov chain $(G_n \cdot v)$, 
by \eqref{def-invariant-mea} and Lemma \ref{lem equiv Kesten}, 
 under condition \ref{Condi-FK-weak} (or stronger condition \ref{Condi-FK}), 
it holds that $\nu (\mathcal S_{\epsilon}) = 1$. 
Since $\mathcal S_{\epsilon}$ is closed, this implies that
\begin{align}\label{inclusion-supp-mu}
\supp\nu  \subseteq \mathcal S_{\epsilon}.
\end{align}

Denote by $V(\Gamma_{\mu})$ the unique minimal closed $\Gamma_{\mu}$-invariant subset of $\mathcal{S}$. 
By \cite[Lemma 4.2]{BDGM14} and \ref{Condi-FK-weak}, it holds that $V(\Gamma_{\mu}) = \supp\nu$.

\subsection{Spectral gap properties of the perturbed operator $R_{it}$}\label{Sec-spec-gap-Rit}
In this section we study the perturbed operator $R_{it}$ under  polynomial moment conditions  only.
These properties will be important in proving the Berry-Esseen theorem and the first-order Edgeworth expansion. 

For any $t \in \bb R$, 
define the perturbed operator $R_{it}$ as follows: for $\varphi\in \mathcal{C(S)}$, 
\begin{align} \label{def R tu 01-it}
R_{it}\varphi(v) = \mathbb{E} \left[ e^{ it( \log \|g v\| - \lambda_{\mu} ) } \varphi(g \cdot v) \right],  
\quad v \in \mathcal{S}.  
\end{align}
When $t = 0$, we have $R_0 = P$, where $P$ is defined by \eqref{def-transition-oper-P}. 
Applying \eqref{cocycle02} and the induction argument, 
one can check that for any $n\geq 1$,
\begin{align*}
R_{it}^n \varphi(v) =\mathbb{E} \left[ e^{ it( \log \|G_n v\| - nq ) } \varphi(G_n \cdot v)  \right],
\quad v \in \mathcal{S}. 
\end{align*}
The following result gives the spectral gap properties of $R_{it}$ 
and a third-order expansion for the eigenvalue $\lambda_{it}$ as $t \to 0$, under the third moment condition. 
Recall that $m_3$ is defined by \eqref{Def-m3}.

\begin{lemma}  \label{transfer operator_Pit} 

\noindent 1.   
Assume \ref{Condi-FK-weak}.      
Then there exists a constant $\delta>0$ such that for any $t \in (-\delta, \delta)$ and $n \geq 1$, 
\begin{align}\label{Pzn-decom_Pit}
R_{it}^n = \lambda_{it}^n \Pi_{it} + N_{it}^n,
\end{align}
where $\Pi_{it}$ is a rank-one projection on $\scr B_{1}$ satisfying
$\Pi_0 \varphi(v) = \nu(\varphi)$
for any $\varphi \in \scr B_{1}$ and $v \in \mathcal S$.  
Moreover, $\Pi_{it} N_{it} = N_{it} \Pi_{it} =0$ and $\rho(N_{it}) < 1$ for all $t \in (-\delta, \delta)$. 

\noindent 2.   
Assume \ref{Condi-FK-weak} and $\int_{\bb G} \log^3 N(g) \mu(dg) < \infty$.       
Then $\lambda_{it} = 1 - \frac{ \sigma^2 }{2} t^2 +  \frac{m_3}{6} (it)^3 + o(t^3)$ as $t \to 0$,
where $\lambda_{it}$ satisfies $\frac{d \lambda_{it}}{dt} |_{t =0} = 0$ and 
\begin{align*}
&\sigma^2
= - \frac{d^2 \lambda_{it}}{dt^2} \Big|_{t =0} = \lim_{n \to \infty} \frac{1}{n} \int_{\mathcal S} \bb E  \big[ (\log \|G_n v\| - n \lambda_{\mu})^2 \big] \nu(dv),  \notag\\
& 
m_3
= i \frac{d^3 \lambda_{it}}{dt^3} \Big|_{t =0} 
= \lim_{n \to \infty} \frac{1}{n} \int_{\mathcal S} \big[ (\log \|G_n v\| - n \lambda_{\mu})^3 \big] \nu(dv). 
\end{align*}
\end{lemma}

\begin{proof}
The first part of the lemma is proved in \cite[Theorem 3.3]{HH08} 
 under weaker conditions 
which are implied  by \ref{Condi-FK-weak}.   

Now we give a proof of the second part. 
By \eqref{Pzn-decom_Pit}, we have $R_{it} \Pi_{it} = \lambda_{it} \Pi_{it}$. 
Taking derivative with respect to $t$ at $0$, we get 
\begin{align*}
\frac{d R_{it}}{dt} \big|_{t=0} 1 +  P  \frac{d \Pi_{it} }{dt} \big|_{t=0} 1 
= \frac{d \lambda_{it}}{dt} \big|_{t =0}  +  \frac{ d \Pi_{it} }{dt} \big|_{t =0} 1, 
\end{align*}
with $P$ defined by \eqref{def-transition-oper-P}.
Integrating with respect to $\nu$, using \eqref{formu-lambda-mu} and the fact that $\nu P = \nu$ (cf.\ \eqref{def-invariant-mea}), we obtain 
\begin{align}\label{derivative-lambda-0}
\frac{d \lambda_{it}}{dt} \big|_{t =0}  = i \left(  \int_{\mathcal S} \int_{\bb G}  \log \|g v\|  \mu(dg)  \nu(dv)  - \lambda_{\mu}  \right)= 0. 
\end{align}
Since $\Pi_{it}^2 = \Pi_{it}$, 
we have $2 \Pi_0  \frac{d \Pi_{it} }{dt} \big|_{t=0} = \frac{d \Pi_{it} }{dt} \big|_{t=0}$ so that $\nu(\frac{d \Pi_{it} }{dt} \big|_{t=0} 1) = 0$, by using the fact that $\nu \Pi_0 = \nu$. 
From \eqref{Pzn-decom_Pit} we have 
\begin{align*}
\bb E_{\nu}  e^{it (\log \|G_n v\| - n \lambda_{\mu}) } 
= \lambda_{it}^n \nu(\Pi_{it} 1) + \nu (N_{it}^n 1), 
\end{align*}
where, for brevity, we use the notation $\bb E_{\nu}$ to mean that the starting point $v$ has the law $\nu$. 
Therefore, by Taylor's formula, we get that for sufficiently small $t$, 
\begin{align*}
\bb E_{\nu}  e^{it (\log \|G_n v\| - n \lambda_{\mu} ) } 
& =  1 +  it \bb E_{\nu}  (\log \|G_n v\| - n \lambda_{\mu} ) 
  - \frac{t^2}{2}  \bb E_{\nu}  \big[ (\log \|G_n v\| - n \lambda_{\mu} )^2 \big]  \notag\\
  & \quad  -  i  \frac{t^3}{6}  \bb E_{\nu}  \big[ (\log \|G_n v\| - n \lambda_{\mu} )^3 \big]  +  o(t^3),   \notag\\
\lambda_{it}^n & =  1 + n \frac{d^2 \lambda_{it}}{dt^2} \Big|_{t =0} \frac{t^2}{2} + n \frac{d^3 \lambda_{it}}{dt^3} \Big|_{t =0} \frac{t^3}{6} +  n \, o(t^3),   \notag\\
\nu(\Pi_{it} 1) & =  1  +  c_1  t^2  + c_2 t^3 +  o(t^3). 
\end{align*}
It follows that 
\begin{align*}
\bb E_{\nu}  \big[ (\log \|G_n v\| - n \lambda_{\mu} )^2 \big]  & = - n \frac{d^2 \lambda_{it}}{dt^2} \Big|_{t =0}  - c_1 + O(a^n),   \notag\\
\bb E_{\nu}  \big[ (\log \|G_n v\| - n \lambda_{\mu} )^3 \big]  &  = i  n \frac{d^3 \lambda_{it}}{dt^3} \Big|_{t =0} + 6 i c_2  + O(a^n), 
\end{align*}
which completes the proof of the lemma. 
\end{proof}

The following result shows that both functions $b$ and $d$ defined in \eqref{drift-b001}
are Lipschitz continuous on $\mathcal{S}$. 

\begin{lemma}\label{Lem-B0}
\noindent 1.  
Assume \ref{Condi-FK-weak} and $\int_{\bb G} \log N(g) \mu(dg) < \infty$.  
Then $b(\cdot) \in \scr B_{1}$ and
\begin{align} \label{Def2-b0}
b(v) = -i \left. \frac{ d \Pi_{it} }{ dt } \right|_{t=0} 1(v),  \quad  v \in \mathcal{S}. 
\end{align}

\noindent 2.  
Assume \ref{Condi-FK-weak}. 
Then $d(\cdot) \in \scr B_{1}$. 
\end{lemma}

\begin{proof}
We start by proving the first part of the lemma. 
By Lemma \ref{transfer operator_Pit}, we have that for any $t \in \bb R$ and $n \geq 1$,  
\begin{align*}
\mathbb{E} \left[ e^{it (\log \|G_n v\| - n \lambda_{\mu} ) }  \right]
=  \lambda_{it}^n \Pi_{it}  1(v) + N^{n}_{it} 1(v),  \    v \in \mathcal{S}.
\end{align*}
Since $\lambda_{0} = 1$ and $\frac{d \lambda_{it}}{dt} \big|_{t =0} =0$ (by \eqref{derivative-lambda-0}), 
differentiating both sides of the above equation with respect to $t = 0$ gives that for any $v \in \mathcal{S}$, 
\begin{align}\label{b0-01}
i \mathbb{E}  (\log \|G_n v\| - n \lambda_{\mu} ) 
 = \left. \frac{ d \Pi_{it} }{ dt } \right|_{t=0} 1(v)
   + \left.  \frac{ d N^{n}_{it} }{ dt } \right|_{t=0} 1(v).
\end{align}
Using again Lemma \ref{transfer operator_Pit},
we see that the first term on the right-hand side of \eqref{b0-01}
belongs to $\scr B_{1}$, and the second term converges to $0$ exponentially fast as $n \to \infty$.
Letting $n \to \infty$ in \eqref{b0-01}, we obtain \eqref{Def2-b0}.

Now we show that second part of the lemma. 
For any $u \in \mathcal{S}_{\epsilon}$, the function $\log \langle \cdot, u \rangle$ is Lipschitz continuous on $\mathcal{S}$. 
By \eqref{inclusion-supp-mu}, we have $\supp \nu \subseteq \mathcal{S}_{\epsilon}$
and hence $d(\cdot)$ is Lipschitz continuous on $\mathcal{S}$. 
\end{proof}

\subsection{Spectral gap properties of the transfer operator $P_s$}\label{Sect-Spec-norm-cocycle}

The following result gives the existence of the eigenfunctions and eigenmeasures
of the transfer operators $P_s$ and $P_s^*$ under exponential moments.
The uniqueness will be shown separately in Lemma \ref{unique r su et nu su}.

For $\epsilon \in (0,1)$ being from Lemma \ref{lem equiv Kesten},
let $\mathcal{C}(\mathcal{S}_\epsilon)$ be the space of complex-valued continuous functions on $\mathcal{S}_\epsilon$,
and let $\mathcal{P(S)}$ (resp. $\mathcal{P(S_\epsilon)}$) be the set of Borel probability measures on $\mathcal{S}$ (resp. $S_\epsilon$). 
Recall that $I_{\mu} = I_{\mu}^+ \cup I_{\mu}^-$.

\begin{proposition} \label{change-measure-neg01}
Assume either \ref{Condi-FK} when $s \in I_{\mu}^-$, or \ref{Condi-FK-weak}  when $s \in I_{\mu}^+$. 
Let  $\kappa(s)=\rho(P_s)$, $s \in  I_{\mu}$.
Then, there exist a constant $\epsilon \in (0,1)$,  an  eigenfunction $r_{s} \in \mathcal{C}(\mathcal{S}_\epsilon)$ 
and an eigenmeasure $\nu_s \in \mathcal{P(S_\epsilon)}$ 
with $V(\Gamma_\mu)\subseteq \supp \nu_{s}$ such that 
\begin{align*}
P_s r_s = \kappa(s) r_s \quad \mbox{and} \quad P_s \nu_s = \kappa(s) \nu_s. 
\end{align*}
Besides,  $\rho(P_s^*) = \kappa(s)$ and there exist a constant $\epsilon \in (0,1)$, an 
eigenfunction $r_s^* \in \mathcal{C}(\mathcal{S}_\epsilon)$  
and an eigenmeasure $\nu_s^* \in \mathcal{P(S_\epsilon)}$ such that 
\begin{align*}
P_s^* r_{s}^* = \kappa(s) r_{s}^* \quad \mbox{and}  \quad P_s^*  \nu_{s}^* = \kappa(s) \nu_{s}^*. 
\end{align*}
Moreover, the eigenfunctions $r_{s}$ and $r_{s}^*$ are given by 
\begin{align} \label{Def-rs-neg}
r_{s}(v)   = \int_{\mathcal{S}}  \frac{1}{ \langle v, u \rangle^{-s} } \nu^*_{s}(du),   \quad
r_{s}^*(v)   = \int_{\mathcal{S}}  \frac{1}{ \langle v, u \rangle^{-s} } \nu_{s}(du),  \quad v \in \mathcal{S},  
\end{align}
which are strictly positive and $\bar{s}$-H\"{o}lder continuous
with respect to the distance $\mathbf{d}$, with $\bar{s}=\min\{1, |s|\}$.
In addition, there exist constants $c_s, c_s' >0$ such that for any $n \geq 1$, 
\begin{align}\label{Inequa-kappa-s}
 c_s \mathbb{E}  (\| G_n \|^s)  \leq \kappa(s)^n  \leq  c_s' \mathbb{E} (\| G_n \|^s). 
\end{align}  
In particular, we have $\kappa(s) = \lim_{n\to\infty}\left(\mathbb{E}\| G_n \|^{s}\right)^{\frac{1}{n}}$. 
\end{proposition}

For $s \in I_{\mu}^+$, Proposition \ref{change-measure-neg01} has been established in \cite{BDGM14}
under weaker assumptions than those of \ref{Condi-FK}; 
for $s \in I_{\mu}^-$, Proposition \ref{change-measure-neg01} is new. 
 It is noteworthy that condition \ref{Condi-FK} enables us to address cases with arbitrary $s<0$.
This is because, under condition \ref{Condi-FK}, both the eigenmeasure $\nu_s$ and $\nu_s^*$
are supported on $\mathcal{S}_\epsilon$ for any $s<0$.
As a result, we can establish the spectral gap properties for the transfer operators 
$P_s$ and $P_s^*$ in this range.
We include the case $s \in I_{\mu}^+$ here,
as it will be utilized in Section \ref{sec:spec gap entries} to develop the spectral gap theory for the new cocycle
$\sigma_f(G_n, v) = \log \frac{ \langle f, G_n v \rangle }{ \langle f, v \rangle }$.

The proof of Proposition \ref{change-measure-neg01}
is preceded by the following technical lemma,
which plays a crucial role in obtaining a lower bound for the spectral radius of $P_s^*$.

\begin{lemma}\label{lemma kappa 1}
Assume \ref{Condi-FK}. 
Then there exists a constant $c>0$ such that for any 
 $n \geq 1$, $\mu^{* n}$-almost every $g = (g^{i,j})_{1 \leq i, j \leq d} \in \bb G$, 
and any $v \in \mathcal{S}$, 
\begin{align*}
c \|g\| \leq  \| gv \| \leq  \|g\|. 
\end{align*}
\end{lemma}

\begin{proof}
By definition, we have that for any $g\in \bb G$ and $v \in \mathcal{S}$, 
\begin{align*}
\| gv \| = \langle gv, \mathbf 1 \rangle  \leq  \langle g \mathbf 1, \mathbf 1 \rangle = \|g\|. 
\end{align*}
On the other hand, by Lemma \ref{Lem-FK60}, there exists a constant $\varkappa'_0 >1$ 
such that for any 
$n \geq 1$, $\mu^{* n}$-almost every $g = (g^{i,j})_{1 \leq i, j \leq d} \in \bb G$, 
 and any $v = (v_1, \ldots, v_d) \in \mathcal{S}$, 
\begin{align*}
\| gv \| 
= \sum_{i = 1}^d \sum_{j = 1}^d  g^{i,j}  v_j
 \geq \frac{1}{\varkappa'_0}  \sum_{i = 1}^d \sum_{j = 1}^d  g^{i,i} v_j
=  \frac{1}{\varkappa'_0}  \sum_{i = 1}^d  g^{i,i}  
 \geq  \frac{1}{ (\varkappa'_0)^2 d}  \sum_{i = 1}^d \sum_{j = 1}^d  g^{i,j}  = \frac{1}{ (\varkappa'_0)^2 d}  \|g\|. 
\end{align*}
This completes the proof. 
\end{proof}

A matrix is called allowable if every row and every column contains at least one strictly positive element. 
Let $\bb G_{0}$ be the multiplicative semigroup of $d \times d$ allowable matrices.
The following result is proved in \cite[Lemma 3.1]{HH08}. 

\begin{lemma}[\cite{HH08}] \label{Lem-Contin-Cocycle}
Assume \ref{Condi-FK-weak}.
Then, for any $g\in \bb G_0$ and $v_1, v_2 \in \mathcal{S}$ with $\mathbf{d}(v_1, v_2) < 1$, 
\begin{align*}
\left|  \log\|g v_1\| - \log\|g v_2\| \right| 
\leq  2 \log \frac{1}{1 - \mathbf{d}(v_1, v_2)}. 
\end{align*}
In particular, for any $a \in (0,1)$, there exists a constant $c>0$ 
such that for any $g\in \bb G_0$ and $v_1, v_2 \in \mathcal{S}$ with $\mathbf{d}(v_1, v_2) < a$, 
\begin{align*}
\left|  \log\|g v_1\| - \log\|g v_2\| \right| 
\leq  c \mathbf{d}(v_1, v_2). 
\end{align*}
\end{lemma}
By Lemma \ref{Lem-equiv-weak-FK}, condition \ref{Condi-FK-weak} implies the allowability and  positivity conditions,
so that Lemma \ref{Lem-Contin-Cocycle} holds true, according to \cite[Lemma 3.1]{HH08}.

\begin{proof} [Proof of Proposition \ref{change-measure-neg01}] 
We provide a proof of Proposition \ref{change-measure-neg01} only for $s \in I_{\mu}^-$,
as the corresponding result for $s \in I_{\mu}^+$ has been established in \cite{BDGM14}.
The case of negative $s \in I_{\mu}^-$ encounters additional difficulties compared with the case of positive $s \in I_{\mu}^+$.

\textit{Step 1.} We prove the existence of the eigenmeasures $\nu_{s}, \nu_{s}^*$ 
and of the eigenfunctions $r_{s}, r_{s}^*$. 
For any $s\in I_{\mu}^-$ and $\eta \in \mathcal{P(S)}$, 
define a non-linear operator $\hat{P}_s^*$ on $\mathcal{P(S)}$ 
by $\hat{P}_s^* \eta = \frac{P_s^* \eta}{(P_s^* \eta)(1)}$, 
where 
for any $\varphi \in \mathcal{C(S)}$, 
\begin{align*}
( P_s^* \eta)\varphi = \eta(P_s^* \varphi)
= \int_{\mathcal{S}}  \int_{\bb G}
\| g^{\mathrm T} v \|^s  \varphi(g^{\mathrm T}\!\cdot\! v)\mu(dg)\eta(dv). 
\end{align*}
By \eqref{equicon A4-bb} of Lemma \ref{lem equiv Kesten},  
 under condition \ref{Condi-FK}, 
 there exists a constant $\epsilon >0$ such that
  $g^{\mathrm T} \cdot  v \in \mathcal{S}_{\epsilon}$ for any $v \in \mathcal{S}$
  and  $\mu$-almost every $g \in \bb G$, 
which implies that $\hat{P}_s^* \eta$ is a probability measure on $\mathcal{S}_\epsilon$ for any $\eta \in \mathcal{P(S_\epsilon)}$. 
Since the operator $\hat{P}_s^*$ is weakly continuous on $\mathcal{P(S_\epsilon)}$,
and the space $\mathcal{P(S_\epsilon)}$ is compact and convex, 
by the Schauder-Tychonoff theorem, there exists
$\nu_{s}^*\in \mathcal{P(S_\epsilon)}$ such that $\hat{P}_s^* \nu_{s}^* = \nu_{s}^*$, 
so that 
$P_s^* \nu_{s}^* = k_{s} \nu_{s}^*$
with $k_{s} = (P_s^* \nu_{s}^*)(1) > 0$. 
Now we check that 
$r_s$ defined by \eqref{Def-rs-neg} satisfies $P_s r_s = k_s r_s$.   
By the definition of $P_s$ and $r_s$, we have
\begin{align} \label{Scalequa 01}
P_s r_{s}(v) 
& =  \int_{\bb G} \frac{1}{\|gv\|^{-s}}
\left(  \int_{\mathcal{S}}  \frac{1}{ \langle g \!\cdot\! v, u \rangle^{-s} } \nu^*_{s}(du) \right)\mu(dg)   \notag\\
& =  \int_{\bb G} 
\left(  \int_{\mathcal{S}}  \frac{1}{ \langle  v, g^{\mathrm{T} } u \rangle^{-s} } \nu^*_{s}(du) \right)\mu(dg) \notag\\
& =  \int_{\mathcal{S}}  \int_{\bb G}
  \frac{1}{\|g^{\mathrm{T} } u\|^{-s}}   \frac{1}{ \langle v,  g^{\mathrm{T} } \!\cdot\! u \rangle^{-s} }
\mu(dg)\nu_{s}^*(du). 
\end{align}
Letting $\varphi_{v}(u) = \frac{1}{\langle v, u\rangle^{-s} }$ 
and using $P_s^* \nu_{s}^* = k_{s} \nu_{s}^*$, 
from \eqref{Scalequa 01} 
it follows that
\begin{align*}
P_s r_{s}(v) 
= (\nu_s^* P_s^*) \varphi_{v} 
= k_{s} \int_{\mathcal{S}}  \frac{1}{\langle v, u\rangle^{-s} }  \nu_{s}^*(du)
= k_{s} r_{s}(v). 
\end{align*}
Similarly one can verify that there exists $\nu_{s} \in \mathcal{P(S_\epsilon)}$ 
satisfying $P_s \nu_{s} = k_{s}^* \nu_{s}$ 
with $k_{s}^* = (P_s \nu_{s})(1) > 0$
and that  $r_s^*$ defined by \eqref{Def-rs-neg}
satisfies $P_s^* r_{s}^* = k_{s}^*  r_{s}^*$.

\textit{Step 2.}
We prove $V(\Gamma_{\mu})\subseteq \supp \nu_{s}$, 
the positivity and the H\"{o}lder continuity of $r_{s}$.
From $P_s \nu_{s} = k_{s}^* \nu_{s}$, 
we have that for any $\varphi \in \mathcal{C(S)}$,
\begin{align*} 
\int_{\mathcal{S}}
\int_{\bb G}  \frac{1}{\|gv\|^{-s}}  \varphi(g\!\cdot\! v) \mu(dg) \nu_{s}(dv)
= k_{s}^* \int_{\mathcal{S}} \varphi(x) \nu_{s}(dx),
\end{align*}
which implies that $\Gamma_{\mu} (\supp \nu_{s})\subseteq \supp \nu_{s}$. 
Since $V(\Gamma_{\mu})$ is the unique minimal closed $\Gamma_{\mu}$-invariant subset of $\mathcal{S}$,
we obtain that $V(\Gamma_{\mu})\subseteq \supp \nu_{s}$. 

Since $\supp \nu_{s}^* \subseteq \mathcal{S}_{\epsilon}$, 
it holds that $\inf_{v \in \mathcal{S}} r_{s}(v) > 0$. 
Now we show that $r_{s}$ is $\bar{s}$-H\"{o}lder continuous with respect to $\mathbf{d}$. 
By the definition of $r_{s}$ and the fact that $\langle v, u\rangle\geq \varepsilon$, 
for any $v \in \mathcal{S}$ and $u \in \supp \nu_{s}^*$, we have that for any $v_1, v_2 \in \mathcal{S}$,
\begin{align}
 |r_{s}(v_1) - r_{s}(v_2)| 
& \leq  
\int_{\mathcal{S}}  \left|\frac{1}{\langle  v_1, u \rangle^{-s} }-\frac{1}{\langle v_2, u \rangle^{-s} } \right| \nu_{s}^*(du)
\notag\\
& \leq c_s  \int_{\mathcal{S}}  | \langle v_1, u\rangle^{-s} - \langle v_2, u\rangle^{-s} |  \nu_{s}^*(du). 
\end{align}
For $s \in [-1, 0)$, it holds $| \langle v_1, u\rangle^{-s} - \langle v_2, u\rangle^{-s} |  \leq \|v_1 - v_2\|^{-s}$. 
For $s < -1$, it holds
$| \langle v_1, u\rangle^{-s} - \langle v_2, u\rangle^{-s} |  \leq c_s \| v_1 - v_2\|^{\bar{s}}$. 
Therefore, using $\|v_1 - v_2\| \leq c \mathbf d(v_1, v_2)$, we get
\begin{align*}
|r_{s}(v_1) - r_{s}(v_2)| \leq c'_s \| v_1 - v_2\|^{\bar{s}} \leq  c''_s  \mathbf d(v_1, v_2)^{\bar{s}}.
\end{align*}
Similarly, one can show that $\supp \nu_{s} \subseteq \mathcal{S}_{\epsilon}$
and that $r_{s}^*$ is strictly positive and $\bar{s}$-H\"{o}lder continuous with respect to $\mathbf{d}$.

\textit{Step 3.}
We prove $k_{s} = \kappa(s) = \rho(P_s)$. 
Using Lemma \ref{lemma kappa 1}, we get
\begin{align*}
\sup_{v \in \mathcal{S} } | P_s^n \varphi(v)|
& =  \sup_{v \in \mathcal{S} }  \left| \int \frac{1}{ \| g_n\ldots g_1 v \|^{-s} }
       \varphi[(g_n\ldots g_1)\!\cdot\! v] \mu(dg_1)\ldots \mu(dg_n)  \right|    \notag\\
& \leq  c_s \|\varphi\|_\infty \mathbb{E}\left(\|g_n \ldots g_1\|^s\right).
\end{align*}
This implies $k_{s} \leq \rho(P_s) \leq  \lim_{n\to\infty}\left(\mathbb{E}\| G_n \|^{s}\right)^{\frac{1}{n}}.$
On the other hand, 
using $P_s^* \nu_{s}^* = k_{s} \nu_{s}^*$ and Lemma \ref{lemma kappa 1}, 
we obtain
\begin{align} 
k_{s}^n = \nu_{s}^*( (P_s^*)^n 1)
 =  \mathbb{E} \int_{\mathcal{S}}
\frac{1}{\| g_n^{\mathrm T}\ldots g_1^{\mathrm T} v \|^{-s} }  \nu_{s}^*(dv)  
 \geq  c_s \mathbb{E}\|g_1 \ldots g_n\|^s
= c_s \mathbb{E}\|g_n \ldots g_1\|^s,  \nonumber
\end{align}
which implies $k_{s} \geq \lim_{n\to\infty}\left(\mathbb{E} \| G_n \|^{s}\right)^{\frac{1}{n}}.$ 
Therefore, 
$$
k_{s}=\kappa(s) = \rho(P_s) = \lim_{n\to\infty}\left(\mathbb{E} \| G_n \|^{s}\right)^{\frac{1}{n}}.
$$
Moreover,  
\begin{align*}
\rho(P_s^*) = \lim_{n\to\infty}\left(\mathbb{E}\| g_n^{\mathrm T}\ldots g_1^{\mathrm T}  \|^{s}\right)^{\frac{1}{n}} = 
\lim_{n\to\infty}\left(\mathbb{E}\| g_1\ldots g_n  \|^{s}\right)^{\frac{1}{n}}  =\kappa(s). 
\end{align*}
This concludes the proof of Proposition \ref{change-measure-neg01}. 
\end{proof}

 For any $s \in I_{\mu}$, 
define the operator $Q_s$ as follows: for any $\varphi\in \mathcal{C(S)}$,
\begin{align}\label{defi Q_s}  
Q_s \varphi(v): =  \frac{1}{\kappa(s)r_{s}(v)} P_s(\varphi r_{s})(v), \quad v \in \mathcal{S},
\end{align}
which is a Markov operator satisfying $Q_s 1 = 1$ and $Q_s \varphi \geq 0$ for $\varphi \geq 0$. 
By iteration, we get for any $v \in \mathcal{S}$ and $n \geq 1$, 
\begin{align*}
Q_s^n\varphi(v)
= \frac{1}{\kappa(s)^n r_{s}(v)} P_s^n(\varphi r_{s})(v)
= \int \varphi(G_n \!\cdot\! v) q_{n}^{s}(G_n, v)  \mu(dg_1)\ldots \mu(dg_n),
\end{align*}
where
\begin{align}\label{def-qnsGn}
q_{n}^{s}(G_n, v)
=\frac{1}{\kappa(s)^{n} \|G_n v \|^{-s} }
\frac{r_{s} (G_n \!\cdot\! v )}{r_{s}(v)}.
\end{align}
It holds that for any $n,m\geq 1$, 
\begin{align} \label{cocycle02}
q_{n}^{s}(G_n, v)
q_{m}^{s} ( g_{n+m} \ldots g_{n+1}, G_n \!\cdot\! v )
=q_{n+m}^{s}(G_{n+m}, v).
\end{align}
For any $v \in\mathcal{S}$, 
the sequence of probability measures
\begin{align*}
\bb Q_{s,n}^v(dg_1, \dots, dg_n): = q_{n}^{s}(G_n, v)\mu(dg_1) \dots \mu(dg_n)
\end{align*}
forms a projective system on $M(d,\mathbb{R})^{\mathbb{N}}$. 
By the Kolmogorov extension theorem,  
there exists a unique probability
measure $\mathbb{Q}_{s}^{v}$ on $M(d,\mathbb{R})^{\mathbb{N}}$ with marginals $\bb Q_{s,n}^v$. 
The corresponding expectation is denoted by $\mathbb{E}_{\mathbb{Q}_{s}^{v}}$. 
The following change of measure formula holds: 
for any bounded measurable function $h$ on $(\mathcal{S} \times \mathbb R)^n$ and $n\geq 1$, 
\begin{align} \label{change measure equ coeff}
&\frac{1}{\kappa(s)^{n} r_{s}(v)}
\mathbb{E} \Big[  r_{s}(G_n \!\cdot\! v) \| G_n v\|^s   h\big( G_1 \!\cdot\! v,  \log \|G_1 v\|, \ldots,  G_n \!\cdot\! v,  \log \|G_n v\|  \big)  \Big]  \nonumber \\
&\qquad\qquad\quad
=\mathbb{E}_{\mathbb{Q}_{s}^{v}} \Big[ h\big( G_1 \!\cdot\! v,  \log \|G_1 v\|, \ldots,  G_n \!\cdot\! v,  \log \|G_n v\|  \big) \Big].
\end{align}
We next show that the kernel $q_{n}^{s}(G_n, \cdot)$ is $\bar{s}$-H\"{o}lder continuous 
with respect to $\mathbf{d}$.

\begin{lemma}\label{kernel inequa coeff 01}
Assume either \ref{Condi-FK} when $s \in I_{\mu}^-$, or \ref{Condi-FK-weak}  when $s \in I_{\mu}^+$.
Then, for any $s \in I_{\mu}$,  there exists a constant $c_s>0$ such that for any $n\geq 1$
and $v_1, v_2\in \mathcal{S}$, 
\begin{align*}
|q_{n}^{s}(G_n, v_1 ) - q_{n}^{s}(G_n, v_2 )|
\leq  \frac{ c_{s} }{\kappa(s)^n  \| G_n \|^{-s} } \mathbf{d}(v_1, v_2)^{\bar{s}}. 
\end{align*}
\end{lemma}

\begin{proof}
We only give a proof of Lemma \ref{kernel inequa coeff 01} for $s \in I_{\mu}^-$,
since the corresponding result for $s \in I_{\mu}^+$ has been established in \cite{BDGM14}.

Since, by Proposition \ref{change of measure coeffi 01}, 
$r_{s}$ is $\bar{s}$-H\"{o}lder continuous 
with respect to the distance $\mathbf{d}$, 
from \eqref{Contracting inequa} it follows that 
\begin{align*} 
|r_{s}(G_n \!\cdot\! v_1) - r_{s}(G_n \!\cdot\! v_2)|
\leq c_s \mathbf{d}(G_n \!\cdot\! v_1, G_n \!\cdot\! v_2)^{\bar{s}}
\leq c_s  a^{\bar s n} \mathbf{d}(v_1, v_2)^{\bar{s}}.
\end{align*}
Using again the $\bar{s}$-H\"{o}lder continuity and the positivity of $r_{s}$ 
in Proposition \ref{change of measure coeffi 01},
we get 
\begin{align} \label{r su holder 1}
 \left| \frac{r_{s}(G_n \!\cdot\! v_1)}{r_{s}(v_1)}- \frac{r_{s}(G_n \!\cdot\! v_2)}{r_{s}(v_2)} \right| 
& \leq  
r_{s}(G_n \!\cdot\! v_1)\left| \frac{1}{r_{s}(v_1)}-\frac{1}{r_{s}(v_2)} \right|  
+
 \frac{1}{r_{s}(v_2)}| r_{s}(G_n \!\cdot\! v_1) - r_{s}(G_n \!\cdot\! v_2) | 
\notag\\
& \leq  c_s \mathbf{d}(v_1, v_2)^{\bar{s}} +  
c'_s  a^{\bar s n} \mathbf{d}(v_1, v_2)^{\bar{s}}
\leq  c''_s  \mathbf{d}(v_1, v_2)^{\bar{s}}. 
\end{align}
If $0 \leq -s \leq 1$, using $|x^{-s} - y^{-s}| \leq |x- y|^{-s}$ for $x, y \geq 0$, 
and \eqref{DistanPosi}, we have 
\begin{align*}
\left|  \|G_n v_1\|^{-s} - \|G_n v_2\|^{-s}  \right|  
& \leq   \left|  \|G_n v_1\| - \|G_n v_2\| \right|^{-s}  
\leq  \|G_n\|^{-s}  \|v_1 - v_2\|^{-s}    \notag\\
& \leq  c_s  \|G_n\|^{-s} \mathbf{d}(v_1, v_2)^{-s}. 
\end{align*}
If $-s >1$, using $|x^{-s} - y^{-s}| \leq (-s) \max \{ x, y\}^{-s-1} |x- y|$ for $x, y \geq 0$, 
Lemma \ref{lemma kappa 1} and again \eqref{DistanPosi}, we have 
\begin{align*}
\left|  \|G_n v_1\|^{-s} - \|G_n v_2\|^{-s}  \right|  
& \leq  c_s \|G_n\|^{-s-1}  \left|  \|G_n v_1\| - \|G_n v_2\| \right|    \notag\\ 
&  \leq  c'_s \|G_n\|^{-s}  \|v_1 - v_2\|   
 \leq  c''_s  \|G_n\|^{-s} \mathbf{d}(v_1, v_2). 
\end{align*}
Hence, for any $s<0$, we get that with $\bar{s}=\min\{1, -s\}$, 
\begin{align*}
\left|  \|G_n v_1\|^{-s} - \|G_n v_2\|^{-s}  \right|  
\leq  c_s \|G_n\|^{-s} \mathbf{d}(v_1, v_2)^{\bar{s}}, 
\end{align*}
Using this and Lemma \ref{lemma kappa 1}, we obtain
\begin{align*} 
 \left| \frac{1}{\|G_n v_1\|^{-s}} - \frac{1}{\|G_n v_2\|^{-s}} \right|  
 =  \frac{ \left|  \|G_n v_1\|^{-s} - \|G_n v_2\|^{-s}  \right|  }{ \|G_n v_1\|^{-s}  \|G_n v_2\|^{-s} }
\leq \frac{c_s}{ \|G_n\|^{-s} }  \mathbf{d}(v_1, v_2)^{\bar{s}}. 
\end{align*}
This, together with 
\eqref{r su holder 1} and Lemma \ref{lemma kappa 1},
concludes the proof of the lemma. 
\end{proof}

The following result is the Doeblin-Fortet inequality for the Markov operator $Q_s$.

\begin{lemma} \label{equiconti Q_s}
Assume either \ref{Condi-FK} when $s \in I_{\mu}^-$, or \ref{Condi-FK-weak}  when $s \in I_{\mu}^+$.
Then, for any $s \in I_{\mu}$, 
there exist constants $c_{s}>0$ and $0<a<1$ such that 
for any $\bar{s}$-H\"older continuous function $\varphi$ on $\mathcal{S}$ and $n\geq1$, 
\begin{align}\label{Equi-contin-Inequ}
[Q_s^n \varphi]_{\bar{s}}
\leq  c_{s} \|\varphi\|_\infty + a^{\bar s n} [\varphi]_{\bar{s}}.
\end{align} 
In particular, for any $\varphi \in \mathcal{C(S)}$, the sequence $\{ Q_s^n \varphi \}_{n \geq 1}$
is equicontinuous on $\mathcal{S}$. 
\end{lemma}

\begin{proof}
By the definition of $Q_s$, we write
\begin{align} \label{equiconti Q su I}
  |Q_s^n \varphi(v_1) - Q_s^n \varphi(v_2)|  
& \leq   \left|  \int \varphi( G_n \!\cdot\! v_1) 
\left[ q_{n}^{s}(G_n, v_1)- q_{n}^{s}(G_n, v_2)   \right] \mu(dg_1) \ldots \mu(dg_n)
\right|  \notag\\
& \quad  +  
\left| \int \left[ \varphi( G_n\! \cdot\! v_1)- \varphi( G_n\! \cdot\! v_2)  \right] 
q_{n}^{s}(G_n, v_2) \mu(dg_1) \ldots \mu(dg_n)
 \right|  \notag\\
& =: I_1+I_2. 
\end{align}
For $I_1$, using Lemma \ref{kernel inequa coeff 01} and \eqref{Inequa-kappa-s} gives
\begin{align}  \label{equiconti Q su I 1}
I_1  & \leq 
\|\varphi\|_{\infty} 
\mathbb{E} \left|q_{n}^{s}(G_n, v_1) - q_{n}^{s}(G_n, v_2) \right| \notag\\
&  \leq  
c_{s}\|\varphi\|_{\infty}  \frac{ \mathbb{E} (\|G_n\|^{s}) }{\kappa(s)^n} \mathbf{d}(v_1,v_2)^{\bar{s}} 
\leq  c'_{s}\|\varphi\|_{\infty} \mathbf{d}(v_1,v_2)^{\bar{s}}.
\end{align}
For $I_2$, 
from \eqref{Contracting inequa}, we have 
$\mathbf{d}(G_n \!\cdot\! v_1, G_n \!\cdot\! v_2)\leq a^n \mathbf{d}(v_1, v_2)$ for some constant $0 < a <1$, 
and hence
\begin{align} \label{equiconti Q su I 2 posi}
I_2
\leq [\varphi]_{\bar{s}} a^{\bar s n }  \mathbf{d}(v_1, v_2)^{\bar{s}}.
\end{align}
Combining \eqref{equiconti Q su I}, \eqref{equiconti Q su I 1} and \eqref{equiconti Q su I 2 posi}
concludes the proof of \eqref{Equi-contin-Inequ}.
Since each $\varphi\in \mathcal{C(S)}$ 
can be approximated by a sequence of $\bar{s}$-H\"{o}lder continuous functions,
the sequence $\{ Q_s^n \varphi \}_{n \geq 1}$
is equicontinuous on $\mathcal{S}$. 
\end{proof}

The following result shows that 
the Markov operator $Q_s$ is irreducible and aperiodic.

\begin{lemma} \label{irredu Q u 01}
Assume either \ref{Condi-FK} when $s \in I_{\mu}^-$, or \ref{Condi-FK-weak}  when $s \in I_{\mu}^+$.
Let $s \in I_{\mu}$.  
If there exist 
$\varphi\in \mathcal{C(S)}$ and $\theta\in \mathbb{R}$ satisfying 
$Q_s \varphi = e^{i\theta} \varphi$,
then necessarily  $e^{i\theta}=1$ and  $\varphi$ is a constant on $\mathcal{S}$.
\end{lemma}

The proof of this lemma can be carried out in the same way as in \cite[Theorem 2.6]{GL16}, and therefore will be omitted.

By \cite[Theorem 6]{Ros1964}, 
for any equicontinuous Markov operator $Q$ on $\mathcal{C(S)}$, if
the equation $Q \varphi=e^{i\theta} \varphi$, $\varphi\in \mathcal{C(S)}$
implies $e^{i\theta}=1$ and $\varphi$ is a constant,
then 
$Q^n \varphi$ converges 
to $\pi(\varphi)$ as $n\to \infty$, 
uniformly in $\varphi \in \mathcal{C(S)}$,  
where $\pi$ is the unique invariant measure of $Q$.
Therefore, by Lemmas \ref{equiconti Q_s} and \ref{irredu Q u 01},
for any $s \in I_{\mu}$,
there exists a unique invariant probability measure $\pi_s$ on $\mathcal{S}$
such that uniformly in $\varphi\in \mathcal{C(S)}$,
\begin{align} \label{limit Q_s 01}
\lim_{n\to\infty}Q_s^n\varphi=\pi_s(\varphi).  
\end{align}
Since $Q_s \pi_s = \pi_s$ and $\pi_s$ is the unique invariant measure of $Q_s$, we get that for $\varphi\in \mathcal{C(S)}$, 
\begin{align} \label{limit Q_s 02}
\pi_s(\varphi)=\frac{\nu_{s}(\varphi r_{s})}{\nu_{s}(r_{s})}. 
\end{align}

We next show the uniqueness and continuity of the eigenfunction $r_{s}$
and the eigenmeasure $\nu_{s}$.

\begin{lemma} \label{unique r su et nu su}  
Assume either \ref{Condi-FK} when $s \in I_{\mu}^-$, or \ref{Condi-FK-weak}  when $s \in I_{\mu}^+$.  \\
1.   
For any $s \in I_{\mu}$, 
the eigenfunction $r_{s}$ is unique up to a scaling constant  
and the eigenmeasure $\nu_{s}\in \mathcal{P(S_\epsilon)}$ is unique. 
In addition, $\supp \nu_{s}=V(\Gamma_{\mu})$.  \\
2.  
The mappings $s\mapsto r_{s}$ and $s\mapsto \nu_{s}$ 
are continuous on $I_{\mu}$ with respect to the uniform topology and weak convergence topology, respectively. 
\end{lemma}

Based on Proposition \ref{change-measure-neg01} and Lemma \ref{irredu Q u 01}, 
the proof of the first assertion of Lemma \ref{unique r su et nu su}  
can be carried out in the same way as that of \cite[Corollary 4.12]{BDGM14}. 
Following the proof of \cite[Proposition 4.13]{BDGM14}, 
 using the first assertion and the fact that $\supp \nu_{s} \subseteq \mathcal S_{\epsilon}$,
one can obtain the second assertion of Lemma \ref{unique r su et nu su}.

We proceed to establish spectral gap properties of the Markov operator $Q_s$ by using the well-known theorem of 
Ionescu-Tulcea and Marinescu (see \cite{IM1950} or \cite[Theorem II.5]{HH01}).  
Below we denote by $\mathscr{L(B_{\gamma},B_{\gamma})}$ 
the set of bounded linear operators from $\scr B_{\gamma}$ to $\scr B_{\gamma}$
equipped with the operator norm $\left\| \cdot \right\|_{\scr B_{\gamma} \to \scr B_{\gamma}}$,
by $\rho(A)$ the spectral radius of an operator $A$ acting on $\scr B_{\gamma}$ 
and by $A|_E$ the operator $A$ restricted to the subspace $E$.

\begin{proposition} \label{spectra gap Q su 01}
Assume either \ref{Condi-FK} when $s \in I_{\mu}^-$, or \ref{Condi-FK-weak}  when $s \in I_{\mu}^+$.
 Let $s\in I_{\mu}^{\circ} \setminus \{0\}$. 
Then, for any $0 < \gamma \leq \min \{|s|, 1\}$, 
we have $Q_s \in \mathscr{L(B_{\gamma},B_{\gamma})}$ and $Q_s^n = \Pi_s + N^n_{s}$
 for any $n\geq1$, 
where 
$\Pi_s$ is a rank-one projection with $\Pi_s(\varphi)(v) = \pi_s(\varphi)$ 
for $\varphi \in \scr B_{\gamma}$ and $v\in \mathcal{S}$, 
and $N_s$ 
satisfies $\Pi_s N_s = N_s \Pi_s = 0$ 
and $\rho(N_s) < a_s$ for some constant $a_s \in (0,1)$. 
\end{proposition}

\begin{proof}
We first verify the following three claims.

\textit{Claim 1:} There exists a constant $c_s$ such that $\|Q_s \varphi\|_{\infty} \leq c_s\|\varphi\|_{\infty}$
for any $\varphi \in \scr B_{\gamma}$. 
This is clear by the definition of $Q_s$ (cf.\ \eqref{defi Q_s}). 

\textit{Claim 2:} The operator $Q_s$
is conditionally compact from 
$(\scr B_{\gamma}, \|\cdot\|_\gamma)$ to $(\mathcal{C(S)}, \|\cdot\|_\infty)$.
Since $Q_s$ is a Markov operator, the set
$\{ Q_s \varphi: \|\varphi\|_\gamma\leq 1\}$ is uniformly bounded. 
This, together with the equicontinuity of $Q_s$ (by \eqref{equiconti Q_s}), proves the claim by 
applying the Arzel\`{a}-Ascoli theorem.

\textit{Claim 3:} The Doeblin-Fortet inequality holds for $Q_s$.
This is proved in Lemma \ref{equiconti Q_s}. 

Applying the theorem of Ionescu-Tulcea and Marinescu to $Q_s$, 
we get that  $Q_s$ is quasi-compact, that is, 
$\scr B_{\gamma}$ can be decomposed into two $Q_s$ invariant closed subspaces 
$\scr B_{\gamma}=E\oplus F$ such that 
$\dim E < \infty$, each eigenvalue of $Q_s|_{E}$ has modulus $\rho(Q_s)$
and $\rho(Q_s|_{F})<\rho(Q_s)$.
Since $\rho(Q_s)=1$, it follows from Lemma \ref{irredu Q u 01} that $\dim E=1$.
We conclude the proof by applying Lemma III.3 in \cite{HH01}.
\end{proof}

As a direct consequence of Proposition \ref{spectra gap Q su 01} and \eqref{defi Q_s}, we get the following 

\begin{corollary}  \label{transfer operator}
Assume either \ref{Condi-FK} when $s \in I_{\mu}^-$, or \ref{Condi-FK-weak}  when $s \in I_{\mu}^+$.
  Let $s\in (I_{\mu}^+)^{\circ} \cup (I_{\mu}^-)^{\circ}$. 
Then, for any $0 < \gamma \leq \min \{|s|, 1\}$, 
we have $P_s \in \mathscr{L(B_{\gamma},B_{\gamma})}$ and $P_s^n = \kappa(s)^n M_s + L_s^n$ for any $n\geq1$, 
where $M_s : = \nu_s \otimes r_s$ is a rank-one projection on $\scr B_{\gamma}$ defined by
$M_s \varphi = \frac{ \nu_s(\varphi) }{ \nu_s(r_s) }  r_s$
for $\varphi \in \scr B_{\gamma}$,
and $L_s \in \mathscr{L(B_{\gamma},B_{\gamma})}$ satisfies $M_s L_s = L_s M_s =0$ and $\rho(L_s) < \kappa(s).$ 
\end{corollary}

\subsection{Spectral gap properties of the perturbed transfer operator $R_{s, z}$}  
For any $s \in I_{\mu}$, $q = \Lambda'(s)$ 
 and $z \in \bb C$ with $s + \Re z \in I_{\mu}$, 
define the perturbed operator $R_{s,z}$ as follows: for $\varphi\in \mathcal{C(S)}$, 
\begin{align} \label{def R tu 01}
R_{s, z}\varphi(v) = \mathbb{E}_{\mathbb{Q}_{s}^{v}} \left[ e^{ z( \log \|g v\| - q ) } \varphi(g \cdot v) \right],  
\quad v \in \mathcal{S}.  
\end{align}
Applying \eqref{cocycle02} and the induction argument, 
one can check that for any $n\geq 1$,
\begin{align*}
R_{s, z}^n \varphi(v) =\mathbb{E}_{\mathbb{Q}_{s}^{v}} \left[ e^{ z( \log \|G_n v\| - nq ) } \varphi(G_n \cdot v)  \right],
\quad v \in \mathcal{S}. 
\end{align*}
Now we use the spectral gap theory for the norm cocycle $\log \frac{\|G_n v\|}{\|v\|}$ shown in Section \ref{Sect-Spec-norm-cocycle}
and the perturbation theorem in \cite[Theorem III.8]{HH01}
to obtain the spectral gap properties of $R_{s, z}$. 
For $\delta>0$, denote $B_{\delta}(0) = \{z \in \mathbb{C}: |z| < \delta\}$. 

\begin{proposition} \label{perturbation thm 02}
Assume either \ref{Condi-FK} when $s \in I_{\mu}^-$, or \ref{Condi-FK-weak}  when $s \in I_{\mu}^+$.
Let $s \in I_{\mu}^{\circ} \setminus \{0\}$. 
Then, 
for any $0 < \gamma \leq \min \{|s|, 1\}$, 
there exists a constant $\delta>0$ such that 
for any $z \in B_{\delta}(0)$ and $n \geq 1$,
\begin{align}  \label{Pertur iden 01}
R^{n}_{s,z}=\lambda^{n}_{s,z} \Pi_{s,z} + N^{n}_{s,z},
\end{align}
where $z \mapsto \lambda_{s,z} = e^{-qz} \frac{\kappa(s+z)}{\kappa(s)}: B_{\delta}(0) \to \mathbb{C}$ is analytic,  
$z \mapsto \Pi_{s,z}$ and $z \mapsto N_{s,z}$: $B_{\delta}(0) \to \mathscr{L(B_{\gamma},B_{\gamma})}$
are analytic 
with $\Pi_{s,z} N_{s,z}=N_{s,z}\Pi_{s,z}=0$ and $\Pi_{s,0}(\varphi)(v) = \pi_s(\varphi)$ for $\varphi \in \scr B_{\gamma}$ 
and $v\in \mathcal{S}$. 
Moreover, for any compact set $J \subseteq (I_{\mu}^+)^{\circ} \cup (I_{\mu}^-)^{\circ}$, 
it holds that 
$$
\sup_{s \in J} \sup_{z \in B_{\delta}(0)} \rho(N_{s,z})  <1.
$$ 
\end{proposition}

\begin{proof}
We first show that $R_{s,z} \in \mathscr{L(B_{\gamma},B_{\gamma})}$. 
Using \eqref{change measure equ coeff}, Lemma \ref{lemma kappa 1}
and the fact that $r_{s}$ is strictly positive and bounded (cf.\ Proposition \ref{change of measure coeffi 01}), we get 
\begin{align}\label{Q usz 001}
\|R_{s,z} \varphi\|_{\infty} \leq c_s \|\varphi\|_{\infty} \mathbb{E}(\|g\|^{s+\Re z}). 
\end{align}
Note that, by Lemma \ref{Lem-Contin-Cocycle}, we have for any $g\in \Gamma_\mu$ and $v_1, v_2 \in \mathcal{S}$, 
\begin{align}\label{Inequality-log-v}
& \left| \log\|g v_1\| - \log\|g v_2\| \right|   \notag\\
& = \left| \log\|g v_1\| - \log\|g v_2\| \right| \mathds 1_{ \{ \mathbf{d}(v_1, v_2) \leq \frac{1}{2} \} } 
   +  \left| \log\|g v_1\| - \log\|g v_2\| \right|  \mathds 1_{ \{ \mathbf{d}(v_1, v_2) > \frac{1}{2} \} }  \notag\\
& \leq  c \mathbf{d}(v_1, v_2) +  2  \left| \log\|g v_1\| - \log\|g v_2\| \right|  \mathbf{d}(v_1, v_2)  \notag\\
& \leq  c' \mathbf{d}(v_1, v_2) +  c' (\log \|g \|)  \mathbf{d}(v_1, v_2)  \notag\\
& \leq  c'' \left( 1 + \log \|g \| \right) \mathbf{d}(v_1, v_2). 
\end{align}
Since for any $z_1, z_2 \in \bb C$ and $\gamma \in (0, 1]$,
\begin{align}\label{Inequa-ez1-ez2}
|e^{z_1} - e^{z_2}| \leq 2 \max \{ |z_1|^{1 - \gamma}, |z_2|^{1 - \gamma} \}  \max \{ e^{\Re  z_1},  e^{\Re  z_2} \} |z_1 - z_2|^{\gamma}, 
\end{align}
 using Lemma \ref{lemma kappa 1} and \eqref{Inequality-log-v}, we have
\begin{align}\label{Inequa-exp-zlog}
\left| e^{ z \log \|g v_1\| }  -   e^{ z \log \|g v_2\| }  \right|
& \leq c  (\log \|g\|)^{1- \gamma}  e^{(\Re z) \log \|g\|}  
  \left| \log\|g v_1\| - \log\|g v_2\| \right|^{\gamma}  \notag\\
&  \leq c'  (1 + \log \|g\|)  e^{(\Re z) \log \|g\|}  
  \mathbf{d}(v_1, v_2)^{\gamma}. 
\end{align}
From \eqref{Inequa-exp-zlog} and Lemma \ref{lemma kappa 1},
it follows that for any $v_1, v_2 \in \mathcal S$, 
\begin{align*} 
 |R_{s,z} \varphi(v_1) - R_{s,z} \varphi(v_2)|   
& \leq  c_s \|\varphi\|_{\infty}  
   \mathbb{E} \left[ \frac{1}{ \|g\|^{-s} }  \left| e^{ z \log \|g v_1\| }  -   e^{ z \log \|g v_2\| }  \right|  \right]    \notag\\
& \quad   +  c_s \mathbb{E} \left[  e^{ (s + \Re z) \log \|g\| }  \left| \varphi(g \cdot v_1) - \varphi(g \cdot v_2)  \right|  \right] 
  \notag\\
& \leq   c'_s \|\varphi\|_{\infty}  \bb E \left[ (1 + \log \|g\|)  e^{(s + \Re z) \log \|g\|}  \right] 
  \mathbf{d}(v_1, v_2)^{\gamma}   \notag\\
&\quad  +  c'_s  \|\varphi\|_{\gamma} \mathbb{E} \left[  e^{ (s + \Re z) \log \|g\| }  \right]  \mathbf{d}(v_1, v_2)^{\gamma}   \notag\\
& \leq  c''_s \|\varphi\|_{\gamma}  \bb E \left[ (1 + \log \|g\|)  e^{(s + \Re z) \log \|g\|}  \right] 
  \mathbf{d}(v_1, v_2)^{\gamma}. 
\end{align*}
Combining this with \eqref{Q usz 001} shows that
$R_{s,z} \in \mathscr{L(B_{\gamma},B_{\gamma})}$.

Since $R_{s,0} = Q_s$, by Proposition \ref{spectra gap Q su 01},
we get that $R_{s,0}$ has a simple dominate eigenvalue
and $\rho(R_{s,0}) = 1$. 
Since $z\mapsto e^{(s+z) \log \|gv\|}$ is analytic on $B_{\delta}(0)$ for some constant $\delta>0$, 
one can check that $z\mapsto R_{s,z}$ is holomorphic by applying \cite[Theorem V.3.1]{Yos80}.
Therefore, the assertions follow by applying the perturbation theorem (cf.\  \cite[Theorem III.8]{HH01}). 
The formula $\lambda_{s,z} = e^{-qz} \frac{\kappa(s+z)}{\kappa(s)}$ can be obtained in the same way as that in \cite[Proposition 3.3]{XGL19a}
and therefore we omit the details. 
\end{proof}

Now we show a Doeblin-Fortet inequality for the perturbed operator $R_{s,it}$.

\begin{lemma}\label{Lem-R-sit-DF-ineq}
Assume either \ref{Condi-FK} when $s \in I_{\mu}^-$, or \ref{Condi-FK-weak}  when $s \in I_{\mu}^+$.
Assume also that 
$\mathbb{E} [ \|g\|^{s} (1 + \log\|g\|)] < \infty$ for $s \in I_{\mu}$.
Then there exist constants $c_{s, n} > 0$ and $0<a<1$ such that 
for any $\bar{s}$-H\"older continuous function $\varphi$ on $\mathcal{S}$ and $n\geq1$, 
\begin{align}\label{Equi-contin-Inequ-Rsit}
[R_{s,it}^n \varphi]_{\bar{s}}
\leq  c_{s, n} \|\varphi\|_\infty + a^{\bar s n} [\varphi]_{\bar{s}}.
\end{align} 
In particular, for any $\varphi \in \mathcal{C(S)}$, the sequence $\{ R_{s,it}^n \varphi \}_{n \geq 1}$
is equicontinuous on $\mathcal{S}$. 
\end{lemma}

\begin{proof}
By the definition of $R_{s,it}$, we write
\begin{align} \label{equiconti Q su I-Rsit}
&  |R_{s,it}^n \varphi(v_1) - R_{s,it}^n \varphi(v_2)|  \notag\\
& \leq   \left|  \int \varphi( G_n \!\cdot\! v_1)  e^{ it( \log \|G_n v_1\| - nq ) }
\left[ q_{n}^{s}(G_n, v_1)- q_{n}^{s}(G_n, v_2)   \right] \mu(dg_1) \ldots \mu(dg_n)  \right|  \notag\\
& \quad  +  \left| \int \left[ \varphi( G_n\! \cdot\! v_1)- \varphi( G_n\! \cdot\! v_2)  \right] 
   e^{ it( \log \|G_n v_1\| - nq ) }   q_{n}^{s}(G_n, v_2) \mu(dg_1) \ldots \mu(dg_n)  \right|  \notag\\
& \quad  +    \left| \int  \varphi( G_n\! \cdot\! v_2)  
  \left| e^{ it( \log \|G_n v_1\| - nq ) } -  e^{ it( \log \|G_n v_2\| - nq ) } \right|   q_{n}^{s}(G_n, v_2) \mu(dg_1) \ldots \mu(dg_n)  \right|  \notag\\
& =: I_1 + I_2 + I_3.  
\end{align}
For $I_1$, we use Lemma \ref{kernel inequa coeff 01} and \eqref{Inequa-kappa-s}  to get
\begin{align}  \label{equiconti Q su I 1-Rsit}
I_1   \leq  c_{s}\|\varphi\|_{\infty}   \frac{ \bb E (\| G_n \|^s) }{\kappa(s)^n} \mathbf{d}(v_1,v_2)^{\bar{s}}
\leq  c'_{s}\|\varphi\|_{\infty} \, \mathbf{d}(v_1,v_2)^{\bar{s}}.
\end{align}
For $I_2$, by \eqref{equiconti Q su I 2 posi}, 
there exists a constant $0 < a <1$ such that 
\begin{align} \label{equiconti Q su I 2 posi-Rsit}
I_2
\leq [\varphi]_{\bar{s}}  \,  a^{\bar s n}  \mathbf{d}(v_1, v_2)^{\bar{s}}.  
\end{align}
For $I_3$,  from \eqref{Inequa-ez1-ez2} and \eqref{Inequality-log-v}, 
it follows that for any $\gamma \in (0, 1]$, 
\begin{align*}
 \left| e^{ it( \log \|G_n v_1\| - nq ) } -  e^{ it( \log \|G_n v_2\| - nq ) } \right| 
& \leq c  (\log \|G_n\|)^{1- \gamma}  
  \left| \log\|G_n v_1\| - \log\|G_n v_2\| \right|^{\gamma}  \notag\\
&  \leq c'  (1 + \log \|G_n\|)
  \mathbf{d}(v_1, v_2)^{\gamma}. 
\end{align*}
By Lemma \ref{lemma kappa 1}, we get $\|g_2 g_1 v\| = \|g_2 \frac{g_1 v}{\|g_1 v\|} \|  \|g_1 v\| \geq c \|g_2\| \|g_1\|$, so that 
\begin{align*}
\left|  q_{n}^{s}(G_n, v_2) \right| \leq  \frac{c}{\kappa(s)^n} \|G_n v_2\|^s \leq  c_n  \|g_n\|^s \ldots \|g_1\|^s. 
\end{align*}
Therefore, 
\begin{align}\label{Inequa-exp-zlog-bb}
I_3  & \leq  c \|\varphi\|_{\infty}  \left| \int  (1 + \log \|G_n\|)  q_{n}^{s}(G_n, v_2) \mu(dg_1) \ldots \mu(dg_n)  \right|
 \mathbf{d}(v_1, v_2)^{\gamma}   \notag\\
& \leq  c_n  \|\varphi\|_{\infty}   \big[ \bb E  \|g\|^s  (1 + \log \|g\|)  \big]^n    \mathbf{d}(v_1, v_2)^{\gamma}.  
\end{align}
Combining \eqref{equiconti Q su I-Rsit}, \eqref{equiconti Q su I 1-Rsit}, \eqref{equiconti Q su I 2 posi-Rsit} and \eqref{Inequa-exp-zlog-bb}, 
we conclude the proof of \eqref{Equi-contin-Inequ-Rsit}.
Since each $\varphi\in \mathcal{C(S)}$ 
can be approximated by a sequence of $\bar{s}$-H\"{o}lder continuous functions,
it follows that $\{ R_{s,it}^n \varphi \}_{n \geq 1}$ is equicontinuous on $\mathcal{S}$. 
\end{proof}

The following result gives the exponential decay of  $R_{s,it}$ for $t\in \mathbb{R} \setminus \{0\}$. 
Recall that $\rho(R_{s,it})$ denotes the spectral radius of $R_{s,it}$ acting on $\mathscr B_{\gamma}$. 

\begin{proposition} \label{Sca expdecay pertur}  
Assume either \ref{Condi-FK} when $s \in I_{\mu}^-$, or \ref{Condi-FK-weak}  when $s \in I_{\mu}^+$.
Assume also that \ref{Condi-NonLattice} holds. 
Let $J$ be a compact subset of $I_\mu^\circ \setminus \{0\}$.
Then,  for any $0 < \gamma \leq \min_{s \in J} \{ |s|, 1\}$ and any compact subset $T$ of $\mathbb{R}\backslash\{0\}$, 
we have 
$$\sup_{ s\in J }  \sup_{t \in T} \rho(R_{s,it}) < 1.$$
\end{proposition}

\begin{proof}
We provide only a sketch of the proof. Similarly to the approach used in the proof of Proposition \ref{spectra gap Q su 01}, 
one can apply Lemma \ref{Lem-R-sit-DF-ineq} 
along with the theorem of Ionescu-Tulcea and Marinescu to demonstrate that the operator $R_{s,it}$ is quasi-compact. 
Subsequently, by following the argument in  \cite[Proposition 3.7]{XGL19b}, 
one can show that $\rho(R_{s,it}) < 1$ for any fixed $s \in J$ and $t \in T$.  
Finally,  the desired result is obtained using the strategy outlined in \cite[Proposition 3.10]{XGL19b}.  
\end{proof}

\section{Spectral gap theory for the coefficients} \label{sec:spec gap entries}

\subsection{Spectral gap properties of the perturbed operator $R_{it, f}$}\label{Sec-pertur-Ritf}

For any $t \in \bb R$ and $f \in \mathcal S$, define the perturbed transfer operator $R_{it, f}$ as follows: 
for any $\mathcal{C(S)}$ and $v \in \mathcal S_{\epsilon}$, 
\begin{align}\label{transf-oper-it-f}
R_{it, f}\varphi(v)   
= \int_{\bb G} e^{it (\log \frac{\langle f, g v\rangle}{\langle f, v \rangle} - \lambda_{\mu} ) } \varphi( g \cdot v ) \mu(dg).
\end{align}
It is easy to see that for any $n \geq 1$, 
\begin{align}\label{transf-oper-it-f-n}
R_{it, f}^n\varphi(v)  
 = \int e^{it (\log \frac{\langle f, G_n v\rangle}{\langle f, v \rangle} - n \lambda_{\mu} ) } 
    \varphi( G_n \cdot v ) \mu(dg_1) \ldots \mu(dg_n).
\end{align}
The following result gives the spectral gap properties of the transfer operator $R_{it, f}$.

\begin{lemma}\label{transfer operator_Pit-f}
\noindent 1.   
Assume \ref{Condi-FK-weak}.      
Then, for any $0 < \gamma \leq 1$, 
there exists a constant $\delta>0$ such that for any $t \in (-\delta, \delta)$, $f \in \mathcal S$ and $n \geq 1$,  
\begin{align}\label{Pzn-decom_Pit-f}
R_{it, f}^n = \lambda_{it}^n \Pi_{it, f} + N_{it, f}^n,
\end{align}
where $\lambda_{it}$ is the same as in Lemma \ref{transfer operator_Pit},
$\Pi_{it, f}$ is a rank-one projection on $\scr B_{\gamma}$ satisfying
$\Pi_{0, f} \varphi(v) = \nu(\varphi)$
for any $\varphi \in \scr B_{\gamma}$, $f \in \mathcal S$ and $v \in \mathcal S_{\epsilon}$.  
Moreover, $\Pi_{it, f} N_{it, f} = N_{it, f} \Pi_{it, f} =0$ and $\rho(N_{it, f}) < 1$ 
for all $t \in (-\delta, \delta)$ and $f \in \mathcal S$. 

\noindent 2.   
Assume \ref{Condi-FK-weak} and \ref{Condi-NonLattice}. 
Then, for any $0 < \gamma \leq 1$ and any compact set $T \subseteq \mathbb{R}\backslash\{0\}$, 
there exists a constant $c = c(\gamma,T) > 0$ such that for any $t \in T$, $f \in \mathcal S$ and $n\geq 1$, 
\begin{align}\label{Non-lattice-Rnit}
\| R^{n}_{it, f} \varphi \|_{\gamma} \leq e^{-cn}  \| \varphi \|_{\gamma}.
\end{align} 
\end{lemma} 

\begin{proof}
By  \eqref{transf-oper-it-f-n} and the fact that $R_{it} = R_{0, it}$ with $R_{0,it}$ defined by \eqref{def R tu 01}, 
it holds that for any $t \in \bb R$, 
$f \in \mathcal S$, $v \in \mathcal S_{\epsilon}$, $n \geq 1$ and $\varphi \in \mathcal{C(S)}$, 
\begin{align*}
R_{it, f}^n\varphi(v) =  e^{-it \log \langle f, v\rangle } R_{it}^n \varphi_f(v), 
\quad \mbox{where} \quad 
\varphi_f(v) = e^{it \log \langle f, v \rangle} \varphi(v). 
\end{align*}
By \eqref{Pzn-decom_Pit}, we have 
\begin{align*}
R_{it}^n \varphi_f(v) =  \lambda_{it}^n \Pi_{it} \varphi_f(v) + N_{it}^n \varphi_f(v),
\end{align*}
so that for any $t \in \bb R$, $f \in \mathcal S$, $v \in \mathcal S_{\epsilon}$ and $\varphi \in \mathcal{C(S)}$, 
\begin{align*}
\Pi_{it, f} \varphi (v) =  e^{-it \log \langle f, v\rangle }  \Pi_{it}  \left[ e^{it \log \langle f, \cdot \rangle} \varphi(\cdot)  \right](v),
\quad  
N_{it, f} \varphi(v) =  e^{-it \log \langle f, v\rangle }  N_{it}  \left[ e^{it \log \langle f, \cdot \rangle} \varphi(\cdot)  \right](v). 
\end{align*}
Using Lemma \ref{transfer operator_Pit}, 
one can check all the properties of $\Pi_{it, f}$ and $N_{it, f}$ stated in Lemma \ref{transfer operator_Pit-f}
and we omit the details. This proves the first part of the lemma. 

The proof of \eqref{Non-lattice-Rnit} can be done in the same way as that of Proposition \ref{Sca expdecay pertur}. 
\end{proof}



\begin{lemma}\label{Lem-Bs-f-001}
Assume \ref{Condi-FK-weak} and $\int_{\bb G} \log N(g) \mu(dg) < \infty$. 
Let  $\epsilon \in (0,1)$ be from \eqref{equicon A4} of Lemma \ref{lem equiv Kesten}. 
Then, for any  $f \in \bb R_+^d \setminus \{0\}$, the function
\begin{align}\label{func-phi-001}
b(f, v) : = \lim_{n \to \infty}
\mathbb{E} \left( \log \frac{\langle f, G_n v \rangle}{\langle f, v \rangle} - n \lambda_{\mu}  \right),
\quad   v \in \mathcal{S}_{\epsilon}
\end{align}
is Lipschitz continuous on $\mathcal{S}_{\epsilon}$ and 
\begin{align}\label{def-bfv-aaa}
b(f, v) = \left.  \frac{ d \Pi_{it,f} }{ dt } \right|_{t=0} 1(v),  \quad  v \in \mathcal{S}_{\epsilon}. 
\end{align}
\end{lemma}

\begin{proof}
By Lemma \ref{transfer operator_Pit-f}, we have that for any Lipschitz continuous function $\varphi$ on $\mathcal{S}_{\epsilon}$, 
\begin{align*}
\mathbb{E}  
 \left[ \varphi(G_n \cdot v)  e^{ it \left( \log \frac{\langle f, G_n v \rangle}{\langle f, v \rangle} - n \lambda_{\mu} \right) }  \right]
= \lambda^{n}_{it} \Pi_{it, f} \varphi(v) + N^{n}_{it,f} \varphi(v),  \    v \in \mathcal{S}_{\epsilon}.
\end{align*}
Since $\lambda_{0} = 1$ and $\frac{d \lambda_{it} }{ dt } |_{t=0} =0$, 
differentiating both sides of the above equation with respect to $t$ 
at $0$ gives that for any $v \in \mathcal{S}_{\epsilon}$, 
\begin{align}\label{Bs01-f-001}
\mathbb{E}
  \left[ \varphi(G_n \cdot v)  \left( \log \frac{\langle f, G_n v \rangle}{\langle f, v \rangle} - n \lambda_{\mu} \right)  \right]  
  = \left. \frac{ d \Pi_{it, f} }{ dt } \right|_{t=0} \varphi(v) 
     + \left.  \frac{ d N^{n}_{it,f} }{ dt } \right|_{t=0} \varphi(v).
\end{align}
Again by Lemma \ref{transfer operator_Pit-f}, 
the first term on the right-hand side of \eqref{Bs01-f-001}
is Lipschitz continuous on $\mathcal{S}_{\epsilon}$, and the second term converges to $0$ exponentially fast as $n \to \infty$.
Hence, letting $n \to \infty$ in \eqref{Bs01-f-001}, 
we obtain \eqref{def-bfv-aaa}.
\end{proof}

\subsection{Spectral gap properties of the transfer operator $P_{s,f}$}\label{Sec:change-mes-prod}

The main goal of this section is to perform a change of measure for the coefficients $\langle f, G_n v \rangle$, 
utilizing the spectral gap properties of the transfer operator $P_{s,f}$ defined below. 
Let  $\epsilon \in (0,1)$ be as specified in \eqref{equicon A4-bb} of Lemma \ref{lem equiv Kesten}. 
For $s\in I_{\mu}$, $f \in  \mathcal{S}$ and $\varphi \in \mathcal{C(S_\epsilon)}$, define
\begin{align} \label{Def_P_sf}
P_{s,f} \varphi(v) 
=  \int_{\bb G} \frac{\langle f, g v \rangle^s}{\langle f, v \rangle^s} \varphi(g\cdot v) \mu(dg),
\quad  v \in \mathcal{S}_\epsilon. 
\end{align}
The operator $P_{s,f}$ relates to $P_s$ as follows: 
for any $\varphi \in \mathcal{C(S_\epsilon)}$,
\begin{align}\label{relation-Ps-Psf}
\langle f, \cdot \rangle^s P_{s,f} \varphi(\cdot) =  P_s (\langle f, \cdot \rangle^s \varphi). 
\end{align}
The following result describes the spectral gap properties of $P_{s,f}$.
Notably, the spectral radius of $P_{s,f}$ is independent of $f$, 
which plays an important role in defining the change of measure pertaining to the coefficients. 
In the result below the relevant quantities
such as $\kappa(s), r_s, \nu_s, L_s$
 are defined in Proposition \ref{change-measure-neg01} and  Corollary \ref{transfer operator}.

\begin{proposition} \label{change of measure coeffi 01}
Assume either \ref{Condi-FK} when $s \in I_{\mu}^-$, or \ref{Condi-FK-weak}  when $s \in I_{\mu}^+$.
 Let $s\in I_{\mu}^{\circ} \setminus \{0\}$. 
Then, for any $0 < \gamma \leq \min \{|s|, 1\}$, $f \in \mathcal{S}$ and $n \geq 1$, 
\begin{align}\label{Pzn-decom_Entry}
P_{s,f}^n = \kappa(s)^n M_{s,f} + L_{s,f}^n,
\end{align}
where, for $\varphi \in \scr B_{\gamma}$ and $v \in \mathcal{S}_\epsilon$, 
\begin{align}\label{Msf_Lsf}
M_{s,f} \varphi(v) = \nu_{s,f}(\varphi)  r_{s,f}(v), 
\qquad
L_{s,f}^n \varphi(v) =  \frac{1}{\langle f, v \rangle^s}  L_s^n \big(\langle f, \cdot \rangle^s \varphi \big)(v), 
\end{align}
with 
\begin{align}\label{eigenfunc_Psf}
r_{s, f}(v) = \frac{1}{\langle f, v \rangle^s} r_s(v),   
 \qquad 
 \nu_{s,f}(\varphi) = \frac{\nu_s(\langle f, \cdot \rangle^s \varphi)}{\nu_s(r_s)}. 
\end{align}
Moreover, $ \nu_{s,f}(r_{s,f}) = 1$, $M_{s,f} L_{s,f} = L_{s,f} M_{s,f} =0$ 
and $\rho(L_{s,f}) < \kappa(s)$
for any $f \in \mathcal{S}$. 
\end{proposition}

\begin{proof}
We first show that $r_{s,f}$ defined by \eqref{eigenfunc_Psf} satisfies 
 $P_{s,f} r_{s,f} = \kappa(s) r_{s,f}$. 
By the definition of $P_{s,f}$ and $P_s$ (cf.\  \eqref{Def_P_sf} and \eqref{transfoper001}), 
and the fact that $P_s r_s = \kappa(s) r_s$, we have 
\begin{align*}
P_{s, f} r_{s, f}(v) 
& =  \int_{\bb G} \frac{\langle f, g v \rangle^s}{\langle f, v \rangle^s} 
    \frac{1}{\langle f, g \cdot v \rangle^s} r_s(g \cdot v) \mu(dg)  \\ 
& =  \int_{\bb G} \frac{1}{\langle f, v \rangle^s}  |gv|^s  r_s(g \cdot v) \mu(dg)   \\ 
& =  \kappa(s) \frac{1}{\langle f, v \rangle^s} r_s(v)  \\
& =  \kappa(s) r_{s, f}(v). 
\end{align*}
We next show that  $\nu_{s,f}$ defined by \eqref{eigenfunc_Psf} satisfies $\nu_{s,f} P_{s,f} = \kappa(s) \nu_{s,f}$.
Using $\nu_s P_s = \kappa(s) \nu_s$, we get 
\begin{align*}
\nu_{s,f}(P_{s, f} \varphi) 
& = \frac{\nu_s(\langle f, \cdot \rangle^s P_{s, f} \varphi)}{\nu_s(r_s)}  \\
& = \frac{1}{\nu_s(r_s)} 
    \int_{\mathcal{S}} \int_{\bb G} \langle f, g v \rangle^s  \varphi(g \cdot v) \mu(dg) \nu_s(dv)  \\
& = \frac{1}{\nu_s(r_s)} \nu_s \left( P_s (\langle f, \cdot \rangle^s \varphi) \right)  \\
& = \frac{1}{\nu_s(r_s)} \kappa(s) \nu_s \left( \langle f, \cdot \rangle^s \varphi \right)  \\
& = \kappa(s) \nu_{s,f}(\varphi).  
\end{align*}

Now we prove \eqref{Pzn-decom_Entry} and the related properties. 
For any $n \geq 1$, $s\in I_{\mu}^{\circ} \setminus \{0\}$,  
$f \in \mathcal{S}$, $v \in \mathcal{S}_\epsilon$ and $\varphi \in \mathcal{C(S_\epsilon)}$, it holds
\begin{align*}
P_{s,f}^n \varphi(v) 
=  \frac{1}{\langle f, v \rangle^s} P_s^n \varphi_f(v)  
\quad  \mbox{with} \quad
\varphi_f(v) = \langle f, v \rangle^s \varphi(v). 
\end{align*} 
By Lemma \ref{transfer operator}, we have
\begin{align*}
P_s^n \varphi_f(v) 
= \kappa(s)^n \frac{ \nu_s(\varphi_f) }{ \nu_s(r_s) }  r_s(v)  + L_s^n \varphi_f(v).
\end{align*}
Therefore, 
\begin{align*}
P_{s,f}^n \varphi(v) 
& = \kappa(s)^n \frac{ \nu_s(\langle f, \cdot \rangle^s \varphi) }{ \nu_s(r_s) } \frac{1}{\langle f, v \rangle^s}  r_s(v)  
  + \frac{1}{\langle f, v \rangle^s}  L_s^n \big(\langle f, \cdot \rangle^s \varphi \big)(v)  \\
& =  \kappa(s)^n \nu_{s,f}(\varphi) r_{s,f}(v) + L_{s,f}^n \varphi(v), 
\end{align*}
which proves \eqref{Pzn-decom_Entry} and \eqref{Msf_Lsf}. 
Using $M_{s} L_{s} = L_{s} M_{s} =0$, one can verify that $M_{s,f} L_{s,f} = L_{s,f} M_{s,f} =0$. 
Since $\rho(L_s) < \kappa(s)$,
we get $\rho(L_{s,f}) < \kappa(s)$ for all $f \in \mathcal{S}$. 
\end{proof}

Let $s\in I_{\mu}$. 
For any $f \in \mathcal{S}$, 
define the Markov operator $Q_{s,f}$ as follows: for any $\varphi \in \mathcal{C(S_\epsilon)}$,
\begin{align*}
Q_{s,f} \varphi(v): = \frac{1}{\kappa(s) r_{s,f}(v)} P_{s,f}(\varphi r_{s,f})(v), \quad x\in \mathcal{S}_{\epsilon}. 
\end{align*}
Using \eqref{relation-Ps-Psf}, and \eqref{eigenfunc_Psf} in Proposition \ref{change of measure coeffi 01}, we can check that
$Q_{s,f}$ coincides with $Q_s$ (cf.\ \eqref{defi Q_s}) on $\mathcal{C(S_\epsilon)}$: 
for any $\varphi \in \mathcal{C(S_\epsilon)}$,
\begin{align*}
Q_{s,f} \varphi = Q_s \varphi.
\end{align*} 
By iteration, we get
\begin{align*}
Q_{s,f}^n \varphi(v)
= \int \varphi(G_n \!\cdot\! v) q_{s,f,n}(G_n, v) \mu(dg_1) \ldots \mu(dg_n), \quad n\geq 1,
\end{align*}
where
\begin{align*}
q_{s,f,n}(G_n, v)
= \frac{\langle f,G_n v \rangle^s}{\kappa(s)^n \langle f, v \rangle^s}
\frac{r_{s,f} (G_n \!\cdot\! v )}{r_{s,f}(v)}
= q_{s,n}(G_n, v),
\end{align*}
with $q_{s,n}(G_n, v)$ defined by \eqref{def-qnsGn}. The last equality can be checked using \eqref{eigenfunc_Psf}. 
From \eqref{cocycle02}, it holds that for any $n,m \geq 1$, 
\begin{align} \label{cocycle02-f}
q_{s,f,n}(G_n, v)  q_{s,f,m} ( g_{n+m} \ldots g_{n+1}, G_n \!\cdot\! v )
= q_{s,f,n+m}(G_{n+m}, v).
\end{align} 
As in \eqref{change measure equ coeff}, 
let $\bb Q^{v}_{s,f}$ be the unique probability measure on $M(d, \bb R)^{\bb N}$ with marginals 
$q_{s,f,n}(G_n, v) \mu(dg_1) \dots \mu(dg_n)$, then
we have for any bounded measurable function $h$ on $(\mathcal{S}_\epsilon \times \mathbb R)^n$, 
\begin{align} \label{change measure equ coeff2}
& \mathbb{E} \left[ r_{s,f}(G_n \!\cdot\! v) \frac{\langle f, G_n v \rangle^{s}}{\langle f, v \rangle^{s}}
h \left( G_1 \!\cdot\! v, \log \frac{\langle f, G_1 v \rangle}{\langle f, v \rangle}, 
\ldots,  G_n \!\cdot\! v,  \log \frac{\langle f, G_n v \rangle}{\langle f, v \rangle}  \right) \right]  \nonumber \\
& = \kappa(s)^n r_{s,f}(v)
  \mathbb{E}_{\bb Q^{v}_{s,f}} \left[ h \left( G_1 \!\cdot\! v, \log \frac{\langle f, G_1 v \rangle}{\langle f, v \rangle}, 
\ldots,  G_n \!\cdot\! v,  \log \frac{\langle f, G_n v \rangle}{\langle f, v \rangle} \right) \right]. 
\end{align}

Now we show the spectral gap properties for the Markov operator $Q_{s,f}$. 
Define $\pi_{s,f}(\varphi) =  \nu_{s,f} (\varphi r_{s,f})$ for $\varphi \in \mathcal{C(S)}$, 
and note that $\pi_{s,f}(1) = \nu_{s,f} (r_{s,f}) = 1$.

\begin{proposition} \label{spectra gap Q su 01-f}
Assume either \ref{Condi-FK} when $s \in I_{\mu}^-$, or \ref{Condi-FK-weak}  when $s \in I_{\mu}^+$.
Let $s\in I_{\mu}^{\circ} \setminus \{0\}$. 
Then, for any $0 < \gamma \leq \min \{|s|, 1\}$,  $f \in \mathcal{S}$ and $n \geq 1$, we have 
\begin{align*}
Q_{s,f}^n = \Pi_{s,f} + N^n_{s,f}, 
\end{align*}
where, for $\varphi \in \scr B_{\gamma} \cap \mathcal{C(S_\epsilon)}$ and $v \in \mathcal{S}_{\epsilon}$,
\begin{align*}
\Pi_{s,f} \varphi(v) = \pi_{s,f}(\varphi),
\qquad  
N^n_{s,f} \varphi(v) =  \frac{1}{\kappa(s)^n r_{s,f}(v)}  L_{s,f}^n(\varphi r_{s,f})(v), 
\end{align*}
and $\Pi_{s,f} N_{s,f} = N_{s,f} \Pi_{s,f} = 0$.  
In addition,
there exist constants $c_s > 0$ and $a \in (0,1)$ such that for all $f \in \mathcal{S}$ and $n \geq 1$, 
\begin{align}\label{Bound_Nsf_expo}
\|N^n_{s,f}\|_{\scr B_{\gamma} \to \scr B_{\gamma}}\leq c_s a^n. 
\end{align} 
\end{proposition}

\begin{proof}
By \eqref{Pzn-decom_Entry} and \eqref{Msf_Lsf}, 
we have that for any $\varphi \in \mathcal{C(S_\epsilon)}$ and $v \in \mathcal{S}_{\epsilon}$, 
\begin{align*}
Q_{s,f}^n \varphi(v) 
& = \frac{1}{\kappa(s)^n r_{s,f}(v)} P_{s,f}^n(\varphi r_{s,f})(v)  \\
& = \frac{1}{\kappa(s)^n r_{s,f}(v)} \left[ \kappa(s)^n M_{s,f}(\varphi r_{s,f})(v) + L_{s,f}^n(\varphi r_{s,f})(v) \right]  \\
& = \nu_{s,f} (\varphi r_{s,f}) +  \frac{1}{\kappa(s)^n r_{s,f}(v)}  L_{s,f}^n(\varphi r_{s,f})(v)  \\
& = \Pi_{s,f} \varphi + N^n_{s,f} \varphi(v).  
\end{align*}
Now we show that $\Pi_{s,f} N_{s,f} = 0$. 
By the definition of $\Pi_{s,f}$ and $N_{s,f}$,  we have for any $\varphi \in \scr B_{\gamma}$,
\begin{align*}
\Pi_{s,f} N_{s,f} \varphi 
= \Pi_{s,f} \left[ \frac{1}{\kappa(s) r_{s,f}}  L_{s,f}(\varphi r_{s,f}) \right]
= \frac{1}{\kappa(s)}  \nu_{s,f} L_{s,f}(\varphi r_{s,f}) 
= 0,
\end{align*}
where in the last equality we used $\nu_{s,f} L_{s,f} = 0$.
Similarly, one can use $L_{s,f} M_{s,f} = 0$ to verify that $N_{s,f} \Pi_{s,f} = 0$.  
The bound \eqref{Bound_Nsf_expo} follows from the fact that
$\rho(L_{s,f}) < \kappa(s)$ for all $f \in \mathcal{S}$ (cf. Proposition \ref{change of measure coeffi 01}). 
\end{proof}

\subsection{Spectral gap properties of the perturbed transfer operator $R_{s,z,f}$}  
Let $s \in I_\mu^{\circ} \setminus \{0\}$ and $q = \Lambda'(s)$. 
Let $\epsilon \in (0,1)$ be from \eqref{equicon A4-bb} of Lemma \ref{lem equiv Kesten}.
For any $f \in \mathcal{S}$ and $z \in \mathbb{C}$ with $s + \Re z \in I_{\mu}$, 
we define the perturbed operator $R_{s,z,f}$ as follows: for $\varphi\in \mathcal{C(S_\epsilon)}$, 
\begin{align} \label{def R tu 01-f}
R_{s,z,f} \varphi(v) 
= \mathbb{E}_{\mathbb{Q}_{s,f}^v} 
  \left[ e^{ z \left( \log \frac{\langle f, g v \rangle}{\langle f, v \rangle} - q \right) } \varphi(g \cdot v) \right], 
\quad v \in \mathcal{S}_\epsilon.   
\end{align}
By \eqref{cocycle02-f} and the induction argument, we have that for any $n \geq 1$,
\begin{align*}
R_{s,z,f}^n \varphi(v) 
= \mathbb{E}_{\mathbb{Q}_{s,f}^v} 
  \left[ e^{ z \left( \log \frac{\langle f, G_n v \rangle}{\langle f, v \rangle} - nq \right) } \varphi(G_n \cdot v) \right], 
\quad v \in \mathcal{S}_\epsilon.  
\end{align*}

Next we demonstrate the spectral gap properties of $R_{s,z,f}$. 
It is important that the dominant eigenvalue $\lambda_{s,z}$ of $R_{s,z,f}$ does not depend on $f \in  \mathcal{S}$.
Recall that $B_{\delta}(0) = \{z \in \mathbb{C}: |z| < \delta\}$ for $\delta>0$.

\begin{proposition} \label{perturbation thm 02-f}
Assume either \ref{Condi-FK} when $s \in I_{\mu}^-$, or \ref{Condi-FK-weak}  when $s \in I_{\mu}^+$.
Let $s \in I_\mu^{\circ} \setminus \{0\}$. 
Then, for any $0 < \gamma \leq \min \{|s|, 1\}$ and $f \in \mathcal{S}$,
there exists a constant $\delta > 0$ such that for any $z \in B_{\delta}(0)$, 
\begin{align}  \label{Pertur iden 01-f}
R^{n}_{s,z,f} = \lambda^{n}_{s,z} \Pi_{s,z,f} + N^{n}_{s,z,f}, 
\end{align}
where $\lambda_{s,z} = e^{-qz} \frac{\kappa(s+z)}{\kappa(s)}$, 
$z \mapsto \Pi_{s,z,f}$ and $z \mapsto N_{s,z,f}$: $B_{\delta}(0) \to \mathscr{L(B_{\gamma},B_{\gamma})}$
are analytic with $\Pi_{s,z,f} N_{s,z,f} = N_{s,z,f} \Pi_{s,z,f} = 0$ 
and $\Pi_{s,0,f}(\varphi)(v) = \pi_{s,f}(\varphi)$ for $\varphi \in \scr B_{\gamma}$ and $v \in \mathcal{S}_{\epsilon}$. 
Moreover, for any compact set $J \subseteq I_\mu^{\circ} \setminus \{0\}$, 
we have $\sup_{f \in \mathcal{S}} \sup_{s \in J} \sup_{z \in B_{\delta}(0)} \rho(N_{s,z,f})  <1$. 
\end{proposition}  

\begin{proof}
By the definition of $R_{s,z,f}$, it is straightforward to verify that for any $s \in I_{\mu}^{\circ}$ and $f \in \mathcal{S}$,
the mapping $z \mapsto R_{s,z,f}$ is analytic in a small neighborhood of $0$. 
According to  Proposition \ref{spectra gap Q su 01}, the operator $R_{s,0,f} = Q_{s,f}$ has spectral gap on $\scr B_{\gamma}$. 
Consequently, by applying the perturbation theorem (see \cite[Theorem III.8]{HH01}), 
we conclude that the operator $R_{s,z,f}$ also has spectral gap on $\scr B_{\gamma}$. 
In particular, $R_{s,z,f}$ has a dominant eigenvalue $\lambda_{s,f,z}$ with the corresponding eigenvector $\tilde{r}_{s,z,f}$:
\begin{align}\label{SP_R_szf_01}
R_{s,z,f} \tilde{r}_{s,z,f} = \lambda_{s,f,z} \tilde{r}_{s,z,f}. 
\end{align}
So it remains to prove that $\lambda_{s,f,z}$ coincides with $\lambda_{s,z}$, as defined in \eqref{Pertur iden 01}. 
By the definition of $R_{s,f,z}$ and $P_{z,f}$, along with the change of measure formula \eqref{change measure equ coeff},
we have for any $\varphi \in \scr B_{\gamma}$ and $n \geq 1$,
\begin{align}\label{lambda com 01-f}
R_{s,f,z} (\varphi)(v)
=  e^{ - qz} \frac{ P_{s+z,f} (\varphi r_{s,f})(v) }{ \kappa(s) r_{s,f}(v)  },  
  \quad v \in \mathcal{S}_{\epsilon}. 
\end{align}
Taking $\varphi = \tilde{r}_{s, f, u}$ and using \eqref{SP_R_szf_01}, we get
\begin{align} \label{lambda equ 001}
\lambda_{s,f, u} \kappa(s)  (r_{s,f} \tilde{r}_{s,z, f})(v)
= P_{s+z,f} (r_{s,f} \tilde{r}_{s, z, f})(v). 
\end{align}
Integrating the equation \eqref{lambda equ 001} with respect to $\nu_{s+z,f}$
and using the fact that $P_{s+z,f} \nu_{s+z,f} = \kappa(s+z) \nu_{s+z,f}$, we obtain that
$\lambda_{s,z,f} = e^{-qz} \frac{\kappa(s+z)}{\kappa(s)} = \lambda_{s,z}.$
\end{proof}

The following result gives the exponential decay of $R_{s,it,f}$ for $t\in \mathbb{R} \setminus \{0\}$. 
Recall that $\rho(R_{s,it,f})$ denotes the spectral radius of $R_{s,it,f}$ acting on $\mathscr B_{\gamma}$.

\begin{proposition} \label{Sca expdecay pertur-f}  
Assume either \ref{Condi-FK} when $s \in I_{\mu}^-$, or \ref{Condi-FK-weak}  when $s \in I_{\mu}^+$.
Assume also that \ref{Condi-NonLattice} holds. Let $J \subseteq I_\mu^{\circ} \setminus \{0\}$ be a compact set.
Then,  for any $0 < \gamma \leq \min_{s \in J} \{ |s|, 1\}$ and any compact set $T \subseteq \mathbb{R}\backslash\{0\}$, 
it holds that 
$$
\sup_{f \in \mathcal{S}} \sup_{ s\in J}  \sup_{t \in T} \rho(R_{s,it,f}) < 1.
$$
\end{proposition}

\begin{proof}
The proof can be carried out in the same manner as that of Proposition \ref{Sca expdecay pertur}. 
\end{proof}

For any $s \in I_\mu^{\circ} \setminus \{0\}$, $f \in \mathcal{S}$ and $\varphi \in \scr B_{\gamma}$, we define   
\begin{align}\label{Def-bsvarphi-fv}
b_{s, \varphi}(f, v): = \lim_{n \to \infty}  \mathbb{E}_{\mathbb{Q}_{s,f}^v} 
   \left[ \varphi(G_n \cdot v)  \left( \log \frac{\langle f, G_n v \rangle}{\langle f, v \rangle} - n \Lambda'(s) \right)  \right],
\   v \in \mathcal{S}_{\epsilon}. 
\end{align}
The following result shows that $v \mapsto b_{s, \varphi}(f, v)$ 
is a H\"{o}lder continuous function on $\mathcal{S}_{\epsilon}$. 

\begin{lemma}\label{Lem-Bs}
Assume \ref{Condi-FK}. 
Let $s \in I_\mu^{\circ} \setminus \{0\}$.  
Then, for any $0 < \gamma \leq \min \{|s|, 1\}$, $f \in \mathcal{S}$ and $\varphi \in \scr B_{\gamma}$,  
we have $b_{s,\varphi}(f, \cdot) \in \scr B_{\gamma}$ and
\begin{align} \label{Def2-bs}
b_{s,\varphi}(f, v) = \left.  \frac{ d \Pi_{s,f,z} }{ dz } \right|_{z=0} \varphi(v),  \quad  v \in \mathcal{S}_{\epsilon}. 
\end{align}
\end{lemma}

\begin{proof}
By Proposition \ref{perturbation thm 02-f}, we have that for any $\varphi \in \scr B_{\gamma}$, 
\begin{align*}
\mathbb{E}_{\mathbb{Q}_{s,f}^{v}} 
 \left[ \varphi(G_n \cdot v)  e^{ z \left( \log \frac{\langle f, G_n v \rangle}{\langle f, v \rangle} - n\Lambda'(s) \right) }  \right]
= \lambda^{n}_{s, z} \Pi_{s,f,z} \varphi(v) + N^{n}_{s,f,z} \varphi(v),  \    v \in \mathcal{S}_{\epsilon}.
\end{align*}
Since $\lambda_{s,0} = 1$ and $\frac{d \lambda_{s,z} }{ dz } |_{z=0} =0$, 
differentiating both sides of the above equation with respect to $z$ 
at $0$ yields that for any $v \in \mathcal{S}_{\epsilon}$, 
\begin{align}\label{Bs01}
\mathbb{E}_{\mathbb{Q}_{s,f}^{v}} 
  \left[ \varphi(G_n \cdot v)  \left( \log \frac{\langle f, G_n v \rangle}{\langle f, v \rangle} - n\Lambda'(s) \right)  \right]  
  = \left. \frac{ d \Pi_{s,f,z} }{ dz } \right|_{z=0} \varphi(v) 
     + \left.  \frac{ d N^{n}_{s,f,z} }{ dz } \right|_{z=0} \varphi(v).
\end{align}
Again by Proposition \ref{perturbation thm 02-f},
the first term on the right-hand side of \eqref{Bs01}
belongs to $\scr B_{\gamma}$, while the second term converges to $0$ exponentially fast as $n \to \infty$.
Therefore, letting $n \to \infty$ in \eqref{Bs01}, 
we obtain \eqref{Def2-bs}.
\end{proof}

\section{Proof of Edgeworth expansion and Berry-Esseen theorem} 

The goal of this section is to establish Theorems \ref{MainThm_Edgeworth_01} and \ref{Thm-BerryEsseen}. 
The proof of these theorems relies on the spectral gap properties of the perturbed operators $R_{it}$
and $R_{it, f}$ under polynomial moment assumptions, 
as developed in Sections \ref{Sec-spec-gap-Rit} and \ref{Sec-pertur-Ritf}.

\subsection{Proof of Theorem \ref{MainThm_Edgeworth_01}}

To establish Theorem \ref{MainThm_Edgeworth_01}, we begin by proving the following 
first-order Edgeworth expansion for the cocycle $\log  \frac{ \langle f, G_n v \rangle}{\langle f, v \rangle}$.

\begin{proposition}\label{Prop_Edgewoth_aa}
Assume \ref{Condi-FK-weak}, \ref{Condi-NonLattice} and $\int_{\bb G} \log^3 N(g) \mu(dg) < \infty$. 
Then, as $n \to \infty$, uniformly in $f \in \mathcal{S}$, $v \in \mathcal{S}_{\epsilon}$ and $y \in \mathbb{R}$,
\begin{align*}
& \bb P \left( \frac{ \log  \frac{ \langle f, G_n v \rangle}{\langle f, v \rangle} - n \lambda_{\mu} }{\sigma \sqrt{n}} \leq y \right)
  =  \Phi(y) + \frac{ m_3 }{ 6 \sigma^3 \sqrt{n}} (1-y^2) \phi(y) 
     - \frac{ b(f, v) }{ \sigma \sqrt{n} } \phi(y)   
     +  o \left( \frac{ 1 }{\sqrt{n}} \right). 
\end{align*}
\end{proposition} 

\begin{proof}
Throughout the proof we assume that $\gamma \in (0,1)$, ensuring that the conclusions in Lemma \ref{transfer operator_Pit-f} hold.  
For any $f \in \mathcal{S}$ and $v \in \mathcal{S}_{\epsilon}$, we denote
\begin{align*}
F_n(y): &= \Phi(y) + \frac{ m_3 }{ 6 \sigma^3 \sqrt{n}} (1-y^2) \phi(y) 
           - \frac{ b(f, v) }{ \sigma \sqrt{n} } \phi(y)   \notag\\
    & =  \Phi(y) - \frac{ m_3 }{ 6 \sigma^3 \sqrt{n}} \phi''(y) 
           - \frac{ b(f, v) }{ \sigma \sqrt{n} } \phi(y),  \quad  y \in \bb R.   
\end{align*}
It follows that 
\begin{align*}
F_n'(y) = \phi(y) - \frac{ m_3 }{ 6 \sigma^3 \sqrt{n}} \phi'''(y) 
           - \frac{ b(f, v) }{ \sigma \sqrt{n} } \phi'(y)
\end{align*}
and 
\begin{align*}
H_n(t): =  \int_{\bb R} e^{ity} F_n'(y) dy 
=  \left[  1 + (it)^3  \frac{ m_3 }{ 6 \sigma^3 \sqrt{n}} + it \frac{ b(f, v) }{ \sigma \sqrt{n} } \right]  e^{- \frac{t^2}{2}},  \quad  t \in \bb R. 
\end{align*}
Since $F_n(\infty) = 1$ and $F_n(-\infty) = 0$, 
by the Berry-Esseen inequality, we get that for any $T>0$, $f \in \mathcal{S}$ and $v \in \mathcal{S}_{\epsilon}$,
\begin{align*}
I_n:  =  \sup_{y \in \bb R}
  \left| \bb P \left( \frac{ \log  \frac{ \langle f, G_n v \rangle}{\langle f, v \rangle} - n \lambda_{\mu} }{\sigma \sqrt{n}} \leq y \right)  - F_n(y)  \right| 
  \leq \frac{1}{\pi} \int_{-T}^T  \Bigg|  \frac{ R^n_{\frac{it}{\sigma \sqrt{n}}, f} 1(v) - H_n(t) }{t} \Bigg| dt 
+ \frac{24 m}{\pi T},
\end{align*}
where $m = \sup_{n \geq 1} \sup_{y \in \bb R} \sup_{f \in \mathcal{S}} \sup_{ v \in \mathcal{S}_{\epsilon} } |F_n'(y)| < \infty$
since $b(f, v)$ is bounded uniformly in $f \in \mathcal{S}$ and $v \in \mathcal{S}_{\epsilon}$. 
For any $\ee >0$ (this should not be confused with $\epsilon$ in the definition of $\mathcal S_{\epsilon}$), 
we choose $\delta_1 >0$ satisfying $\frac{24 m}{\pi \delta_1} < \ee$, 
so that $\frac{24 m}{\pi T} < \frac{\ee}{\sqrt{n}}$ by taking $T = \delta_1 \sqrt{n}$. 
Therefore, for $\delta_2 \in (0,\delta_1)$, we get
\begin{align*}
I_n  \leq  \frac{1}{\pi}  \int_{\delta_2 \sqrt{n} \leq |t| \leq \delta_1 \sqrt{n}}  
     \Bigg|  \frac{ R^n_{\frac{it}{\sigma \sqrt{n}}, f} 1(v) - H_n(t) }{t} \Bigg| dt  
 +  \frac{1}{\pi} \int_{|t| < \delta_2 \sqrt{n} }  \Bigg|  \frac{ R^n_{\frac{it}{\sigma \sqrt{n}}, f} 1(v) - H_n(t) }{t} \Bigg| dt
   + \frac{\ee}{\sqrt{n}}. 
\end{align*}
By \eqref{Non-lattice-Rnit}, 
the first integral is bounded by $C e^{-cn}$, uniformly in $f \in \mathcal{S}$ and $v \in \mathcal{S}_{\epsilon}$. 
Here and in subsequent instances, the constants $c$ vary with $\gamma$.
For the second integral, by Lemma \ref{transfer operator_Pit-f}, we have 
\begin{align*}
R^n_{\frac{it}{\sigma \sqrt{n}}, f} 1(v) - H_n(t)
& =  \lambda^n_{ \frac{it}{\sigma \sqrt{n}} } \Pi_{ \frac{it}{\sigma \sqrt{n}}, f } 1(v) + N_{ \frac{it}{\sigma \sqrt{n}}, f }^n 1(v) 
   -  \left[  1 + (it)^3  \frac{ m_3 }{ 6 \sigma^3 \sqrt{n}} + it \frac{ b(f, v) }{ \sigma \sqrt{n} } \right]  e^{- \frac{t^2}{2}}   \notag\\
& =: \sum_{k=1}^4 J_k(t),
\end{align*}
where 
\begin{align*}
& J_1(t) =  \lambda^n_{ \frac{it}{\sigma \sqrt{n}} }  
   - \left[  1 + (it)^3  \frac{ m_3 }{ 6 \sigma^3 \sqrt{n}}  \right]  e^{- \frac{t^2}{2}}  
  =  e^{- \frac{t^2}{2}} \left( e^{ n \log \lambda_{ \frac{it}{\sigma \sqrt{n}} } + \frac{t^2}{2} } - 1
                            -  (it)^3  \frac{ m_3 }{ 6 \sigma^3 \sqrt{n}}   \right),   \notag\\
& J_2(t) = \lambda^n_{ \frac{it}{\sigma \sqrt{n}} } 
        \left[ \Pi_{ \frac{it}{\sigma \sqrt{n}}, f } 1(v) - 1 -  it \frac{ b(f, v) }{ \sigma \sqrt{n} }  \right],   \notag\\
& J_3(t) =  it \frac{ b(f, v) }{ \sigma \sqrt{n} }  \left[ \lambda^n_{ \frac{it}{\sigma \sqrt{n}} } -  e^{- \frac{t^2}{2}}  \right],     \notag\\
& J_4(t) = N_{ \frac{it}{\sigma \sqrt{n}}, f }^n 1(v). 
\end{align*}
For $J_1(t)$, let $h_n(t) = n \log \lambda_{ \frac{it}{\sigma \sqrt{n}} } + \frac{t^2}{2}$.
Using Lemma \ref{transfer operator_Pit-f}, we have $h_n(0) = h_n'(0) = h_n''(0) = 0$
and $h_n'''(0) = - \frac{i m_3}{\sigma^3 \sqrt{n}}$, so that, as $n\to\infty$, 
\begin{align*}
h_n(t) =  (it)^3 \frac{m_3}{6 \sigma^3 \sqrt{n}}  +  t^3 o \left( \frac{1}{\sqrt{n}} \right). 
\end{align*}
It follows that, for any $t \in [- \delta_2 \sqrt{n}, \delta_2 \sqrt{n}]$,  
\begin{align*}
\left| h_n(t) - (it)^3 \frac{i m_3}{6 \sigma^3 \sqrt{n}} \right| \leq |t|^3  o \left( \frac{1}{\sqrt{n}} \right),
\quad
\left|  (it)^3 \frac{i m_3}{6 \sigma^3 \sqrt{n}} \right|^2 \leq  c \frac{t^6}{n},
\end{align*}
and, by taking $\delta_2 >0$ sufficiently small, 
\begin{align*}
 e^{| h_n(t) |}  \leq  e^{\frac{t^2}{4}}, 
\qquad 
 e^{|  (it)^3 \frac{i m_3}{6 \sigma^3 \sqrt{n}}| } \leq  e^{\frac{t^2}{4}}.   
\end{align*}
Therefore, using the inequality $|e^z - 1 - w| \leq (|z - w| + \frac{1}{2} |w|^2) e^{\max\{|z|, |w| \}}$ for $z, w \in \bb C$,
 we get that for any $t \in [- \delta_2 \sqrt{n}, \delta_2 \sqrt{n}]$, 
\begin{align}\label{Boind-I1t-f}
|J_1(t)| \leq  e^{-\frac{t^2}{4}}  \left[ |t|^3  o \left( \frac{1}{\sqrt{n}} \right)  +  c \frac{t^6}{n} \right], 
\end{align}
which implies that uniformly in $f \in \mathcal{S}$ and $v \in \mathcal{S}_{\epsilon}$, 
\begin{align}\label{Int-Bound-I1t-f}
\int_{|t| < \delta_2 \sqrt{n} }   \left|  \frac{ J_1(t) }{t} \right| dt
=  o \left( \frac{1}{\sqrt{n}} \right). 
\end{align}
For $J_2(t)$,  
by Lemma \ref{transfer operator_Pit-f} and \eqref{def-bfv-aaa}, 
there exists a constant $c>0$ such that for any $t \in [- \delta_2 \sqrt{n}, \delta_2 \sqrt{n}]$,
\begin{align*}
\left| \Pi_{ \frac{it}{\sigma \sqrt{n}}, f } 1(v) - 1 -  it \frac{ b(f, v) }{ \sigma \sqrt{n} }  \right|
\leq  c \frac{t^2}{n}. 
\end{align*}
From \eqref{Boind-I1t-f}, we derive that 
\begin{align*}
\left|  \lambda^n_{ \frac{it}{\sigma \sqrt{n}} } \right|
\leq  e^{-\frac{t^2}{4}}  \left( 1 + c\frac{|t|^3}{\sqrt{n}}  +  c \frac{t^6}{n} \right). 
\end{align*}
Therefore,  we get
\begin{align}\label{Int-Bound-I2t-f}
\int_{|t| < \delta_2 \sqrt{n} }  \left|  \frac{ J_2(t) }{t} \right| dt
\leq  \frac{c}{n}  \int_{|t| < \delta_2 \sqrt{n} }  e^{-\frac{t^2}{4}}  \left( |t| + c \frac{|t|^4}{\sqrt{n}} + c \frac{|t|^7}{n} \right)   dt
\leq \frac{c'}{n}. 
\end{align}
For $J_3(t)$, by \eqref{Boind-I1t-f}
and the fact that $b(f, v)$ is bounded uniformly in $f \in \mathcal{S}$ and $v \in \mathcal{S}_{\epsilon}$, 
it holds that 
\begin{align*}
| J_3(t) | \leq  c \frac{|t|}{\sqrt{n}} \left( | J_1(t) | + c \frac{|t|^3}{\sqrt{n}}  e^{-\frac{t^2}{2}}  \right)
\leq  c'  e^{-\frac{t^2}{4}}   \left( \frac{|t|^4}{n} + \frac{|t|^7}{n^{3/2}} \right), 
\end{align*}
which implies that
\begin{align}\label{Int-Bound-I3t-f}
\int_{|t| < \delta_2 \sqrt{n} }  \left|  \frac{ J_3(t) }{t} \right| dt
\leq \frac{c}{n}  \int_{|t| < \delta_2 \sqrt{n} }  e^{-\frac{t^2}{4}}  \left( |t|^3 +  \frac{|t|^6}{\sqrt{n}} \right) dt
\leq  \frac{c'}{n}.  
\end{align}
For $J_4(t)$, using Lemma \ref{transfer operator_Pit-f} and the fact that $N_{0,f} 1 = 0$, we deduce that 
\begin{align}\label{Int-Bound-I4t-f}
\int_{ |t| < \delta_2 \sqrt{n} }  \left|  \frac{ J_4(t) }{t} \right| dt
\leq C e^{-cn}.  
\end{align}
Putting together  \eqref{Int-Bound-I1t-f}, \eqref{Int-Bound-I2t-f}, \eqref{Int-Bound-I3t-f} and \eqref{Int-Bound-I4t-f}
completes the proof of Proposition \ref{Prop_Edgewoth_aa}. 
\end{proof}

For $f, v  \in \bb R_+^{d} \setminus \{0\}$, let 
\begin{align}\label{Def_A_fv}
A(f,v)
 = \int_{\bb G}  \left( \log \langle f, g v \rangle - \lambda_{\mu} \right)   \mu(dg),
\quad 
B(f,v) = \int_{\bb G} b(f, g \cdot v) \mu(dg). 
\end{align}
It is easy to see that both $A(f,v)$ and $B(f,v)$ are well defined 
under condition \ref{Condi-FK-weak} and $\int_{\bb G} \log N(g) \mu(dg) < \infty$.

\begin{proof}[Proof of Theorem \ref{MainThm_Edgeworth_01}]
We divide the proof into the following three steps.

\textit{Step 1.}
We prove that, as $n \to \infty$, uniformly in $f, v  \in \mathcal{S}$ and $y \in \mathbb{R}$,
\begin{align}\label{MyEdgExp_hh}
& \bb P \left( \frac{\log \langle f, G_n v \rangle - n \lambda_{\mu} }{\sigma \sqrt{n}} \leq y \right)  \notag\\
& =  \Phi(y) + \frac{ m_3 }{ 6 \sigma^3 \sqrt{n}} (1-y^2) \phi(y) 
    - \frac{A(f,v) + B(f, v) }{ \sigma \sqrt{n} } \phi(y)  +  o \left( \frac{ 1 }{\sqrt{n}} \right). 
\end{align}
For $n \geq 2$, $f \in \mathcal{S}$, $v' \in \mathcal{S}_{\epsilon}$ and $y' \in \mathbb{R}$, we denote  
\begin{align*}
P_{n}(v', y') = \bb P 
\left( \frac{ \log  \frac{ \langle f, g_n \ldots g_2 v' \rangle}{\langle f, v' \rangle} - (n-1) \lambda_{\mu} }{\sigma \sqrt{n - 1}} \leq y' \right). 
\end{align*}
By Proposition \ref{Prop_Edgewoth_aa}, we get that, 
as $n \to \infty$, uniformly in $f \in \mathcal{S}$, $v' \in \mathcal{S}_{\epsilon}$ and $y' \in \mathbb{R}$,
\begin{align}\label{Edgeworth_aaa}
P_{n}(v', y') 
& = \Phi(y') + \frac{m_3}{ 6 \sigma^3 \sqrt{n-1}} [1 - (y')^2] \phi(y') 
    - \frac{ b(f, v') }{ \sigma \sqrt{n-1} } \phi(y')  +  o \left( \frac{ 1 }{\sqrt{n}} \right)  \notag\\
& = \int_{- \infty}^{y'} H_n(v', t) dt +  o \left( \frac{ 1 }{\sqrt{n}} \right),  
\end{align}
where 
\begin{align}\label{Pf_Edge_Def_Hn}
H_n(v', t) = \phi(t) - \frac{m_3}{ 6 \sigma^3 \sqrt{n-1}} \phi'''(t) 
    - \frac{ b(f, v') }{ \sigma \sqrt{n-1} } \phi'(t). 
\end{align}
For short, denote 
\begin{align*}
\ell_{f,v} (g,y) = y \sqrt{\frac{n}{n - 1}} + \frac{ \lambda_{\mu} - \log \langle f, g v \rangle}{\sigma \sqrt{n-1}}. 
\end{align*}
Since for any fixed $g \in \supp \mu$, 
\begin{align} \label{Edgeworth_Coeff01}
\bb P \left( \frac{\log \langle f, g_n \ldots g_2 g v \rangle - n \lambda_{\mu} }{\sigma \sqrt{n}} \leq y \right)
=  P_{n} \left(g \cdot v, \ell_{f,v}(g,y) \right),
\end{align}
taking the conditional expectation with respect to $g_1$, we have
\begin{align} \label{Edgeworth_Coeff02}
\bb P \left( \frac{\log \langle f, g_n \ldots g_2 g_1 v \rangle - n \lambda_{\mu} }{\sigma \sqrt{n}} \leq y \right)
= \mathbb E P_{n} \left( g_1 \cdot v,  \ell_{f,v} (g_1,y) \right). 
\end{align}
From \eqref{Edgeworth_aaa}, it follows that,
as $n \to \infty$, uniformly in $g \in \Gamma_{\mu}$, $f, v  \in \mathcal{S}$ and $y \in \mathbb{R}$, 
\begin{align*}
P_{n} \left(g \cdot v, \ell_{f,v}(g,y) \right)
&  = \int_{- \infty}^{\ell_{f,v}(g,y)} H_n(g \cdot v, t) dt +  o \left( \frac{ 1 }{\sqrt{n}} \right)  \notag\\
&  = \int_{- \infty}^{y} H_n(g \cdot v, t) dt -  \int_{\ell_{f,v}(g,y)}^{y} H_n'(g \cdot v, t) dt
     +  o \left( \frac{ 1 }{\sqrt{n}} \right). 
\end{align*}
By the Lebesgue dominated convergence theorem, we get that,
as $n \to \infty$, uniformly in $f, v  \in \mathcal{S}$ and $y \in \mathbb{R}$, 
\begin{align*}
\mathbb E P_{n} \left( g_1 \cdot v,  \ell_{f,v} (g_1,y) \right)   = I_1 - I_2  +  o \left( \frac{ 1 }{\sqrt{n}} \right), 
\end{align*}
where 
\begin{align*}
 I_1 = \int_{\bb G} \int_{- \infty}^{y} H_n(g \cdot v, t) dt \mu(dg),  \quad 
 I_2 = \int_{\bb G} \int_{\ell_{f,v}(g,y)}^{y} H_n(g \cdot v, t) dt \mu(dg). 
\end{align*}

\textit{Estimate of $I_1$.} 
In view of \eqref{Pf_Edge_Def_Hn} and \eqref{Def_A_fv}, by elementary calculations, we get
\begin{align}\label{Pf_Edge_I1}
I_1 & = \int_{\bb G} \int_{- \infty}^{y} 
     \left[ \phi(t) - \frac{m_3}{ 6 \sigma^3 \sqrt{n-1}} \phi'''(t) 
            - \frac{ b(f, g \cdot v) }{ \sigma \sqrt{n-1} } \phi'(t)  \right]  dt \mu(dg)  \notag\\
  & =  \Phi(y) + \frac{m_3}{ 6 \sigma^3 \sqrt{n}} (1 - y^2) \phi(y) 
      - \frac{ B(f,v) }{ \sigma \sqrt{n} } \phi(y) 
       +  O \left( \frac{1}{n} \right). 
\end{align}

\textit{Estimate of $I_2$.} 
By Fubini's theorem, we have the following decomposition:
\begin{align}\label{Pf_Edge_Decom_I2}
I_2 = I_{21} - I_{22} - I_{23} - I_{24},
\end{align}
where 
\begin{align*}
& I_{21} = \int_{\bb G} \int_{y + \frac{\lambda_{\mu} - \log \langle f, g v \rangle}{\sigma \sqrt{n - 1}}}^y  
        \phi(t) dt \mu(dg),   \notag\\
& I_{22} = \int_{\bb G} 
  \int_{y \sqrt{\frac{n}{n - 1}} + \frac{\lambda_{\mu} - \log \langle f, g v \rangle}{\sigma \sqrt{n-1}}}
     ^{y + \frac{\lambda_{\mu} - \log \langle f, g v \rangle}{\sigma \sqrt{n - 1}}} 
        \phi(t) dt \mu(dg),    \notag\\
& I_{23} =  \frac{m_3}{ 6 \sigma^3 \sqrt{n-1}}   
     \int_{-\infty}^y  
    \bb P \left(y \geq t > y \sqrt{\frac{n}{n - 1}} + \frac{\lambda_{\mu} - \log \langle f, g v \rangle}{\sigma \sqrt{n-1}} \right)
    \phi'''(t) dt,   \notag\\
& I_{24} =  \frac{1}{\sigma \sqrt{n-1}} 
     \int_{-\infty}^y   \int_{\bb G}  b(f, g \cdot v) 
  \mathds 1_{ \left\{ y \geq t > y \sqrt{\frac{n}{n - 1}} + \frac{\lambda_{\mu} - \log \langle f, g v \rangle}{\sigma \sqrt{n-1}}  \right\} }
   \mu(dg) \phi'(t) dt. 
\end{align*}

\textit{Estimate of $I_{21}$.}
By a change of variable $t' = \sigma \sqrt{n - 1} (y-t)$, we get
\begin{align*}
I_{21} 
& =  \frac{1}{\sigma \sqrt{n - 1}} \int_{\bb G} \int_{0}^{\log \langle f, g v \rangle - \lambda_{\mu}}  
            \phi \left( y - \frac{t}{\sigma \sqrt{n - 1}} \right) dt \mu(dg).   
\end{align*}
Since the density function $\phi$ is Lipschitz continuous on $\bb R$, 
there exist constants $c,c'>0$ such that for any $y \in \bb R$, 
\begin{align*}
&  \frac{1}{\sigma \sqrt{n - 1}}  \int_{\bb G}  \left| \int_{0}^{\log \langle f, g v \rangle - \lambda_{\mu} } 
      \left[ \phi \left( y - \frac{t}{\sigma \sqrt{n - 1}} \right) - \phi(y) \right]  dt  \right|   \mu(dg)   \notag\\
& \leq  \frac{c}{n}  \int_{\bb G}  \left| \int_{0}^{|\log \langle f, g v \rangle - \lambda_{\mu} |}  |t|  dt  \right|   \mu(dg)  \notag\\
&  =  \frac{c}{n} \bb E \left( |\log \langle f, g v \rangle - \lambda_{\mu} |^2 \right)   \leq \frac{c'}{n}. 
\end{align*}
Therefore, in view of \eqref{Def_A_fv}, 
as $n \to \infty$, uniformly in $f, v  \in \mathcal{S}$ and $y \in \mathbb{R}$, 
\begin{align}\label{Pf_Edge_I21}
I_{21} & = \frac{\phi(y)}{\sigma \sqrt{n - 1}}  
       \int_{\bb G} \int_{0}^{\log \langle f, g v \rangle - \lambda_{\mu} }  dt \mu(dg) 
         +  o \left( \frac{ 1 }{\sqrt{n}} \right)  \notag\\
     & = \frac{\phi(y)}{\sigma \sqrt{n - 1}}  
       \int_{\bb G}  \left( \log \langle f, g v \rangle - \lambda_{\mu} \right)   \mu(dg)
         +  o \left( \frac{ 1 }{\sqrt{n}} \right)  \notag\\
      & = \frac{\phi(y)}{\sigma \sqrt{n}} A(f,v) 
         +  o \left( \frac{ 1 }{\sqrt{n}} \right).  
\end{align}

\textit{Estimate of $I_{22}$.}
Taking $\delta > 0$ sufficiently small, we have, for sufficiently large $n \geq 1$, 
\begin{align}\label{Pf_Edge_I22}
|I_{22}| &  \leq  \left| \int_{\bb G} 
  \int_{y \sqrt{\frac{n}{n - 1}} + \frac{\lambda_{\mu}  - \log \langle f, g v \rangle}{\sigma \sqrt{n-1}}}
     ^{y + \frac{\lambda_{\mu}  - \log \langle f, g v \rangle}{\sigma \sqrt{n - 1}}} 
        \phi(t) dt \mathds 1_{ \left\{ | \lambda_{\mu}  - \log \langle f, g v \rangle| >  n^{1/2 - \delta} \right\} } \mu(dg)
        \right|  \notag\\
 & \quad  +   \left| \int_{\bb G} 
  \int_{y \sqrt{\frac{n}{n - 1}} + \frac{ \lambda_{\mu} - \log \langle f, g v \rangle}{\sigma \sqrt{n-1}}}
     ^{y + \frac{\lambda_{\mu}  - \log \langle f, g v \rangle}{\sigma \sqrt{n - 1}}} 
        \phi(t) dt \mathds 1_{ \left\{ |\lambda_{\mu}  - \log \langle f, g v \rangle| \leq  n^{1/2 - \delta} \right\} } \mu(dg)
        \right|   \notag\\
 & \leq  \bb P \left( |\lambda_{\mu}  - \log \langle f, g v \rangle| >  n^{1/2 - \delta} \right) 
        + \sup_{|l| \leq 2 \sigma^{-1} n^{-\delta}} \left| \int^{y \sqrt{\frac{n}{n - 1}} + l}_{y + l}  \phi(t) dt \right|  \notag\\
 &  \leq  c \frac{1 + \bb E (\log^2 N(g))}{n^{1 - 2 \delta}} 
     +  \left( \sqrt{\frac{n}{n-1}} - 1 \right) \mathds 1_{\{ |y| \leq 1 \}}    \notag\\
 & \quad    + \left( \sqrt{\frac{n}{n-1}} - 1 \right) \mathds 1_{\{ |y| > 1 \}}   
              |y| \sup_{|l| \leq 2 \sigma^{-1} n^{-\delta}} \phi \left( y \sqrt{\frac{n}{n - 1}} + l \right)  \notag\\
 & \leq  \frac{c'}{n^{1 - 2 \delta}}. 
\end{align}

\textit{Estimate of $I_{23}$.}
Taking $\delta > 0$ sufficiently small, 
by the Lebesgue dominated convergence theorem and Markov's inequality, we get that,
as $n \to \infty$, uniformly in $f, v  \in \mathcal{S}$ and $y \in \mathbb{R}$, 
\begin{align*}
&  \int_{-\infty}^y
    \bb P \left(y \geq t > y \sqrt{\frac{n}{n - 1}} + \frac{\lambda_{\mu} - \log \langle f, g v \rangle}{\sigma \sqrt{n-1}} \right) 
    \phi'''(t) dt  \notag\\
&  \leq  \int_{-\infty}^y 
    \bb P \left(y \geq t > y \sqrt{\frac{n}{n - 1}} + \frac{\lambda_{\mu} - \log \langle f, g v \rangle}{\sigma \sqrt{n-1}}, 
                    |\log \langle f, g v \rangle| \leq n^{1/2 - \delta} \right)  \phi'''(t) dt  \notag\\
& \quad  +  \bb P \left( |\log \langle f, g v \rangle| >  n^{1/2 - \delta} \right)  
            \int_{-\infty}^y  \phi'''(t) dt   \notag\\
& \leq o(1) + C \frac{\bb E (\log^2 N(g) )}{n^{1 - 2 \delta}}  = o(1). 
\end{align*}
This implies that, as $n \to \infty$, uniformly in $f, v  \in \mathcal{S}$ and $y \in \mathbb{R}$, 
\begin{align}\label{Pf_Edge_I23}
I_{23} = o \left( \frac{ 1 }{\sqrt{n}} \right).
\end{align}

\textit{Estimate of $I_{24}$.} 
Using the fact that $|b(f, g \cdot v)|$ is bounded by a constant,  
uniformly in $f, v  \in \mathcal{S}$ and $g \in \Gamma_{\mu}$, 
one can follow the estimate of $I_{23}$ to show that, 
as $n \to \infty$, uniformly in $f, v  \in \mathcal{S}$ and $y \in \mathbb{R}$, 
\begin{align}\label{Pf_Edge_I24}
I_{24} = o \left( \frac{ 1 }{\sqrt{n}} \right).
\end{align}

Substituting \eqref{Pf_Edge_I21}, \eqref{Pf_Edge_I22}, \eqref{Pf_Edge_I23} and \eqref{Pf_Edge_I24} into \eqref{Pf_Edge_Decom_I2},
we get that, as $n \to \infty$, uniformly in $f, v  \in \mathcal{S}$ and $y \in \mathbb{R}$, 
\begin{align*}
I_{2} = \frac{\phi(y)}{\sigma \sqrt{n}}  A(f,v)
       + o \left( \frac{ 1 }{\sqrt{n}} \right).
\end{align*}
Combining this with \eqref{Pf_Edge_I1}, we conclude the proof of \eqref{MyEdgExp_hh}. 

\textit{Step 2.}
For any compact set $K \subset \bb{R}^{d}_+ \setminus \{0\}$, 
we prove that uniformly in $f, v  \in K$ and $y \in \mathbb{R}$,  
\begin{align}\label{Edgeworth-f-v}
& \bb P \left( \frac{\log \langle f, G_n v \rangle - n \lambda_{\mu} }{\sigma \sqrt{n}} \leq y \right)  \notag\\
& =  \Phi(y) + \frac{ m_3 }{ 6 \sigma^3 \sqrt{n}} (1-y^2) \phi(y) 
    - \frac{A(f,v) + B(f, v) }{ \sigma \sqrt{n} } \phi(y)  +  o \left( \frac{ 1 }{\sqrt{n}} \right). 
\end{align}
Denote 
\begin{align}\label{Notation-fv-yn}
\bar f = \frac{f}{\|f\|},  \quad  \bar v = \frac{v}{\|v\|},  \quad y_n = y - \frac{\log (\|f\| \|v\|)}{\sigma \sqrt{n}}. 
\end{align}
Applying \eqref{MyEdgExp_hh}, we get that uniformly in $f, v  \in K$ and $y \in \mathbb{R}$,
\begin{align}\label{Edge-barfbarv}
& \bb P \left( \frac{\log \langle f, G_n v \rangle - n \lambda_{\mu} }{\sigma \sqrt{n}} \leq y \right) 
  =  \bb P \left( \frac{\log \langle \bar f, G_n \bar v \rangle - n \lambda_{\mu} }{\sigma \sqrt{n}} \leq y_n \right)   \notag\\
& =  \Phi(y_n) + \frac{ m_3 }{ 6 \sigma^3 \sqrt{n}} (1-y_n^2) \phi(y_n) 
    - \frac{A(\bar f, \bar v) + B(\bar f, \bar v) }{ \sigma \sqrt{n} } \phi(y_n)  +  o \left( \frac{ 1 }{\sqrt{n}} \right). 
\end{align}
By Taylor's expansion, it holds that uniformly in $f, v  \in K$ and $y \in \mathbb{R}$, 
\begin{align*}
\Phi(y_n) & = \Phi(y) -  \frac{\log (\|f\| \|v\|)}{\sigma \sqrt{n}} \phi(y)  +  O \left( \frac{1}{n} \right),   \quad\\
(1-y_n^2) \phi(y_n) & =  (1-y^2) \phi(y) 
      +  O \left( \frac{1}{\sqrt{n}} \right),   \quad\\
 \phi(y_n) & =  \phi(y)   
     +  O \left( \frac{1}{\sqrt{n}} \right). 
\end{align*}
Substituting these expansions into \eqref{Edge-barfbarv} gives that uniformly in $f, v  \in K$ and $y \in \mathbb{R}$, 
\begin{align*}
& \bb P \left( \frac{\log \langle f, G_n v \rangle - n \lambda_{\mu} }{\sigma \sqrt{n}} \leq y \right)   \notag\\
& =  \Phi(y) + \frac{ m_3 }{ 6 \sigma^3 \sqrt{n}} (1-y^2) \phi(y) 
    - \frac{A(\bar f, \bar v) + B(\bar f, \bar v) + \log (\|f\| \|v\|) }{ \sigma \sqrt{n} } \phi(y)   
    +  o \left( \frac{ 1 }{\sqrt{n}} \right). 
\end{align*}
By \eqref{Def_A_fv}, we have $A(\bar f, \bar v) = A(f, v) - \log (\|f\| \|v\|)$ and $B(\bar f, \bar v) = B(f, v)$ for any $f, v  \in K$.  
This concludes the proof of \eqref{Edgeworth-f-v}. 

\textit{Step 3.}
We deduce Theorem \ref{MainThm_Edgeworth_01} from \eqref{Edgeworth-f-v}.  
By \eqref{func-phi-001} and \eqref{Def_A_fv}, we have that for any $f, v  \in K$,  
\begin{align*}
B(f,v) & = \int_{\bb G} b(f, g \cdot v) \mu(dg)
 =  \int_{\bb G}  \lim_{n \to \infty}
\mathbb{E} \left( \log \frac{\langle f, G_n (g \cdot v) \rangle}{\langle f, g \cdot v \rangle} - n \lambda_{\mu} \right)  \mu(dg)  \notag\\
& = \int_{\bb G}  \lim_{n \to \infty}
\mathbb{E} \left( \log \frac{\langle f, G_n g v \rangle}{\langle f, g v \rangle} - n \lambda_{\mu} \right)  \mu(dg)  \notag\\
& = \lim_{n \to \infty}
\mathbb{E} \left( \log \frac{\langle f, G_{n+1} v \rangle}{\langle f, g_1 v \rangle} - (n+1) \lambda_{\mu} \right)  + \lambda_{\mu}  \notag\\
& = \lim_{n \to \infty}
\mathbb{E} \left( \log \| G_{n+1} v \| - (n+1) \lambda_{\mu} \right)  - \bb E \left( \log \langle f, g_1 v \rangle - \lambda_{\mu} \right)  \notag\\
& \quad  +  \lim_{n \to \infty}  \bb E \left( \log \langle f, G_{n+1} \cdot v \rangle \right)  \notag\\
& = b(v) - A(f,v) + \lim_{n \to \infty}  \bb E \left( \log \langle f, G_{n+1} \cdot v \rangle \right),
\end{align*}
where in the last equality we used \eqref{drift-b001} and \eqref{Def_A_fv}. 
Since the function $v \mapsto \log \langle f,  g \cdot v \rangle$ is H\"{o}lder continuous on $K$ for any $g \in \Gamma_{\mu}$,
and the Markov chain $(G_n \cdot v)_{n \geq 0}$ has a unique invariant probability measure $\nu$, 
we have that for $f \in K$, 
\begin{align*}
\lim_{n \to \infty}  \bb E \left( \log \langle f, G_{n+1} \cdot v \rangle \right)
=   \lim_{n \to \infty}  \int_{\bb G} \bb E \left( \log \langle f, G_{n} \cdot  (g \cdot v) \rangle \right)  \mu(dg)
=   \int_{ \mathcal{S} }  \log \langle f, v \rangle \nu(dv) = d(f). 
\end{align*}
Using \eqref{MyEdgExp_hh}, we complete the proof of Theorem \ref{MainThm_Edgeworth_01}. 
\end{proof}

\subsection{Proof of Theorem \ref{Thm-BerryEsseen}}

By the Berry-Esseen inequality, there exists a constant $c>0$ such that for any $T>0$ and $v \in \mathcal{S}$,
\begin{align*}
& \sup_{y \in \bb R}
  \left| \bb P \left( \frac{\log \|G_n v \| - n \lambda_{\mu} }{\sigma \sqrt{n}} \leq y \right)  - \Phi(y)  \right|  
 \leq \frac{1}{\pi} \int_{-T}^T  \Bigg|  \frac{ R^n_{\frac{it}{\sigma \sqrt{n}}} 1(v) - e^{- \frac{t^2}{2}} }{t} \Bigg| dt 
+ \frac{c}{T}.   
\end{align*}
Below we shall choose $T = \delta \sqrt{n}$ and take $\delta >0$ to be sufficiently small. 
By Lemma \ref{transfer operator_Pit}, we have 
\begin{align*}
R^n_{\frac{it}{\sigma \sqrt{n}}} 1(v) - e^{- \frac{t^2}{2}}
=  \lambda^n_{ \frac{it}{\sigma \sqrt{n}} } \Pi_{ \frac{it}{\sigma \sqrt{n}} } 1(v) + N_{ \frac{it}{\sigma \sqrt{n}} }^n 1(v)  
   - e^{- \frac{t^2}{2}}   
=: \sum_{j=1}^3 I_j(t),
\end{align*}
where 
\begin{align*}
& I_1(t) =  \lambda^n_{ \frac{it}{\sigma \sqrt{n}} }  -  e^{- \frac{t^2}{2}}
  =  e^{- \frac{t^2}{2}} \bigg( e^{ n \log \lambda_{ \frac{it}{\sigma \sqrt{n}} } + \frac{t^2}{2} } - 1 \bigg),   \notag\\
& I_2(t) = \lambda^n_{ \frac{it}{\sigma \sqrt{n}} } \left[ \Pi_{ \frac{it}{\sigma \sqrt{n}} } 1(v) - 1  \right],   \notag\\
\quad
& I_3(t) = N_{ \frac{it}{\sigma \sqrt{n}} }^n 1(v). 
\end{align*}
Let $h(t) = n \log \lambda_{ \frac{it}{\sigma \sqrt{n}} } + \frac{t^2}{2}$.
By Lemma \ref{transfer operator_Pit}, we have $h(0) = h'(0) = h''(0) = 0$
and $h'''(0) = - \frac{i m_3}{\sigma^3 \sqrt{n}}$, so that 
\begin{align*}
h(t) = - \frac{i m_3}{6 \sigma^3 \sqrt{n}} t^3 + o \left( \frac{t^3}{\sqrt{n}} \right). 
\end{align*}
Using the inequality $|e^z - 1| \leq |z| e^{|z|}$ for $z \in \bb C$ 
and taking $\delta >0$ to be sufficiently small, we get that for any $t \in [- \delta \sqrt{n}, \delta \sqrt{n}]$, 
\begin{align}\label{Boind-I1t}
|I_1(t)| \leq  e^{- \frac{t^2}{2}} \frac{c|t|^3}{\sqrt{n}} e^{\frac{t^2}{4}} = \frac{c|t|^3}{\sqrt{n}} e^{-\frac{t^2}{4}}, 
\end{align}
so that
\begin{align}\label{Int-Bound-I1t}
\int_{- \delta \sqrt{n}}^{\delta \sqrt{n}}  \left|  \frac{ I_1(t) }{t} \right| dt
\leq  \frac{c}{\sqrt{n}}  \int_{- \delta \sqrt{n}}^{\delta \sqrt{n}}  |t|^3 e^{-\frac{t^2}{4}}  dt
\leq \frac{c'}{\sqrt{n}}. 
\end{align}
For $I_2(t)$,  from \eqref{Boind-I1t} we get
\begin{align}\label{Bound-exp-t2}
\left|  \lambda^n_{ \frac{it}{\sigma \sqrt{n}} } \right|
\leq   \frac{c|t|^3}{\sqrt{n}} e^{-\frac{t^2}{4}} +  e^{- \frac{t^2}{2}}
\leq   c' e^{- \frac{t^2}{8}}. 
\end{align}
By Lemma \ref{transfer operator_Pit}, there exists a constant $c>0$ such that for any $t \in [- \delta \sqrt{n}, \delta \sqrt{n}]$,
\begin{align*}
\left| \Pi_{ \frac{it}{\sigma \sqrt{n}} } 1(v) - 1  \right|
\leq  c \frac{|t|}{\sigma \sqrt{n}}, 
\end{align*}
which, together with \eqref{Bound-exp-t2}, implies that
\begin{align}\label{Int-Bound-I2t}
\int_{- \delta \sqrt{n}}^{\delta \sqrt{n}}  \left|  \frac{ I_2(t) }{t} \right| dt
\leq  \frac{c}{\sqrt{n}}  \int_{- \delta \sqrt{n}}^{\delta \sqrt{n}}  |t| e^{-\frac{t^2}{8}}  dt
\leq \frac{c'}{\sqrt{n}}. 
\end{align}
For $I_3(t)$, using again Lemma \ref{transfer operator_Pit} and the fact that $N_0 1 = 0$, we deduce that 
\begin{align}\label{Int-Bound-I3t}
\int_{- \delta \sqrt{n}}^{\delta \sqrt{n}}  \left|  \frac{ I_3(t) }{t} \right| dt
\leq C e^{-cn}.  
\end{align}
Combining \eqref{Int-Bound-I1t}, \eqref{Int-Bound-I2t} and \eqref{Int-Bound-I3t}, we get
\begin{align}\label{Berry-Esseen-cocy}
& \sup_{y \in \bb R}
  \left| \bb P \left( \frac{\log \|G_n v \| - n \lambda_{\mu} }{\sigma \sqrt{n}} \leq y \right)  - \Phi(y)  \right|  
 \leq \frac{c}{\sqrt{n}}.   
\end{align}
This, together with Lemma \ref{lem equiv Kesten} and the fact that $\langle f, G_n v \rangle \leq \|G_n v \|$, 
implies that there exists a constant $c>0$ such that for any $n \geq 1$, $f, v  \in \mathcal S$ and $y \in \mathbb{R}$,
\begin{align*}
\left| \bb P \left( \frac{\log \langle f, G_n v \rangle - n \lambda_{\mu} }{\sigma \sqrt{n}} \leq y \right)  - \Phi(y)  \right|
\leq  \frac{c}{\sqrt{n}}, 
\end{align*}
which yields \eqref{Main-BerryEsseen-Coeff}. 

By \cite[Lemma 4.5]{BDGM14},  
there exists a constant $c>0$ such that $\|G_n\| \leq  c \|G_n v\|$ for any $v \in \mathcal S_{\epsilon}$ and $n \geq 1$. 
This, together with the fact $\|G_n v\| \leq \|G_n\|$ and \eqref{Berry-Esseen-cocy}, yields the Berry-Esseen theorem for $\|G_n\|$. 
Since $\langle e_1, G_n e_1 \rangle \leq \rho(G_n) \leq \|G_n\|$, 
the Berry-Esseen theorem for $\rho(G_n)$ follows from that for  $\|G_n\|$ and \eqref{Main-BerryEsseen-Coeff}.

\section{Proofs of upper and lower large deviations} \label{sec proof scalar product}

The goal of this section is to establish Theorem \ref{Thm-Petrov-Uppertail}. 
Using the spectral gap theory for the cocycle $\sigma_f (G_n, v)$ developed in Section \ref{sec:spec gap entries}, 
together with the Edgeworth expansion for the couple $(G_n \cdot v, \sigma_f (G_n, v))$
with a target function $\varphi$ on the Markov chain $(G_n \cdot v)$ under the changed measure $\bb Q^{v}_{s,f}$, 
we first prove the following large deviation expansions for $\sigma_f (G_n, v)$ 
with $f \in \mathcal{S}$ and $v \in \mathcal{S}_{\epsilon}$.

\begin{proposition}\label{Thm_Cocycle_01} 

\noindent 1. 
Assume \ref{Condi-FK-weak} and \ref{Condi-NonLattice}. 
Let $J^+ \subseteq (I_\mu^+)^\circ$ be a compact set. 
Then, for any $0 < \gamma \leq \min_{t \in J^+} \{ t, 1\}$, 
we have, as $n \to \infty$, 
uniformly in $s \in J^+$, $f \in \mathcal{S}$, $v \in \mathcal{S}_{\epsilon}$ and $\varphi \in \mathscr B_{\gamma}$, 
with $q = \Lambda'(s)$, 
\begin{align}\label{BahaRao_Version_Coeff}
& \mathbb{E} \Big[ \varphi(G_n \!\cdot\! v) 
      \mathds 1_{ \big\{ \log \frac{\langle f,  G_n v \rangle}{\langle f, v \rangle}  \geq  nq \big\} }  \Big]   \notag\\
& = \frac{ r_{s, f}(v) }{ s\sigma_{s}\sqrt{2\pi n} } e^{- n\Lambda^{*}(q)}  
 \left[ \pi_{s,f}(\varphi  r_{s,f}^{-1}) + \pi_{s,f}(\varphi  r_{s,f}^{-1}) o(1)  +  \|\varphi\|_{\gamma} O \left( \frac{1}{\sqrt{n}} \right)  \right]. 
\end{align}

\noindent 2. 
Assume \ref{Condi-FK} and \ref{Condi-NonLattice}.
Let $J^- \subseteq (I_\mu^-)^\circ$ be a compact set. 
Then, 
for any $0 < \gamma \leq \min_{t \in J^-} \{ |t|, 1\}$, 
we have, as $n \to \infty$, 
uniformly in $s \in J^-$, $f \in \mathcal{S}$, $v \in \mathcal{S}_{\epsilon}$ and $\varphi \in \mathscr B_{\gamma}$, 
with $q = \Lambda'(s)$, 
\begin{align}\label{BahaRao_Version_Coeff-Neg}
& \mathbb{E} \Big[ \varphi(G_n \!\cdot\! v) 
      \mathds 1_{ \big\{ \log \frac{\langle f,  G_n v \rangle}{\langle f, v \rangle}  \leq  nq \big\} }  \Big]   \notag\\
& = \frac{ r_{s, f}(v) }{ -s\sigma_{s}\sqrt{2\pi n} } e^{- n\Lambda^{*}(q)}  
 \left[ \pi_{s,f}(\varphi  r_{s,f}^{-1}) + \pi_{s,f}(\varphi  r_{s,f}^{-1}) o(1)  +  \|\varphi\|_{\gamma} O \left( \frac{1}{\sqrt{n}} \right)  \right]. 
\end{align}
\end{proposition}

\begin{proof}
We only prove \eqref{BahaRao_Version_Coeff} since the proof of \eqref{BahaRao_Version_Coeff-Neg}
can be carried out in the same way. 
By a change of measure formula \eqref{change measure equ coeff2}
and the fact that $\kappa(s)^n e^{-snq} = e^{- n\Lambda^{*}(q)}$, we have
\begin{align*} 
 I_n(f, v) &:= \mathbb{E} \Big[ \varphi(G_n \!\cdot\! v) 
      \mathds 1_{ \big\{ \log \frac{\langle f,  G_n v \rangle}{\langle f, v \rangle} - nq \geq 0 \big\} }  \Big]  \notag\\
& = \kappa(s)^n r_{s,f}(v)
  \mathbb{E}_{\bb Q^{v}_{s,f}} \Big[  (\varphi r_{s,f}^{-1})(G_n \!\cdot\! v)
       e^{-s \log \frac{\langle f,  G_n v \rangle}{\langle f, v \rangle}  }
       \mathds 1_{ \big\{ \log \frac{\langle f,  G_n v \rangle}{\langle f, v \rangle} - nq \geq 0 \big\} } \Big]  \notag\\
& =  e^{- n\Lambda^{*}(q)}  r_{s,f}(v)
  \mathbb{E}_{\bb Q^{v}_{s,f}} \Big[  (\varphi r_{s,f}^{-1})(G_n \!\cdot\! v)
       e^{-s \big( \log \frac{\langle f,  G_n v \rangle}{\langle f, v \rangle} - nq \big) }
       \mathds 1_{ \big\{ \log \frac{\langle f,  G_n v \rangle}{\langle f, v \rangle} - nq \geq 0 \big\} } \Big].  
\end{align*}
For brevity, we denote $a_n = s \sigma_s \sqrt{n}$ and 
\begin{align*}
T_n = \frac{ \log \frac{\langle f,  G_n v \rangle}{\langle f, v \rangle} - nq }{ \sigma_s \sqrt{n} }, 
\quad
F_n(y) = \bb Q^{v}_{s,f} \left[  (\varphi r_{s,f}^{-1})(G_n \!\cdot\! v)  \mathds 1_{ \{ T_n \leq y \} }   \right],  
\end{align*}
it follows that 
\begin{align*}
I_n(f, v) & =  e^{- n\Lambda^{*}(q)}  r_{s,f}(v)
  \mathbb{E}_{\bb Q^{v}_{s,f}} \Big[  (\varphi r_{s,f}^{-1})(G_n \!\cdot\! v)
       e^{-s \sigma_s \sqrt{n} T_n }
       \mathds 1_{ \{ T_n \geq 0 \} } \Big]    \notag\\
  & =  e^{- n\Lambda^{*}(q)}  r_{s,f}(v)
  \mathbb{E}_{\bb Q^{v}_{s,f}} \left[  (\varphi r_{s,f}^{-1})(G_n \!\cdot\! v)
       e^{- a_n T_n }
       \mathds 1_{ \{ T_n \geq 0 \} } \right]    \notag\\
 & =  e^{- n\Lambda^{*}(q)}  r_{s,f}(v)  \int_0^{\infty}  e^{- a_n y} d F_n(y).  
\end{align*}
Using integration by parts and a change of variable, we obtain
\begin{align*}
I_n(f, v) 
& = e^{- n\Lambda^{*}(q)}  r_{s,f}(v) \left[ F_n(0) +  a_n \int_0^{\infty}  e^{- a_n y}  F_n(y) dy  \right]
  \notag\\
& =    e^{- n\Lambda^{*}(q)}  r_{s,f}(v)   \int_0^{\infty}  a_n  e^{- a_n y}  \left( F_n(y) - F_n(0) \right)  dy   \notag\\
& =   e^{- n\Lambda^{*}(q)}  r_{s,f}(v)   \int_0^{\infty}   e^{- y}  \left[ F_n\left( \frac{y}{a_n} \right) - F_n(0) \right]  dy,  
\end{align*}
so that 
\begin{align}\label{LD-Decom-a}
\frac{ s\sigma_{s}\sqrt{2\pi n} }{ r_{s, f}(v) } e^{- n\Lambda^{*}(q)}  I_n(f, v) 
=  \sqrt{2\pi}  a_n   \int_0^{\infty}   e^{- y}  \left[ F_n\left( \frac{y}{a_n} \right) - F_n(0) \right]  dy. 
\end{align}
Using the spectral gap theory for the cocycle 
$\log \frac{\langle f, G_n v \rangle}{\langle f, v \rangle}$ developed in Section \ref{sec:spec gap entries} 
and proceeding in the same way as in the proof of \cite[Theorem 2.2]{XGL19b} for the couple $(G_n \cdot v, \log \|G_n v\|)$, 
one can check that, as $n \to \infty$, uniformly in $s \in J^+$,
 $f \in \mathcal{S}$,  $v \in \mathcal{S}_{\epsilon}$, $y \in \mathbb{R}$ and $\varphi \in \scr B_{\gamma}$,
\begin{align}\label{Thm_Edgeworth_Target_02}
&  \mathbb{E}_{\mathbb{Q}_{s, f}^v}
   \Big[  \varphi(G_n \cdot v) 
   \mathds{1}_{ \Big\{ \frac{\log \frac{\langle f, G_n v \rangle}{\langle f, v \rangle} 
                - n \Lambda'(s) }{\sigma_s \sqrt{n}} \leq y \Big\} } \Big]
   \nonumber \\
&  =  \pi_{s,f}(\varphi) \Big[  \Phi(y) + \frac{\Lambda'''(s)}{ 6 \sigma_s^3 \sqrt{n}} (1-y^2) \phi(y) \Big]
    - \frac{ b_{s, \varphi}(f,v) }{ \sigma_s \sqrt{n} } \phi(y)    
     +  \pi_{s,f}(\varphi)  o \Big( \frac{ 1 }{\sqrt{n}} \Big)  +  \lVert \varphi \rVert_{\gamma} O \Big( \frac{ 1 }{n}\Big), 
\end{align}
where $b_{s, \varphi}(f,v)$ is defined in \eqref{Def-bsvarphi-fv}. 
Applying \eqref{Thm_Edgeworth_Target_02},  we have, 
as $n \to \infty$, 
uniformly in $s \in J^+$, $f \in \mathcal{S}$, $v \in \mathcal{S}_{\epsilon}$ and $\varphi \in \mathscr B_{\gamma}$, 
\begin{align}\label{LD-Decom-b}
F_n\left( \frac{y}{a_n} \right) - F_n(0) = I_1(y) + I_2(y) + I_3(y) + I_4(y), 
\end{align}
where 
\begin{align*}
I_1(y) & = \pi_{s,f}(\varphi  r_{s,f}^{-1}) \left[  \Phi\left( \frac{y}{a_n} \right) - \Phi(0)  \right],   \notag\\
I_2(y) & = \pi_{s,f}(\varphi  r_{s,f}^{-1})   \frac{\Lambda'''(s)}{ 6 \sigma_s^3 \sqrt{n}}
   \left\{ \left[ 1- \left( \frac{y}{a_n} \right)^2  \right] \phi\left( \frac{y}{a_n} \right)  -  \phi(0) \right\},      \notag\\
I_3(y) & =  \frac{ b_{s, \varphi r_{s,f}^{-1}}(f,v) }{ \sigma_s \sqrt{n} }  \left[  \phi(0)  -  \phi\left( \frac{y}{a_n} \right) \right],   \notag\\
I_4(y) & = \pi_{s,f}(\varphi  r_{s,f}^{-1})  o \Big( \frac{ 1 }{\sqrt{n}} \Big)  
   +  \lVert \varphi  r_{s,f}^{-1} \rVert_{\gamma} O \Big( \frac{ 1 }{n}\Big). 
\end{align*}
For $I_1(y)$ and $I_2(y)$,  by Taylor's formula and Lebesgue's dominated convergence theorem, we get that 
as $n \to \infty$, 
uniformly in $s \in J^+$, $f \in \mathcal{S}$ and $\varphi \in \mathscr B_{\gamma}$, 
\begin{align}\label{LD-I1y}
\sqrt{2\pi}  a_n   \int_0^{\infty}   e^{- y}  I_1(y)  dy
& = \sqrt{2\pi}  \pi_{s,f}(\varphi  r_{s,f}^{-1})  \phi(0)  \int_0^{\infty}   e^{- y}  y  dy 
      +  \|\varphi\|_{\infty} O \left( \frac{1}{\sqrt{n}} \right)  \notag\\
& =  \pi_{s,f}(\varphi  r_{s,f}^{-1}) + \|\varphi\|_{\infty} O \left( \frac{1}{\sqrt{n}} \right)
\end{align}
and
\begin{align}\label{LD-I2y}
\sqrt{2\pi}  a_n   \int_0^{\infty}   e^{- y}  I_2(y)  dy
 =   \|\varphi\|_{\infty} O \left( \frac{1}{\sqrt{n}} \right). 
\end{align}
For $I_3(y)$, using Lemma \ref{Lem-Bs} and Taylor's formula, we get that 
as $n \to \infty$, 
uniformly in $s \in J^+$, $f \in \mathcal{S}$ and $\varphi \in \mathscr B_{\gamma}$, 
\begin{align}\label{LD-I3y}
\sqrt{2\pi}  a_n   \int_0^{\infty}   e^{- y}  I_3(y)  dy
 =   \|\varphi\|_{\gamma} O \left( \frac{1}{\sqrt{n}} \right). 
\end{align}
For $I_4(y)$, it is easy to see that as $n \to \infty$, 
uniformly in $s \in J^+$, $f \in \mathcal{S}$ and $\varphi \in \mathscr B_{\gamma}$, 
\begin{align}\label{LD-I4y}
\sqrt{2\pi}  a_n   \int_0^{\infty}   e^{- y}  I_4(y)  dy
= \pi_{s,f}(\varphi  r_{s,f}^{-1}) o(1)  +  \|\varphi\|_{\gamma} O \left( \frac{1}{\sqrt{n}} \right). 
\end{align}
Combining \eqref{LD-Decom-a}, \eqref{LD-Decom-b}, \eqref{LD-I1y}, \eqref{LD-I2y}, \eqref{LD-I3y} and \eqref{LD-I4y}
concludes the proof of \eqref{BahaRao_Version_Coeff}. 
\end{proof}


Using Proposition \ref{Thm_Cocycle_01}, we now prove Theorem \ref{Thm-Petrov-Uppertail}. 

\begin{proof}[Proof of Theorem \ref{Thm-Petrov-Uppertail}]
We only prove \eqref{Petrov-Upper} using \eqref{BahaRao_Version_Coeff}, since the proof of \eqref{Petrov-Lower}
can be similarly carried out using \eqref{BahaRao_Version_Coeff-Neg}. 
Let  $\epsilon \in (0,1)$ be from \eqref{equicon A4} of Lemma \ref{lem equiv Kesten}. 
It holds that $\langle f, v \rangle \geq \epsilon$ 
for any $f \in \mathcal{S}$ and $v \in \mathcal{S}_\epsilon$, and $g \cdot v \in \mathcal{S}_\epsilon$
for any $g \in \Gamma_{\mu}$ and $v \in \mathcal{S}$.
For $n \geq 2$, $v' \in \mathcal{S}_{\epsilon}$ and $l' \in \mathbb{R}$, we denote  
\begin{align} \label{ScalDefPn}
P_{n}(v', l')= \mathbb{E}
\Big[  \varphi(g_n \ldots g_2 \cdot v') 
    \mathds 1_{ \big\{ \log \frac{\langle f, g_n \ldots g_2 v' \rangle}{\langle f, v' \rangle} \geq (n-1)(q + l')  \big\} }  \Big].
\end{align}
Let
\begin{align*}
\ell_{f,v} (g) =  \frac{q - \log \langle f, g v \rangle}{n-1}. 
\end{align*}
Since for any fixed $g \in \supp \mu$, 
\begin{align} \label{ScalCondi 01a}
\mathbb{E} \Big[  \varphi(g_n \ldots g_2 g \cdot v) 
    \mathds 1_{ \big\{  \log \langle f, g_n \ldots g_2 gv \rangle \geq nq   \big\}  }  \Big]
=  P_{n} \left(g \cdot v, \ell_{f,v}(g) \right),
\end{align}
taking conditional expectation with respect to $g_1$, we have
\begin{align} \label{ScalCondi 01}
\mathbb{E} \Big[  \varphi(g_n \ldots g_1 \cdot v) 
    \mathds 1_{ \big\{  \log \langle f, g_n \ldots g_1 v \rangle \geq nq  \big\}  }  \Big]
 = \mathbb E P_{n} \left( g_1 \cdot v,  \ell_{f,v} (g_1) \right). 
\end{align}
Note that Proposition \ref{Thm_Cocycle_01} holds uniformly in $s$, allowing us to consider a vanishing perturbation $\ell_{f,v} (g)$ on $q$. 
Specifically, for any fixed $g \in \supp\mu$,
applying Proposition \ref{Thm_Cocycle_01} to $P_{n}(g_1 \cdot v,  \ell_{f,v} (g) )$ 
and using the fact that $\pi_{s,f}(\varphi  r_{s,f}^{-1}) = \nu_{s,f}(\varphi)$, 
we obtain that, as $n \to \infty,$ 
uniformly in $s \in J^+$, $f \in \mathcal{S},$ $v \in \mathcal{S}$ and $\varphi \in \mathscr B_{\gamma}$, 
\begin{align*}
P_{n} \left(g \cdot v, \ell_{f,v}(g) \right)
 =  e^{- (n-1) \Lambda^{*}(q + \ell_{f,v}(g))}
    \frac{ r_{s, f}(g \cdot v) }{ s\sigma_{s}\sqrt{2\pi (n-1)} }   
   \big[ \nu_{s,f}(\varphi) + \|\varphi\|_{\gamma}  o(1)  \big]. 
\end{align*}
As in \cite[Lemma 4.1]{XGL19a}, 
for any $s \in (I_{\mu}^+ \cup  I_{\mu}^-)^{\circ}$ and $q = \Lambda'(s)$, 
the rate function $\Lambda^*(q+l)$ has the following expansion with respect to $l$ in a small neighborhood of $0$: 
\begin{align} \label{Def Jsl}
\Lambda^*(q+l)
   = \Lambda^{*}(q) + sl + \frac{l^2}{2 \sigma_s^2} - \frac{l^3}{\sigma_s^3} \zeta_s\Big(\frac{l}{\sigma_s}\Big), 
\end{align}
where $\zeta_s(t)$ is called the Cram\'{e}r series (cf.\ \cite{Pet75, XGL19b}) of $\Lambda(s)$ given by
\begin{align*}
\zeta_s (t) = \frac{\gamma_{s,3} }{ 6 \gamma_{s,2}^{3/2} }  
  +  \frac{ \gamma_{s,4} \gamma_{s,2} - 3 \gamma_{s,3}^2 }{ 24 \gamma_{s,2}^3 } t
  +  \frac{\gamma_{s,5} \gamma_{s,2}^2 - 10 \gamma_{s,4} \gamma_{s,3} \gamma_{s,2} + 15 \gamma_{s,3}^3 }{ 120 \gamma_{s,2}^{9/2} } t^2
  +  \ldots,
\end{align*}
with $\gamma_{s,k} = \Lambda^{(k)} (s)$.
Since $\ell_{f,v}(g) \to 0$ as $n \to \infty$,  and $\Lambda^*(q) = sq - \Lambda(s)$, 
we get that, for any $g \in \supp \mu$, as $n \to \infty$, 
uniformly in $s \in J^+$ and $f, v  \in \mathcal{S}$,  
\begin{align*}
e^{-(n-1) \Lambda^*(q + \ell_{f,v}(g))} 
= e^{-n \Lambda^*(q) - \Lambda(s) + s \log \langle f, gv \rangle} [1 + o(1)]. 
\end{align*}
Together with the previous display, 
this implies
\begin{align*}
P_{n} \left( g \cdot v, \ell_{f,v}(g) \right)
& =  \frac{ e^{- n \Lambda^{*}(q)} }{s\sigma_{s} \sqrt{2\pi n}}
    \frac{ \langle f, gv \rangle^s  r_{s, f}(g \cdot v) }{\kappa(s)}  
  \big[ \nu_{s,f}(\varphi) + \|\varphi\|_{\gamma}  o(1)  \big]
    \notag\\
& = \frac{ e^{- n \Lambda^{*}(q)} }{s\sigma_{s} \sqrt{2\pi n}}
     \frac{ \|gv\|^s r_s(g \cdot v) }{\kappa(s)}  
   \big[ \nu_{s,f}(\varphi) + \|\varphi\|_{\gamma}  o(1)  \big],
\end{align*}
where in the last equality we used the definition of $r_{s,f}$ (see Proposition \ref{change of measure coeffi 01}). 
Denote
\begin{align*} 
I_n(g) = s \sigma_{s} \sqrt{2\pi n}  \,  e^{n\Lambda^{*}(q) }  P_{n} \left( g \cdot v, \ell_{f,v}(g) \right). 
\end{align*}
To apply the Lebesgue dominated convergence theorem, 
we shall find an integrable dominating function for the sequence $(I_n(g))_{n\geq 1}.$ 
Since $\log \langle f, g_n \ldots g_2 g v \rangle \leq  \log\| g_n \ldots g_2 g v \|$, 
we use \eqref{ScalCondi 01a} and \cite[Theorem 2.1]{XGL19a} to get that 
there exist constants $c,c',c'' >0$ such that for any $g \in \Gamma_{\mu}$, 
$s \in J^+$, $f, v \in \mathcal{S}$ and $n \geq 1$, with $q = \Lambda'(s)$ and $q' = \frac{nq}{n-1}$, 
\begin{align} \label{ScalDomiFunc}
I_n(g) & \leq  c \|\varphi\|_{\infty}  \sqrt{n}  \,  e^{n\Lambda^{*}(q) }  
   \bb P \left( \log \| g_n \ldots g_2 gv \| \geq nq \right)  \notag\\
&  \leq  c' \|\varphi\|_{\infty}  \sqrt{n-1}  \,  e^{(n-1)\Lambda^{*}(q) }  
   \bb P \left( \log \| g_n \ldots g_2 gv \| \geq (n-1) q' \right)  \notag\\
& \leq c'' \|\varphi\|_{\infty}  r_s(gv), 
\end{align}
where $r_s$ is defined by Proposition \ref{change-measure-neg01} (the definition of $r_s$ can be extended to $\bb R_+^d \setminus \{0\}$). 
Since  
\begin{align*}
\int_{\bb G} r_s(gv) \mu(dg) = \int_{\bb G} \|gv\|^s r_s(g \cdot v) \mu(dg)
\leq c  \int_{\bb G} \|g\|^s  \mu(dg) < \infty 
\end{align*}
and 
\begin{align*}
\frac{ 1 }{\kappa(s)}  \int_{\bb G}  \|gv\|^s r_s(g \cdot v) \mu(dg) 
=  \frac{ 1 }{\kappa(s)} P_s r_s(v) =  \frac{ 1 }{\kappa(s)} \kappa(s) r_s(v) = r_s(v), 
\end{align*}
we are allowed to apply the Lebesgue dominated convergence theorem to obtain
that, as $n\to \infty$, uniformly in $s \in J^+$, $f, v  \in \mathcal{S}$ and $\varphi \in \mathscr B_{\gamma}$, 
\begin{align} \label{Scallimit 01}
 \mathbb{E} \Big[  \varphi(g_n \ldots g_1 \cdot v) 
    \mathds 1_{ \big\{  \log \langle f, g_n \ldots g_1 v \rangle \geq nq  \big\}  }  \Big]   
 = \frac{ e^{- n \Lambda^{*}(q)} }{s\sigma_{s} \sqrt{2\pi n}}
  r_s(v)  \big[ \nu_{s,f}(\varphi) + \|\varphi\|_{\gamma}  o(1)  \big].  
\end{align}
Since $\nu_{s,f}(\varphi) = \frac{\nu_s(\langle f, \cdot \rangle^s \varphi)}{\nu_s(r_s)}$ (cf.\ \eqref{eigenfunc_Psf}), 
from \eqref{Scallimit 01} it follows that, as $n \to \infty$, 
uniformly in $s \in J^+$, $f, v  \in \mathcal S$ and $\varphi \in \scr B_{\gamma}$, 
\begin{align}\label{Petrov-Upper-aaa}
& \mathbb{E} \left[ \varphi(G_n \cdot v) \mathds 1_{ \left\{ \log \langle f,  G_n v \rangle \geq nq \right\} }  \right]  
   \notag\\
& = \frac{r_s(v)}{\nu_s(r_s)} \frac{ \exp(- n\Lambda^{*}(q)) }{ s\sigma_{s}\sqrt{2\pi n} }  
    \left[ \int_{\mathcal{S}} \varphi(u) \langle f, u \rangle^s \nu_s(du) + \| \varphi \|_{\gamma} o(1) \right]. 
\end{align}
Similarly to \eqref{Notation-fv-yn}, we denote 
\begin{align*}
\bar f = \frac{f}{\|f\|},  \quad  \bar v = \frac{v}{\|v\|},  \quad l' = - \frac{\log (\|f\| \|v\|)}{n}. 
\end{align*}
We use \eqref{Petrov-Upper-aaa} to get that uniformly in $s \in J^+$, $f, v  \in K$
 and $\varphi \in \scr B_{\gamma}$, 
\begin{align*}
& \mathbb{E} \left[ \varphi(G_n \cdot v) \mathds 1_{ \left\{ \log \langle f,  G_n v \rangle \geq n (q+l) \right\} }  \right]  
  =  \mathbb{E} \left[ \varphi(G_n \cdot \bar v) \mathds 1_{ \left\{ \log \langle \bar f,  G_n \bar v \rangle \geq n (q+l') \right\} }  \right]  
   \notag\\
& = \frac{r_s(\bar v)}{\nu_s(r_s)} \frac{ \exp(- n\Lambda^{*}(q+l')) }{ s\sigma_{s}\sqrt{2\pi n} }  
    \left[ \int_{\mathcal{S}} \varphi(u) \langle \bar f, u \rangle^s \nu_s(du) + \| \varphi \|_{\gamma} o(1) \right]. 
\end{align*}
By \eqref{Def Jsl}, it holds uniformly in $s \in J^+$ and $f, v  \in K$, 
\begin{align*}
\Lambda^*(q+l')
   = \Lambda^{*}(q) - s \frac{\log (\|f\| \|v\|)}{n} + o \left( \frac{1}{n} \right), 
\end{align*}
so that 
\begin{align*}
e^{- n\Lambda^{*}(q+l')} = \|f\|^s \|v\|^s  e^{- n\Lambda^{*}(q)}  [1 + o(1)]. 
\end{align*}
Since 
\begin{align*}
\|v\|^s  r_s(\bar v) = r_s(v)  \quad \mbox{and}  \quad 
\|f\|^s \int_{\mathcal{S}} \varphi(u) \langle \bar f, u \rangle^s \nu_s(du) = \int_{\mathcal{S}} \varphi(u) \langle f, u \rangle^s \nu_s(du), 
\end{align*}
we conclude the proof of \eqref{Petrov-Upper}. 
\end{proof}

\section{Proof of Theorems \ref{Thm_LLT_03}, \ref{Thm-Petrov-Norm-cocycle} and \ref{Thm_LDSpectPosi}}

\subsection{Proof of Theorem \ref{Thm_LLT_03}}
We first establish \eqref{LLT-001}. 
It suffices to prove that \eqref{LLT-001} holds uniformly in $f, v  \in \mathcal S$, 
$|l| = o(\frac{1}{\sqrt{n}})$ and $\varphi \in \scr B_{1}$.  
By Lemma \ref{lem equiv Kesten}, 
condition \ref{Condi-FK-weak} implies that $g \cdot v \in \mathcal{S}_\epsilon$
for any $g \in \Gamma_{\mu}$ and $v \in \mathcal{S}$.  
For any $f, v \in \mathcal S$, $|l| = o(\frac{1}{\sqrt{n}})$, $\varphi \in \scr B_{1}$, $n \geq 2$ 
 and any fixed $g \in \Gamma_{\mu}$,  it holds that
\begin{align} \label{LLT_ScalCondi_01}
\bb E \Big[  \varphi( (g_n \ldots g_2 g) \cdot v ) 
      \mathds 1_{ \left\{ 
     \log \langle f, g_n \ldots g_2 g v \rangle - n \lambda_{\mu} \in [a_1, a_2] + \sigma n l
     \right\} }
         \Big]  
=  P_{n} \left(g \cdot v, \ell_{f,v}(g) \right),
\end{align}
where 
\begin{align*}
\ell_{f,v} (g) = \frac{n l}{n-1}  +  \frac{ \lambda_{\mu} - \log \langle f, g v \rangle}{\sigma(n-1)} 
\end{align*}
and, with $f \in \mathcal S$, $v' \in \mathcal{S}_{\epsilon}$ and $l' \in \bb R$, 
\begin{align} \label{LLT_ScalDefPn}
P_{n}(v', l')= \mathbb{E} 
\Big[  \varphi( (g_n \ldots g_2) \!\cdot\! v') 
      \mathds 1_{ \left\{ 
      \log \frac{\langle f, g_n \ldots g_2 v' \rangle}{\langle f, v' \rangle} - (n-1) \lambda_{\mu} \in [a_1, a_2] + \sigma (n-1) l'
       \right\} }
         \Big].
\end{align}
Taking conditional expectation with respect to $g_1$, we have
\begin{align} \label{LLT_ScalCondi_02}
\mathbb{E} \left[ \varphi(G_n \!\cdot\! v) 
      \mathds 1_{ \left\{ \log \langle f,  G_n v \rangle - n \lambda_{\mu} \in   [a_1, a_2] + \sigma n l \right\} }  \right] 
= \mathbb E P_{n} \left( g_1 \cdot v,  \ell_{f,v} (g_1) \right). 
\end{align}
Under \ref{Condi-FK-weak}, \ref{Condi-NonLattice} 
and $\int_{\bb G} \log^2 N(g) \mu(dg) < \infty$, 
one can apply Lemma \ref{transfer operator_Pit-f} and follow the proof of \cite[Theorem 2.4]{XGL19b} to show that, as $n \to \infty$, 
uniformly in $f \in \mathcal{S}$, $v \in \mathcal{S}_{\epsilon}$, $|l| = O(\frac{1}{\sqrt{n}})$ and $\varphi \in \scr B_{1}$, 
\begin{align}\label{Prop_LLT_01}
\mathbb{E} \left[ \varphi(G_n \!\cdot\! v) 
      \mathds 1_{ \left\{
       \log \frac{\langle f, G_n v \rangle}{\langle f, v \rangle} - n \lambda_{\mu}  \in [a_1, a_2] + \sigma nl 
       \right\} }
       \right]
= \frac{ a_2 - a_1 }{\sigma \sqrt{2\pi n} } e^{- \frac{n l^2}{2}} 
   \big[ \nu(\varphi)  + \| \varphi \|_{\gamma}  o(1) \big].
\end{align}
By \eqref{Prop_LLT_01}, we get that for any fixed $g \in \Gamma_{\mu}$, 
as $n \to \infty$, uniformly in $f, v  \in \mathcal{S}$, $|l| = o(\frac{1}{\sqrt{n}})$ and $\varphi \in \scr B_{1}$, 
\begin{align}\label{LLT-Pn-g}
P_{n} \left( g \cdot v,  \ell_{f,v} (g) \right)
= \frac{ a_2 - a_1 }{\sigma \sqrt{2\pi (n-1)} } e^{- \frac{n-1}{2} \ell^2_{f,v}(g) } \big[ \nu(\varphi)  + \| \varphi \|_{\gamma}  o(1) \big].
\end{align}
Note that $\frac{1}{\sqrt{n-1}} = \frac{1}{\sqrt{n}} [1 + O(\frac{1}{n})]$. 
Using the inequality $1 \geq e^{-x} \geq 1 - x$ for $x \geq 0$
and the fact that $|\log \langle f, gv \rangle| \leq \log N(g)$ and $\int_{\bb G} \log^2 N(g) \mu(dg) < \infty$, 
we get that as $n \to \infty$, 
uniformly in $f, v  \in \mathcal{S}$ and $|l| = o(\frac{1}{\sqrt{n}})$, 
\begin{align}\label{Bound-Inte-lg}
1 \geq \int_{\bb G} e^{- \frac{n-1}{2} \ell^2_{f,v}(g) } \mu(dg)
& =  \int_{\bb G}  e^{ - \frac{( \sigma n l + \lambda_{\mu} - \log \langle f, gv \rangle)^2}{2 \sigma^2 (n-1)}} \mu(dg)  \notag\\
& \geq  1 -  \int_{\bb G} \frac{( \sigma n l + \lambda_{\mu} - \log \langle f, gv \rangle)^2}{2 \sigma^2 (n-1)} \mu(dg)
 = 1 - o(1). 
\end{align}
In a similar way as in the proof of \cite[Theorem 2.4]{XGL19b}, 
one can apply Lemma \ref{transfer operator_Pit-f} to show the following uniform upper bound:  
under \ref{Condi-FK-weak}, \ref{Condi-NonLattice} and $\int_{\bb G} \log^2 N(g) \mu(dg) < \infty$, 
there exists a constant $c > 0$ such that for any $f \in \mathcal{S}$, $v \in \mathcal{S}_{\epsilon}$ and $n \geq 1$,  
\begin{align}\label{Prop_LLT_UpperBound}
\sup_{l \in \bb R}
\mathbb{E} \left[ \varphi(G_n \!\cdot\! v) 
      \mathds 1_{ \left\{ \log \frac{\langle f, G_n v \rangle}{\langle f, v \rangle} - n \lambda_{\mu} \in [a_1, a_2] + l  
      \right\}  }
      \right]
\leq  c \|\varphi\|_{\infty} \frac{ a_2 - a_1 }{\sqrt{n} }.
\end{align}
By \eqref{Prop_LLT_UpperBound}, 
the dominating function for $P_{n} \left(g \cdot v, \ell_{f,v}(g) \right)$ is given as follows: 
\begin{align*}
\sup_{g \in \Gamma_{\mu}} P_{n} \left(g \cdot v, \ell_{f,v}(g) \right)
\leq c \frac{ a_2 - a_1 }{\sqrt{n} }. 
\end{align*}
Therefore, from \eqref{LLT_ScalCondi_02},  using \eqref{LLT-Pn-g}, \eqref{Bound-Inte-lg} and the Lebesgue dominated convergence theorem, 
we obtain \eqref{LLT-001}. 

The proof of \eqref{LLT-MD} can be conducted similarly to that of \eqref{LLT-001},
so we only outline the main differences. 
In a manner analogous to \eqref{Prop_LLT_01} and \eqref{Prop_LLT_UpperBound}, 
under \ref{Condi-FK-weak}, \ref{Condi-NonLattice} 
and $\int_{\bb G} N(g)^{\eta} \mu(dg) < \infty$ for some constant $\eta>0$, 
one can apply Propositions \ref{perturbation thm 02-f}, \ref{Sca expdecay pertur-f} 
and follow the proof of \cite[Theorem 2.4]{XGL19b} to show that, as $n \to \infty$, 
uniformly in $f \in \mathcal{S}$, $v \in \mathcal{S}_{\epsilon}$, $|l| \leq l_n$ and $\varphi \in \scr B_{1}$, 
\begin{align*}
\mathbb{E} \left[ \varphi(G_n \!\cdot\! v) 
      \mathds 1_{ \left\{
       \log \frac{\langle f, G_n v \rangle}{\langle f, v \rangle} - n \lambda_{\mu} \in [a_1, a_2] + \sigma nl 
       \right\} }
       \right]
= \frac{ a_2 - a_1 }{\sigma \sqrt{2\pi n} }   e^{ -\frac{nl^2}{2} +   n l^3 \zeta (l) }
   \big[ \nu(\varphi)  + \| \varphi \|_{\gamma}  o(1) \big], 
\end{align*}
and there exist constants $c, \eta > 0$ such that for any $s \in (-\eta, \eta)$, 
$f \in \mathcal{S}$, $v \in \mathcal{S}_{\epsilon}$, $\varphi \in \scr B_{1}$ and $n \geq 1$,  
\begin{align*}
\sup_{l \in \bb R}
\bb E_{\bb Q_{s,f}^v} 
 \left[
 \varphi(G_n \!\cdot\! v) 
      \mathds 1_{ \left\{
 \log \frac{\langle f, G_n v \rangle}{\langle f, v \rangle} - n \Lambda'(s) \in [a_1, a_2] + l 
     \right\} }
 \right]
\leq  c \|\varphi\|_{\infty} \frac{ a_2 - a_1 }{\sqrt{n} }.
\end{align*}

The asymptotic \eqref{LLT-LD} is a consequence of Theorem \ref{Thm-Petrov-Uppertail} 
by following the proof of \cite[Theorem 2.5]{XGL19d}.

\subsection{Proof of Theorems \ref{Thm-Petrov-Norm-cocycle} and \ref{Thm_LDSpectPosi}}
Since $\|G_n v\| = \langle \mathbf 1, G_n v \rangle$ for any $v \in \bb R_+^d$ and $\langle \mathbf 1, u \rangle = 1$ for any $u \in \mathcal S$,  
Theorem \ref{Thm-Petrov-Norm-cocycle} follows from \eqref{Petrov-Lower} of Theorem \ref{Thm-Petrov-Uppertail} 
by taking $f=\mathbf 1$. 

Next we prove Theorem \ref{Thm_LDSpectPosi}. 
We only give a proof of \eqref{LD_Spec01} since the proof of \eqref{LD_Spec02} can be carried out in a similar way. 
Since $\rho(G_n) \leq \| G_n \|$, the upper bound in \eqref{LD_Spec01}
follows from \eqref{LD-BRP-MatrixNorm}: 
there exists a constant $C < \infty$ such that for any $s \in J^+$, with $q = \Lambda'(s)$,
\begin{align}\label{Pf_LD_Spec_000}
\limsup_{n\to \infty} 
\sqrt{n}  \,  e^{ n \Lambda^*(q) }  
\mathbb{P}  \big(\log \rho(G_n)  \geq nq \big)
<  C. 
\end{align}
For the lower bound in \eqref{LD_Spec01}, 
by the Collatz-Wielandt formula, we have for any $g \in \Gamma_{\mu}$, 
\begin{align}\label{Formu_Collatz_Wielandt}
\rho(g) = \sup_{ v \in \mathcal{S} } 
   \min_{ 1 \leq i \leq d, \,  \langle  e_i, v \rangle > 0 } 
     \frac{ \langle e_i, gv \rangle }{ \langle e_i, v \rangle}. 
\end{align}
Taking $x = e_1$, we get $\rho(G_n) \geq  G_n^{1,1}$. 
Using \eqref{Petrov-Upper} of Theorem \ref{Thm-Petrov-Uppertail} with $f = v = e_1$, 
we obtain that there exists a constant $c>0$ such that for any $s \in J^+$, with $q = \Lambda'(s)$,
\begin{align}\label{Pf_LD_Spec_001}
  \liminf_{n\to \infty} 
   \sqrt{n}  \,  e^{ n \Lambda^*(q) }  
\mathbb{P}  \big(\log \rho(G_n)  \geq nq \big) > c, 
\end{align}
which, together with \eqref{Pf_LD_Spec_000}, concludes the proof of \eqref{LD_Spec01}. 
The proof of Theorem \ref{Thm_LDSpectPosi} is complete.

\section*{Declarations}

\textbf{Data Availability}
Data sharing is not applicable to this article as no datasets were generated or analyzed
during the current study.

\textbf{Conflict of interest} The authors have no competing interests to declare that are relevant to the content of
this article.


\end{document}